\documentclass[final,1p,times,twocolumn]{elsarticle}

\usepackage{epsfig}
\usepackage{amssymb}
\usepackage{amsthm}

\usepackage{array,booktabs}
\usepackage[hyphens]{url}
\usepackage{psfrag,subfigure,hyperref}
\hypersetup{breaklinks=true}
\usepackage[english]{babel}
\usepackage{amsfonts, verbatim}
\usepackage{amsmath}
\newcommand{\reals}{{\mathbb{R}}}
\newcommand{\integers}{{\mathbb{Z}}}
\newcommand{\naturals}{{\mathbb{N}}}
\renewcommand{\d}{{\rm{d}}}
\newcommand{\iw}{\textrm{i}_{w}}
\newcommand{\modulo}[1]{{\left|#1\right|}}
\newcommand{\norma}[1]{{\left\|#1\right\|}}
\newcommand{\Cc}[1]{\mathbf{C_c^{#1}}}
\newcommand{\Lip}{\mathbf{Lip}}
\newcommand{\lip}{\mathrm{Lip}}

\newtheorem{thm}{Theorem}[section]

\newtheorem{pro}{Proposition}[section]

 \renewcommand{\L}[1]{\mathbf{L^{#1}}}

\newcommand{\sign}{\mathrm{sign}}

\usepackage{calc}
\journal{Preprint version}

\begin{document}

\begin{frontmatter}



\title{Qualitative behaviour and numerical approximation of solutions to conservation laws with non-local point constraints on the flux and modeling of crowd dynamics at the bottlenecks}


\author[Besancon]{Boris Andreianov}
\ead{boris.andreianov@univ-fcomte.fr}

\author[Besancon]{Carlotta Donadello}
\ead{carlotta.donadello@univ-fcomte.fr}

\author[Besancon]{Ulrich Razafison\corref{mycorrespondingauthor}}
\cortext[mycorrespondingauthor]{Corresponding author}
\ead{ulrich.razafison@univ-fcomte.fr}

\author[Warsaw]{Massimiliano D.~Rosini}
\ead{mrosini@icm.edu.pl}

\address[Besancon]{Laboratoire de Math\'ematiques,
Universit\'e de Franche-Comt\'e,\\
16 route de Gray,
25030 Besan\c{c}on Cedex,
France}

\address[Warsaw]{
ICM, University of Warsaw\\
ul.~Prosta 69,
P.O. Box 00-838,
Warsaw, Poland}

\begin{abstract}
In this paper we investigate numerically the model for pedestrian traffic proposed in [B.~Andreianov, C.~Donadello, M.D.~Rosini, Crowd dynamics and conservation laws with nonlocal constraints and capacity drop, Mathematical Models and Methods in Applied Sciences 24 (13) (2014) 2685-2722] . We prove the convergence of a scheme based on a constraint finite volume method and validate it with an explicit solution obtained in the above reference. We then perform \emph{ad hoc} simulations to qualitatively validate the model under consideration by proving its ability to reproduce typical phenomena at the bottlenecks, such as Faster Is Slower effect and the Braess' paradox. 
\end{abstract}


\end{frontmatter}



\section{Introduction}\label{sec:intro}

\noindent Andreianov, Donadello and Rosini developed in~\cite{BorisCarlottaMax-M3AS} a macroscopic model, called here ADR, aiming at describing the behaviour of pedestrians at bottlenecks. The model is given by the Cauchy problem for a scalar hyperbolic conservation law in one space dimension with non-local point constraint of the form
\begin{subequations}\label{eq:constrianed}
\begin{align}\label{eq:constrianed1}
    &\partial_t\rho + \partial_xf(\rho) = 0 && (t,x) \in \reals_+\times\reals,\\
    \label{eq:constrianed3}
    &\rho(0,x) =\bar\rho(x) && x \in \reals,\\
    \label{eq:constrianed2}
    &f\left(\rho(t,0\pm)\right) \le p\left( \int_{\reals_-} w(x)~ \rho(t,x)~{\d} x\right) && t \in \reals_+,
\end{align}
\end{subequations}
where $\rho(t,x) \in \left[0,R\right]$ is the (mean) density of pedestrians in $x \in \reals$ at time $t \in \reals_+$ and $\bar\rho  \colon \reals \to [0,R]$ is the initial (mean) density, with $R >0$ being the maximal density. 
Then, $f \colon [0,R] \to \reals_+$ is the flow considered to be bell-shaped, which is an assumption commonly used in crowd dynamics. A typical example of such flow is the so-called Lighthill-Whitham-Richards (LWR) flux~\cite{LWR1, LWR2, Greenshields_1934} defined by 
$$f(\rho)=\rho\,v_{\max}\left(1-\dfrac{\rho}{\rho_{\max}}\right),$$where $v_{\max}$ and $\rho_{\max}$ are the maximal velocity and the maximal density of pedestrians respectively. Throughout this paper the LWR flux will be used.
Next $p \colon \reals_+ \to \reals_+$ prescribes the maximal flow allowed through a bottleneck located at $x=0$ as a function of the weighted average density in a left neighbourhood of the bottleneck and $w  \colon \reals_- \to \reals_+$ is the weight function used to average the density. 

Finally in \eqref{eq:constrianed2}, $\rho(t,0-)$ denotes the left measure theoretic trace along the constraint, implicitly defined by
\begin{align*}
&\lim_{\varepsilon\downarrow0}\frac{1}{\varepsilon} \int_0^{+\infty}  \int_{-\varepsilon}^0 \modulo{\rho(t,x) - \rho(t,0-)}~\phi(t,x)~{\d} x~{\d} t=0&
&\text{for all }\phi \in \Cc\infty(\reals^2;\reals).
\end{align*}
The right measure theoretic trace, $\rho(t,0+)$, is defined analogously.\\

\noindent In the last few decades, the study of the pedestrian behaviour through bottlenecks, namely at locations with reduced capacity, such as doors, stairs or narrowings, drawn a considerable attention. The papers \cite{Schadschneider2008Evacuation, Cepolina2009532, Hoogendoorn01052005, Kopylow, Schreckenberg2006, Seyfried:59568, Zhang20132781} present results of empirical experiments. 
However, for safety reasons, experiments reproducing extremal conditions such as evacuation and stampede are not available.
In fact, the unique experimental study of a crowd disaster is proposed in~\cite{Helbing_disaster}.
The available data show that the capacity of the bottleneck (i.e.~the maximum number of pedestrians that can flow through the bottleneck in a given time interval) can drop when high-density conditions occur upstream of the bottleneck. 
This phenomenon is called \emph{capacity drop} and can lead to extremely serious consequences in escape situations.
In fact, the crowd pressure before an exit can reach very high values, the efficiency of the exit dramatically reduces and accidents become more probable due to the overcrowding and the increase of the evacuation time (i.e.~the temporal gap between the times in which the first and the last pedestrian pass through the bottleneck). 
A linked phenomenon is the so-called \emph{Faster Is Slower} (FIS) effect, first described in~\cite{Helbing2000Simulating}. 
FIS effect refers to the jamming and clogging at the bottlenecks, that result in an increase of the evacuation time when the degree of hurry of a crowd is high. 
We recall that the capacity drop and the FIS effect are both experimentally reproduced in~\cite{Cepolina2009532,Soria20121584}.
A further related (partly counter-intuitive) phenomenon is the so-called \emph{Braess' paradox} for pedestrian flows~\cite{hughes2003flow}.
It is well known that placing a small obstacle before an exit door can mitigate the inter-pedestrian pressure and, under particular circumstances, it reduces the evacuation time by improving the outflow of people.

Note that as it happens for any first order model, see for instance~\cite[Part~III]{Rosinibook} and the references therein, ADR can not explain the capacity drop and collective behaviours at the bottlenecks. Therefore one of the difficulties we have to face is that the constraint $p$ has to be deduced together with the fundamental diagram from the empirical observations. 

The aim of this paper is to validate ADR by performing simulations in order to show the ability of the model to reproduce the main effects described above and related to capacity drop that are FIS and Braess' paradox. To this end we propose a numerical scheme for the model and prove its convergence. The scheme is obtained by adapting the local constrained finite volume method introduced in~\cite{scontrainte} to the non-local case considered in ADR, using a splitting strategy. 

The paper is organized as follows. In Section~\ref{sec:model} we briefly recall the main theoretical results for ADR. In Section~\ref{sec:NumericalMethod} we introduce the numerical scheme, prove its convergence and validate it with an explicit solution obtained in \cite{BorisCarlottaMax-M3AS}. In Section~\ref{sec:simulations} we perform simulations to show that ADR is able to reproduce the Braess' paradox and the FIS effect. In Subsection~\ref{sec:Braess+FIS} we combine local and non-local constraints to model a slow zone placed before the exit. Conclusions and perspectives are outlined in Section~\ref{sec:conclusions}.

\section{Well-posedness for the ADR model}\label{sec:model}

Existence, uniqueness and stability for the general Cauchy problem~\eqref{eq:constrianed} are established in~\cite{BorisCarlottaMax-M3AS} under the following assumptions:
\begin{enumerate}[label=\bfseries{(W)}]
  \item[\textbf{(F)}]  $f$ belongs to $\Lip\left( [0,R]; \left[0, +\infty\right[ \right)$ and is supposed to be bell-shaped, that is $f(0) = 0 = f(R)$ and there exists $\sigma \in \left]0,R\right[$ such that $f'(\rho)~(\sigma-\rho)>0$ for a.e.~$\rho \in [0,R]$.
  \item[\textbf{(W)}] $w$ belongs to $\L\infty(\reals_-;\reals_+)$, is an increasing map, $\norma{w}_{\L1(\reals_-)} = 1$ and there exists $\iw >0$ such that $w(x) = 0$ for any $x \le -\iw$.
  \item[\textbf{(P)}] $p$ belongs to $\Lip \left( \left[0,R\right]; \left]0,f(\sigma)\right] \right)$ and is a non-increasing map.
\end{enumerate}

The regularity $w\in \L\infty(\reals_-;\reals_+)$ is the minimal requirement needed in order to prove existence and uniqueness of~\eqref{eq:constrianed}. In this paper, we shall consider continuous $w$.  

The existence of solutions for the Riemann problem for~\eqref{eq:constrianed} is proved in~\cite{AndreianovDonadelloRosiniRazafisonProc} for piecewise constant $p$. 
However, such hypothesis on $p$ is not sufficient to ensure uniqueness of solutions, unless the flux $f$ and the efficiency $p$ satisfy a simple geometric condition, see~\cite{AndreianovDonadelloRosiniRazafisonProc} for details.
In the present paper, we consider either continuous nonlinear $p$ or a piecewise constant $p$ that satisfies such geometric condition.

The definition of entropy solution for a Cauchy problem~\eqref{eq:constrianed1}, \eqref{eq:constrianed3} with a fixed \emph{a priori} time dependent constraint condition
\begin{align}\label{eq:constrianed2bis}
    &f\left(\rho(t,0\pm)\right) \le q(t) && t \in \reals_+
\end{align}
was introduced in~\cite[Definition~3.2]{ColomboGoatinConstraint} and then reformulated in~\cite[Definition~2.1]{scontrainte}, see also~\cite[Proposition~2.6]{scontrainte} and \cite[Definition 2.2]{chalonsgoatinseguin}.
Such definitions are obtained by adding a term that accounts for the constraint in the classical definition of entropy solution given by Kruzkov in~\cite[Definition~1]{Kruzkov}.
The definition of entropy solution given in~\cite[Definition~2.1]{BorisCarlottaMax-M3AS} is obtained by extending these definitions to the framework of non-local constraints.

The following theorem on existence, uniqueness and stability of entropy solutions of the constrained Cauchy problem~\eqref{eq:constrianed} is achieved under the hypotheses~\textbf{(F)}, \textbf{(W)} and~\textbf{(P)}.
\begin{thm}[Theorem~3.1 in~\cite{BorisCarlottaMax-M3AS}]\label{thm:1}
    Let~\textbf{(F)}, \textbf{(W)}, \textbf{(P)} hold. Then, for any initial datum $\bar\rho \in \L\infty(\reals;[0,R])$, the Cauchy problem~\eqref{eq:constrianed} admits a unique entropy solution $\rho$. Moreover, if $\rho' = \rho'(t,x)$ is the entropy solution corresponding to the initial datum $\bar\rho' \in \L\infty(\reals;[0,R])$, then for all~$T>0$ and $L>{\rm{i}_w}$, the following inequality holds
\begin{equation}\label{eq:lipdepen-localized}
\norma{\rho(T) - \rho'(T)}_{\L1([-L,L])} \le e^{CT}\norma{\bar\rho - \bar\rho'}_{\L1(\{\modulo{x}\le L+MT\})},
\end{equation}
where $M=\lip(f)$ and $C=2 \lip(p) \norma{w}_{\L\infty(\reals_-)}$.
\end{thm}

The total variation of the solution may in general increase due to the presence of the constraint.
In~\cite{BorisCarlottaMax-M3AS} the authors provide an invariant domain $\mathcal{D} \subset \L1 \left(\reals ; [0,R]\right)$ such that if $\bar\rho$ belongs to $\mathcal{D}$, then one obtains a Lipschitz estimate with respect to time of the $\L1$ norm and an a priori estimate of the total variation of
\[
\Psi(\rho) = \sign(\rho-\sigma) [f(\sigma) - f(\rho)] = \int_\sigma^\rho \modulo{\dot{f}(r)} \, \d r.
\]

\section{Numerical method for approximation of ADR}\label{sec:NumericalMethod}
In this section we describe the numerical scheme  based on finite volume method that we use to solve~\eqref{eq:constrianed}. 
Then we prove the convergence of our scheme and validate it by comparison with an explicit solution of~\eqref{eq:constrianed}.
In what follows, we assume that \textbf{(F)}, \textbf{(W)} and~\textbf{(P)} hold.

\subsection{Non-local constrained finite volume method}

Let $\Delta x$ and $\Delta t$ be the constant space and time steps respectively.
We define the points $x_{j+1/2}=j\Delta x$, the cells $K_j=[x_{j-1/2},x_{j+1/2}[$ and the cell centers $x_j=(j-1/2)\Delta x$ for $j\in\integers$. We define the time discretization $t^n=n\Delta t$.
We introduce the index $j_c$ such that $x_{j_c+1/2}$ is the location of the constraint (a door or an obstacle). For $n\in\naturals$ and $j\in\integers$, we denote by $\rho_j^n$ the approximation of the average
of $\rho(t^n,\cdot~)$ on the cell $K_j$, namely
\begin{align*}
&\rho_j^0=\frac{1}{\Delta x}\displaystyle\int_{x_{j-1/2}}^{x_{j+1/2}}\overline{\rho}(x)\,\d x&
&\text{ and }&
&\rho_j^n\simeq\frac{1}{\Delta x}\displaystyle\int_{x_{j-1/2}}^{x_{j+1/2}}\rho(t^n,x)\,\d x
&\text{ if }n>0.
\end{align*}

We recall that for the classical conservation law~\eqref{eq:constrianed1}-\eqref{eq:constrianed3},
a standard finite volume method can be written into the form 
\begin{equation}
\label{schema}
\rho_j^{n+1}=\rho_j^n-\frac{\Delta t}{\Delta x}\left(\mathcal{F}_{j+1/2}^n-\mathcal{F}_{j-1/2}^n\right),
\end{equation} 
where $\mathcal{F}_{j+1/2}^n=F\left(\rho_j^n,\rho_{j+1}^n\right)$ is a monotone, consistent numerical flux, that is, $F$ satisfies the following assumptions:
\begin{itemize}
\item $F$ is Lipschitz continuous from $[0,R]^2$ to $\reals$ with Lipschitz constant $\lip(F)$,
\item $F(a,a)=f(a)$ for any $a\in[0,R]$,
\item $(a,b) \in [0,R]^2 \mapsto F(a,b) \in \reals$ is non-decreasing with respect to $a$ and non-increasing with respect to $b$.
\end{itemize} 
We also recall that in~\cite{scontrainte} the numerical flux for the time dependent constraint~\eqref{eq:constrianed2bis} is modified as follow in order to take into account the constraint condition
\begin{align}
\label{def.flux.num.contrainte}
    &\mathcal{F}_{j+1/2}^n=\left\{\begin{array}{l@{\quad\text{if }}l}
    F\left (\rho_j^n,\rho_{j+1}^n\right )&j\ne j_c,\\[6pt]
    \min\left \{F\left(\rho_j^n,\rho_{j+1}^n\right),q^n\right \}&j=j_c,
       \end{array}\right.
\end{align}
where $q^n$ is an approximation of $q(t^n)$.
In the present paper, when dealing with a Cauchy problem subject to a non-local constraint of the form~\eqref{eq:constrianed2} we will use the approximation
\begin{equation}
\label{contrainte.approx}
q^n=p\left(\Delta x\sum_{j\le j_c}w(x_j)\,\rho_j^n\right).
\end{equation}
Roughly speaking
\begin{itemize}
\item we apply the numerical scheme~\eqref{schema} for the problem~\eqref{eq:constrianed1}-\eqref{eq:constrianed3},
\item we apply the numerical scheme~\eqref{schema}-\eqref{def.flux.num.contrainte} for the problem~\eqref{eq:constrianed1}-\eqref{eq:constrianed3}-\eqref{eq:constrianed2bis},
\item we apply the numerical scheme~\eqref{schema}-\eqref{def.flux.num.contrainte}-\eqref{contrainte.approx} for the problem~\eqref{eq:constrianed}.
\end{itemize}

\subsection{Convergence of the scheme}
Let us introduce the finite volume approximate solution $\rho_{\Delta}$ defined by
\begin{align}\label{def.rho.delta}
&\rho_\Delta(t,x)=\rho_j^n&
&\mbox{for }x\in K_j\mbox{ and }t\in [t^n,t^{n+1} [,
\end{align}
where the sequence $(\rho_{j}^n)_{j\in\integers,\,n\in\naturals}$ is obtained by the numerical scheme~\eqref{schema}-\eqref{def.flux.num.contrainte}. 
Analogously, we also define the approximate constraint function
\begin{align}
\label{def.q.delta}
&q_\Delta(t)=q^n&
&\mbox{for }t\in [t^n,t^{n+1} [.
\end{align}

First, we prove a discrete stability estimate valid for any domain $Q = [0,T] \times \reals$ with $T>0$, 
for the scheme~\eqref{schema}-\eqref{def.flux.num.contrainte} applied to problem~\eqref{eq:constrianed1}-\eqref{eq:constrianed3}-\eqref{eq:constrianed2bis}.
This estimate can be seen as the equivalent, in this framework, of the stability result established in~\cite[Proposition~2.10]{scontrainte}.

\begin{pro}
\label{pro.diff.solutions}
Let $\overline{\rho}$ be in $\L\infty(\reals;[0,R])$ and $q_\Delta$, $\hat{q}_\Delta$ be piecewise constant functions of the form~\eqref{def.q.delta}.
If $\rho_\Delta$ and $\hat{\rho}_\Delta$ are the approximate solutions of~\eqref{eq:constrianed1}-\eqref{eq:constrianed3}-\eqref{eq:constrianed2bis} corresponding, respectively, to $q_\Delta$ and $\hat{q}_\Delta$ and constructed by applying the scheme~\eqref{schema}-\eqref{def.flux.num.contrainte}, then we have 
$$\norma{\rho_\Delta-\hat{\rho}_\Delta}_{\L1(Q)}\le2 T \norma{q_\Delta-\hat{q}_\Delta}_{\L1([0,T])}.$$ 
\end{pro}

\begin{proof}
For notational simplicity, let $N = \lfloor T/\Delta t\rfloor$. Let us also introduce $(\tilde{\rho}_j^n)_{j\in\integers,\,n\in\naturals}$ defined by,  
$$\tilde{\rho}_j^{n+1}=\rho_j^n-\frac{\Delta t}{\Delta x}\left(\mathcal{\tilde{F}}_{j+1/2}^n-\mathcal{\tilde{F}}_{j-1/2}^n\right),\quad\mbox{for any }
j\in\integers,\,n\in\naturals,$$
where $\mathcal{\tilde{F}}_{j+1/2}^n$ is defined by 
\begin{align*}
    &\mathcal{\tilde{F}}_{j+1/2}^n=\left\{\begin{array}{l@{\quad\text{if }}l}
    F\left (\rho_j^n,\rho_{j+1}^n\right )&j\ne j_c,\\[6pt]
    \min\left \{F\left(\rho_j^n,\rho_{j+1}^n\right),\hat{q}^n\right \}&j=j_c.
       \end{array}\right.
\end{align*}
 Then using the definitions of $(\rho_j^n)_{j\in\integers,\,n\in\naturals}$ and $(\tilde{\rho}_j^n)_{j\in\integers,\,n\in\naturals}$, we have for any $n=1,\dots,N$, $$\rho_j^n=\tilde{\rho}_j^n\quad\mbox{if}\quad j\notin\{j_c,j_c+1\}$$
and  \begin{align*}
&\rho_{j_c}^n-\tilde{\rho}_{j_c}^n=
-\frac{\Delta t}{\Delta x}\left(\min\left \{F\left (\rho_{j_c}^{n-1},\rho_{j_c+1}^{n-1}\right ),q^{n-1}\right\} +\min\left \{F\left (\rho_{j_c}^{n-1},\rho_{j_c+1}^{n-1}\right ),\hat{q}^{n-1}\right \}\right),&\\
&\rho_{j_c+1}^n-\tilde{\rho}_{j_c+1}^n=
\frac{\Delta t}{\Delta x}\left(\min\left \{F\left(\rho_{j_c}^{n-1},\rho_{j_c+1}^{n-1}\right),q^{n-1}\right \} - \min\left \{F\left(\rho_{j_c}^{n-1},\rho_{j_c+1}^{n-1}\right),\hat{q}^{n-1}\right \} \right),
\end{align*}
which implies that 
\begin{align*}
&\modulo{\rho_{j_c}^n-\tilde{\rho}_{j_c}^n}\le\frac{\Delta t}{\Delta x} ~ \modulo{q^{n-1}-\hat{q}^{n-1}},&
&\modulo{\rho_{j_c+1}^n-\tilde{\rho}_{j_c+1}^n}\le\frac{\Delta t}{\Delta x} ~ \modulo{q^{n-1}-\hat{q}^{n-1}}.
\end{align*}
Therefore we deduce that, for any $n=1,\dots,N$, 
\begin{equation}
\label{pro.diff.solutions.eq1}
\sum_{j\in\integers}\modulo{\rho_{j}^n-\tilde{\rho}_{j}^n}\le2\frac{\Delta t}{\Delta x} ~ \modulo{q^{n-1}-\hat{q}^{n-1}}.
\end{equation}
Besides, observe that the modification of the numerical flux at the interface $x_{j_c+1/2}$ introduced in~\eqref{def.flux.num.contrainte} does not affect the monotonicity of the scheme~\eqref{schema}-\eqref{def.flux.num.contrainte} (see~\cite[Proposition~4.2]{scontrainte}). 
Therefore, for any $n=1,\dots,N$, we have 
\begin{equation}
\label{pro.diff.solutions.eq2}
\sum_{j\in\integers}\modulo{\tilde{\rho}_j^n-\hat{\rho}_j^n}\le
\sum_{j\in\integers}\modulo{\rho_j^{n-1}-\hat{\rho}_j^{n-1}}.
\end{equation} 
Hence thanks to~\eqref{pro.diff.solutions.eq1} and~\eqref{pro.diff.solutions.eq2}, we can write
\begin{align*}
\sum_{j\in\integers}|\rho_j^1-\hat{\rho}_j^1|&\le\sum_{j\in\integers}|\rho_j^1-\tilde{\rho}_j^1|+\sum_{j\in\integers}|\tilde{\rho}_j^1-\hat{\rho}_j^1|\le2\frac{\Delta t}{\Delta x} ~ \modulo{q^0-\hat{q}^0}+\sum_{j\in\integers}|\rho_j^0-\hat{\rho}_j^0|
=2\frac{\Delta t}{\Delta x} ~ \modulo{q^0-\hat{q}^0}.
\end{align*}
Then an induction argument shows that for any $n=1,\dots,N$, 
$$\sum_{j\in\integers}\modulo{\rho_j^n-\hat{\rho}_j^n}\le2\frac{\Delta t}{\Delta x}\sum_{k=0}^{n-1}|q^k-\hat{q}^k|\le\frac{2}{\Delta x}~\|q_\Delta-\hat{q}_\Delta\|_{L^1([0,t^n])}.$$
In conclusion, we find that 
\begin{align*}
\|\rho_\Delta-\hat{\rho}_\Delta\|_{L^1(Q)}&=\Delta t\,\Delta x\sum_{n=1}^N\sum_{j\in\integers}|\rho_j^n-\hat{\rho}_j^n|
\le 2 \|q_\Delta-\hat{q}_\Delta\|_{L^1([0,T])}~\sum_{n=1}^N \Delta t
\le 2 T \|q_\Delta-\hat{q}_\Delta\|_{L^1([0,T])}
\end{align*}
and this ends the proof.
\end{proof}

\noindent Let us now notice that as in~\cite[Proposition~4.2]{scontrainte}, under the CFL condition 
\begin{equation}
\label{CFL}
\lip(F) ~ \frac{\Delta t}{\Delta x}\le\frac{1}{2},
\end{equation}
we have the $\L\infty$ stability of the scheme~\eqref{schema}-\eqref{def.flux.num.contrainte}-\eqref{contrainte.approx} that is 
\begin{align}
\label{stability.Linfty}
&0\le\rho_\Delta(t,x)\le R&&\mbox{for a.e.~}(t,x)\in Q. 
\end{align}
\noindent This stability result allows to prove the statement below. 

\begin{pro}
\label{prop.estimation.bv.q}
Let $q_\Delta$ be defined by~\eqref{contrainte.approx}-\eqref{def.q.delta}. Then under the CFL condition~\eqref{CFL}, for any $T>0$, there exists $C>0$ only depending on $T$, $f$, $F$, $p$, $w$ and $R$ such that:
\begin{equation}
\label{estimation.bv.q}
\modulo{q_\Delta}_{BV([0,T])}\le C.
\end{equation}
\end{pro}

\begin{proof}
Let $N=\lfloor{T/\Delta t}\rfloor$ and $j_w$ be an integer such that supp$(w)\subset\underset{j_w\le j\le j_c}\cup K_j$.
Then for any $n=0,\dots,N-1$, we have 
\begin{align*}
\modulo{q^{n+1}-q^n}& = \modulo{p\left(\Delta x\sum_{j_w\le j\le j_c}w(x_j)\rho_j^{n+1}\right)-p\left(\Delta x\sum_{j_w\le j\le j_c}w(x_j)\rho_j^n\right)}\\
&\le\Delta x\,\lip(p)\left|\sum_{j_{w}\le j\le j_c}w(x_j)(\rho_j^{n+1}-\rho_j^n)\right|
=\Delta t\,\lip(p)\left|\sum_{j_{w}\le j\le j_c}w(x_j)\left(\mathcal{F}_{j+1/2}^n-\mathcal{F}_{j-1/2}^n\right)\right|.
\end{align*}
Now, using a summation by part, we have
\begin{align*}
\sum_{j_{w}\le j\le j_c}w(x_j)\left(\mathcal{F}_{j+1/2}^n-\mathcal{F}_{j-1/2}^n\right)
=w(x_{j_c})\mathcal{F}_{j_c+1/2}^n-w(x_{j_w})\mathcal{F}_{j_w-1/2} - \!\!\!\sum_{j_w\le j\le j_c-1}\left(w(x_{j+1})-w(x_j)\right)\mathcal{F}_{j+1/2}^n.
\end{align*}
Then, it follows that 
$$|q^{n+1}-q^n|\le\Delta t\,\lip(p)\,\|w\|_{\L\infty(\reals_-;\reals)}\sum_{j_w-1\le j\le j_c}|\mathcal{F}_{j+1/2}^n|.$$
Now, from~\eqref{def.flux.num.contrainte}, for any $j\in\integers$ we have the estimate
\begin{align*}
\modulo{\mathcal{F}_{j+1/2}^n} & \le \modulo{F(\rho_{j}^n,\rho_{j+1}^n)}
\le \modulo{F(\rho_{j}^n,\rho_{j+1}^n)-F(\rho_{j}^n,\rho_{j}^n)} + \modulo{f(\rho_{j}^n)}
\le \lip(F) \, \modulo{\rho_{j+1}^n-\rho_{j}^n}+\lip(f)\, \modulo{\rho_{j}^n}
\le R\left (\lip(F)+\lip(f)\right ).
\end{align*}
Hence we deduce that 
$$\modulo{q_\Delta}_{BV([0,T])}=\sum_{n=0}^{N-1}\modulo{q^{n+1}-q^n}\le C,$$
where $C=(j_c-j_w+2)\,T\,R\,\lip(p)\,\norma{w}_{\L\infty(\reals_-;\reals)}\,\left (\lip(F)+\lip(f)\right )$.
\end{proof}

We are now in a position to prove a convergence result for the scheme ~\eqref{schema}-\eqref{def.flux.num.contrainte}-\eqref{contrainte.approx}.

\begin{thm}
\label{th.convergence}
Under the CFL condition~\eqref{CFL}, the constrainted finite volume scheme~\eqref{schema}-\eqref{def.flux.num.contrainte}-\eqref{contrainte.approx} converges in $\L1(Q)$ to the unique entropy solution to~\eqref{eq:constrianed}.
\end{thm}

\begin{proof}
Let $(\rho_\Delta,q_\Delta)$ be constructed by the scheme~\eqref{schema}-\eqref{def.flux.num.contrainte}-\eqref{contrainte.approx}. 
Proposition~\ref{prop.estimation.bv.q} and Helly's lemma give the existence of a subsequence, still denoted $q_\Delta$ and a constraint function $q\in \L\infty([0,T])$ such that $q_\Delta$ converges to $q$ strongly in $\L1([0,T])$ as $\Delta t\to0$. Let $\rho\in \L\infty(\reals_+\times\reals;[0,R])$ be the unique entropy solution to~\eqref{eq:constrianed1}-\eqref{eq:constrianed3}-\eqref{eq:constrianed2bis} associated to $q$. 
It remains to prove that the subsequence $\rho_\Delta$ converges to $\rho$ strongly in $\L1(Q)$ as $\Delta t,\,\Delta x\to0$. The uniqueness of the entropy solution to~\eqref{eq:constrianed1}-\eqref{eq:constrianed3}-\eqref{eq:constrianed2bis} will then imply that the full sequence $\rho_\Delta$ converges to $\rho$ and, as a consequence, the full sequence $q_\Delta$ converges to $q=p\left (\int_{\reals_-}w(x)\,\rho(t,x)\,\d x\right )$.\\ 
Let $\hat{q}_\Delta$ be a piecewise constant approximation of $q$ such that $\hat{q}_\Delta$ converges to $q$ strongly in $\L1([0,T])$. Furthermore, we also introduce $\hat{\rho}_\Delta$ constructed by the scheme~\eqref{schema}-\eqref{def.flux.num.contrainte} and associated to $\hat{q}_\Delta$. Now we have
$$\norma{\rho-\rho_\Delta}_{\L1(Q)}\le\norma{\rho-\hat{\rho}_\Delta}_{\L1(Q)}+\norma{\rho_\Delta-\hat{\rho}_\Delta}_{\L1(Q)}. $$ 
But, thanks to~\cite[Theorem~4.9]{scontrainte}, under the CFL condition~\eqref{CFL}, $\norma{\rho-\hat{\rho}_\Delta}_{\L1(Q)}$ tends to $0$ as $\Delta t$, $\Delta x\to0$. Furthermore, thanks to Proposition~\ref{pro.diff.solutions}, we have 
$$\norma{\rho_\Delta-\hat{\rho}_\Delta}_{\L1(Q)}\le2 ~ T~ \norma{q_\Delta-\hat{q}_\Delta}_{\L1([0,T])}$$
which also shows that $\norma{\rho_\Delta-\hat{\rho}_\Delta}_{\L1(Q)}$ tends to $0$ as $\Delta t$, $\Delta x\to0$. 
\end{proof}

\subsection{Validation of the numerical scheme}\label{sec:validation}

\begin{figure}[htpb]\tiny
\centering
        {\includegraphics[width=0.25\textwidth]{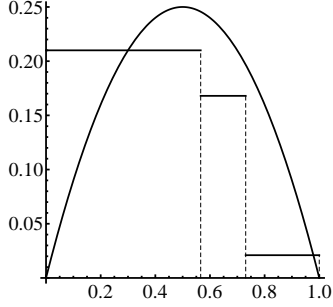}}
\caption{The functions ${[\rho \mapsto f(\rho)]}$ and ${[\xi \mapsto p(\xi)]}$ as in Section~\ref{sec:validation}.}
\label{fig:validation1}
\end{figure}


\begin{figure}
\centering
\begin{subfigure}[The solution in the $(t,x,\rho)$-coordinates.]
    {\includegraphics[width=0.49\textwidth]{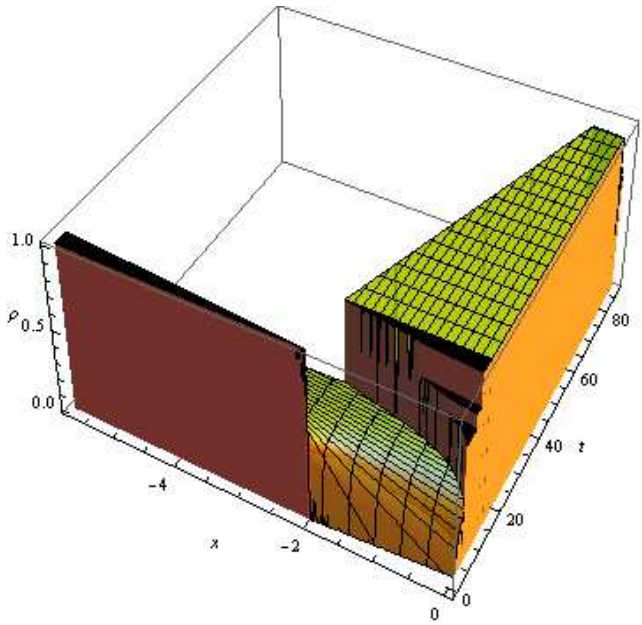}
    \label{fig:gull1}}
\end{subfigure}\hfill
\begin{subfigure}[The solution in the $(x,t)$-coordinates.]
    {\includegraphics[width=0.49\textwidth]{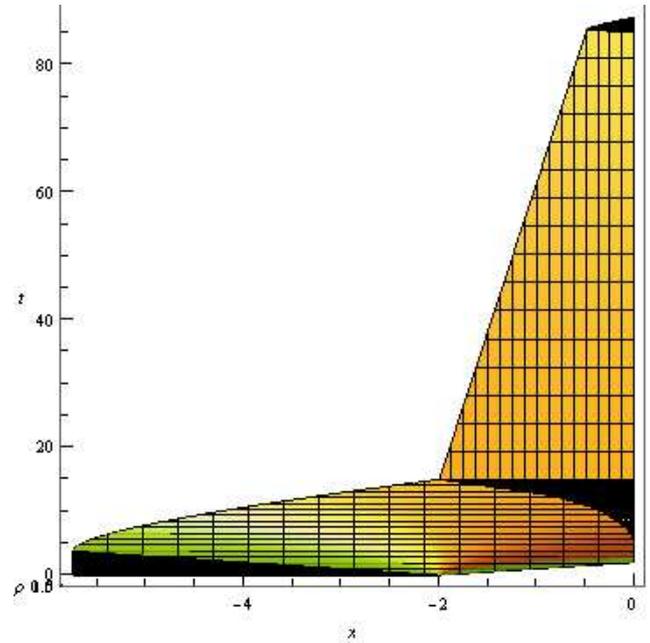}
    \label{fig:gull2}}
\end{subfigure}\\
\begin{subfigure}[The solution in the $(t,x,\rho)$-coordinates for $0 ~\le ~t \le~15$.]
    {\includegraphics[width=0.49\textwidth]{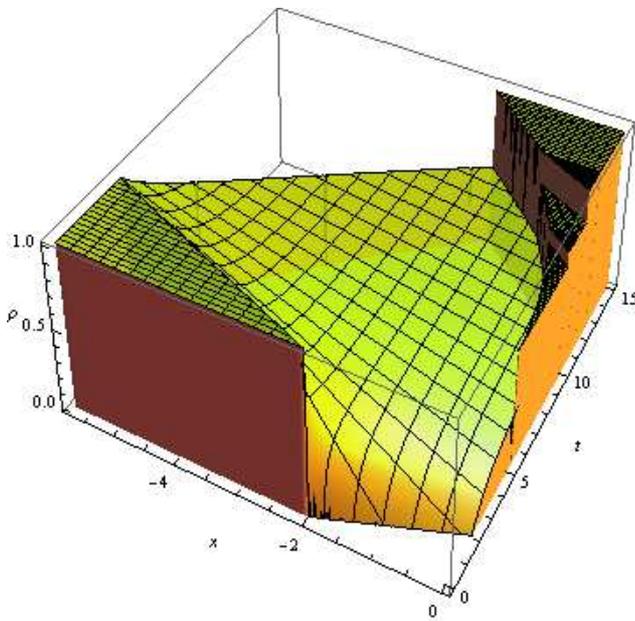}
    \label{fig:gull3}}
\end{subfigure}\hfill
\begin{subfigure}[The solution in the $(t,x,\rho)$-coordinates for $85 ~\le ~t ~\le ~87.5$.]
    {\includegraphics[width=0.49\textwidth]{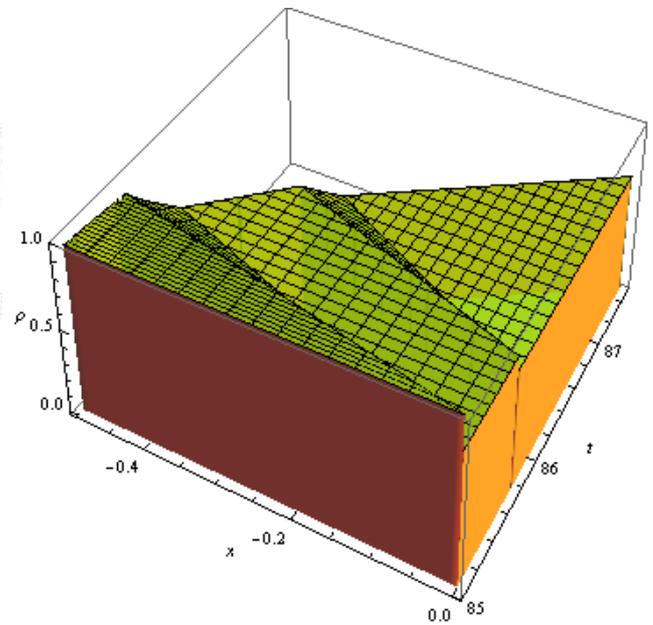}
    \label{fig:gull4}}
\end{subfigure}
\caption{Representation of the solution constructed in~\cite[Section~6]{BorisCarlottaMax-M3AS} and described in Subsection \ref{sec:validation}.}\label{fig:validation2}
\end{figure}

We propose here to validate the numerical scheme~\eqref{schema}-\eqref{def.flux.num.contrainte}-\eqref{contrainte.approx} using the Godounov numerical flux (see e.g.~\cite{GodlewskiRaviartBook, LevequeBook}) which will be used in the remaining of this paper:
\begin{eqnarray*}
    F(a,b) &= \left\{
    \begin{array}{l@{\quad\text{ if }}l}
       \underset{[a,b]}\min f & a\le b,\\
        \underset{[b,a]}\max f & a>b.
    \end{array}
    \right.
\end{eqnarray*}

\noindent We consider the explicit solution to~\eqref{eq:constrianed} constructed in~\cite[Section~6]{BorisCarlottaMax-M3AS} by applying the wave front tracking algorithm. The set up for the simulation is as follows. Consider the domain of computation $[-6,1]$, take  a normalized flux $f(\rho) = \rho(1-\rho)$ (namely the maximal velocity and the maximal density are assumed to be equal to one) and a linear weight function $w(x) = 2 (1 + x)\,\chi_{[-1, 0]}(x)$. Assume a uniform distribution of maximal density in $[x_A, x_B]$ at time $t=0$, namely $\bar\rho = \chi_{[x_A, x_B]}$. The efficiency of the exit, $p$, see Figure~\ref{fig:validation1}, is of the form
\begin{eqnarray*}
    p(\xi) &= \left\{
    \begin{array}{l@{\quad\text{ if }}l}
       p_0 & 0\le\xi<\xi_1,\\
       p_1 & \xi_1\le\xi<\xi_2,\\
       p_2 & \xi_2\le\xi\le 1.
    \end{array}
    \right.
\end{eqnarray*}
The explicit solution $\rho$ corresponding to the values
\begin{align*}
    &p_0 =0.21, && p_1 =0.168, && p_2 =0.021, && \xi_1 \sim 0.566,
    &x_A = -5.75, && x_B =-2, && \xi_2 \sim 0.731,
\end{align*}
is represented in Figure~\ref{fig:validation2}. The above choices for the flux $f$ and the efficiency $p$ ensure that the solution to each Riemann problem is unique, see~\cite{AndreianovDonadelloRosiniRazafisonProc}. We defer to~\cite[Section~6]{BorisCarlottaMax-M3AS} for the details of the construction of the solution $\rho$ and its physical interpretation.
\begin{figure}
\begin{subfigure}[$\rho_\Delta(0,x)$]
    {\begin{psfrags}
    \psfrag{rho}[c,b]{}
    \psfrag{x}[c,t]{\tiny$x$}
        \includegraphics[width=0.22\textwidth]{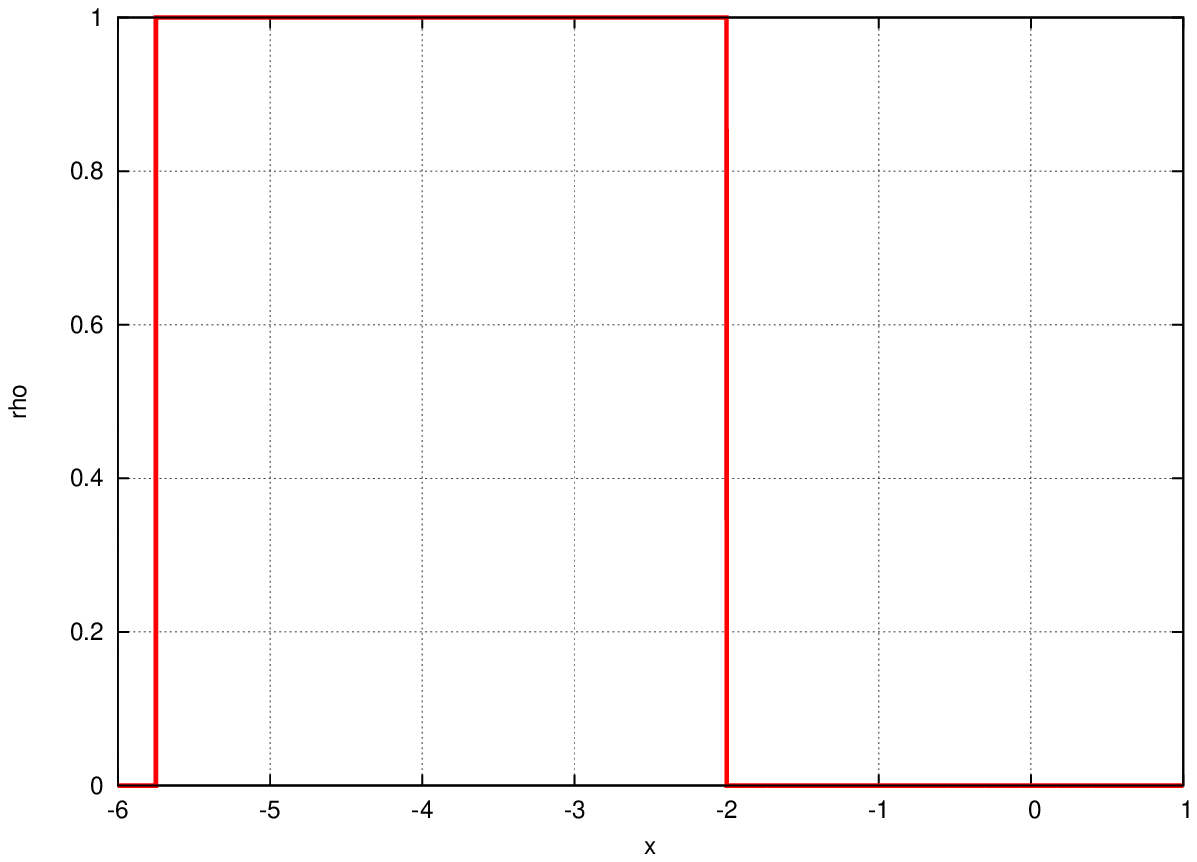}
    \end{psfrags}\label{fig:tiger1}}
\end{subfigure}~
\begin{subfigure}[$\rho_\Delta(1,x)$]
    {\begin{psfrags}
    \psfrag{rho}[c,b]{}
    \psfrag{x}[c,t]{\tiny$x$}
        \includegraphics[width=0.22\textwidth]{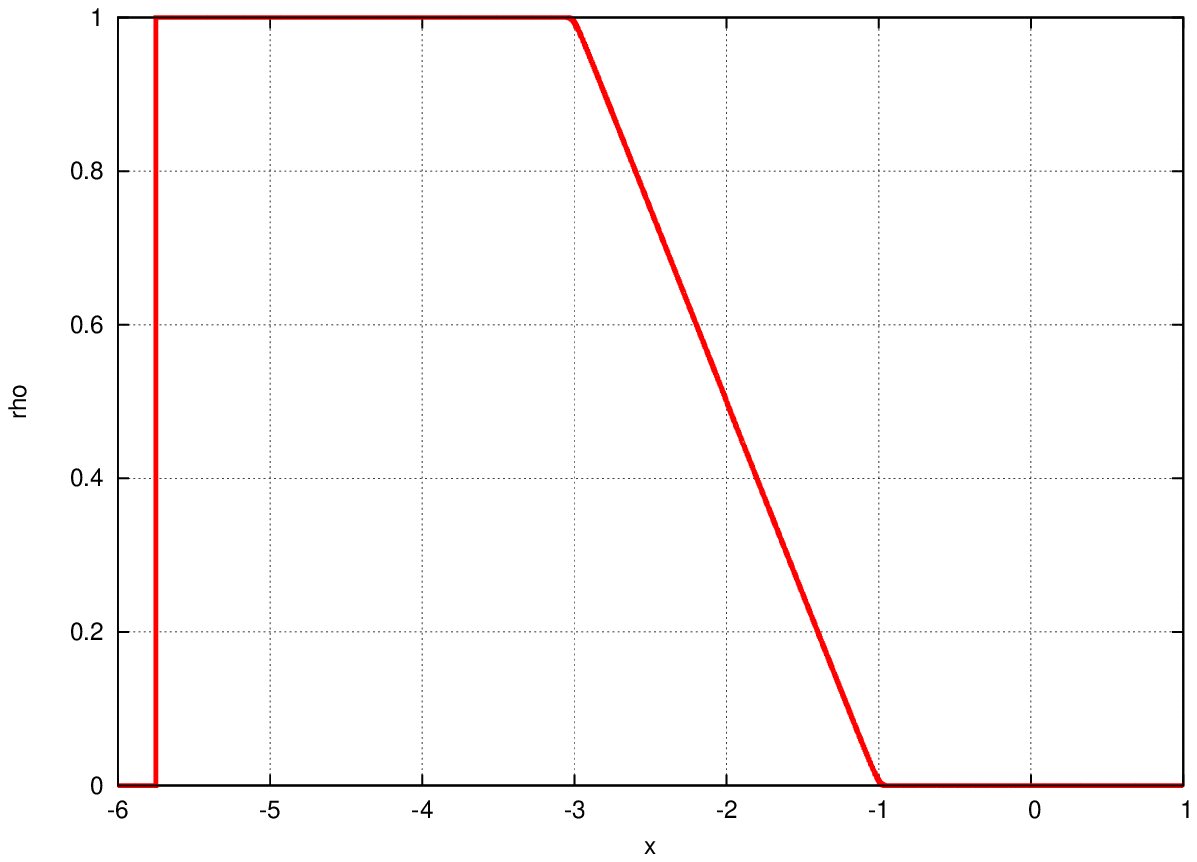}
    \end{psfrags}\label{fig:tiger2}}
\end{subfigure}~
\begin{subfigure}[$\rho_\Delta(7.325,x)$]
    {\begin{psfrags}
    \psfrag{rho}[c,b]{}
    \psfrag{x}[c,t]{\tiny$x$}
        \includegraphics[width=0.22\textwidth]{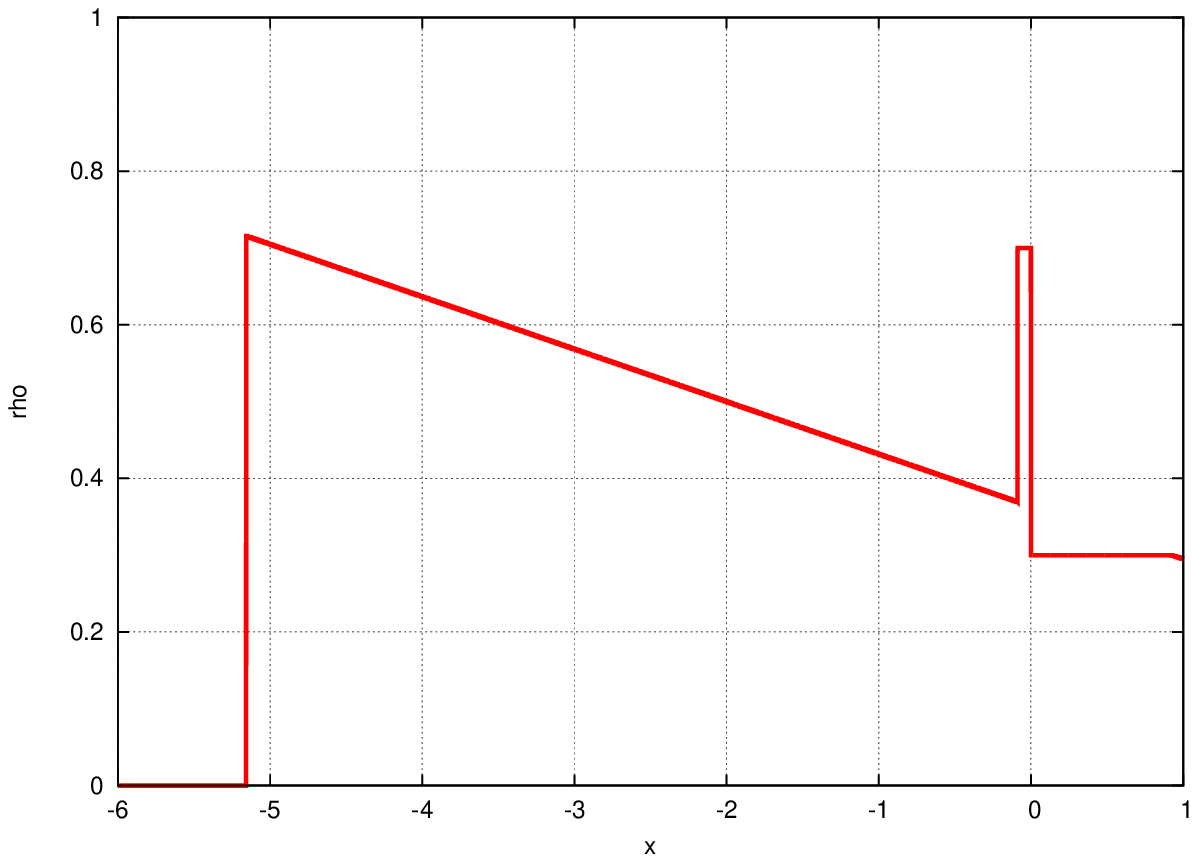}
    \end{psfrags}\label{fig:tiger3}}
\end{subfigure}~
\begin{subfigure}[$\rho_\Delta(10,x)$]
    {\begin{psfrags}
    \psfrag{rho}[c,b]{}
    \psfrag{x}[c,t]{\tiny$x$}
        \includegraphics[width=0.22\textwidth]{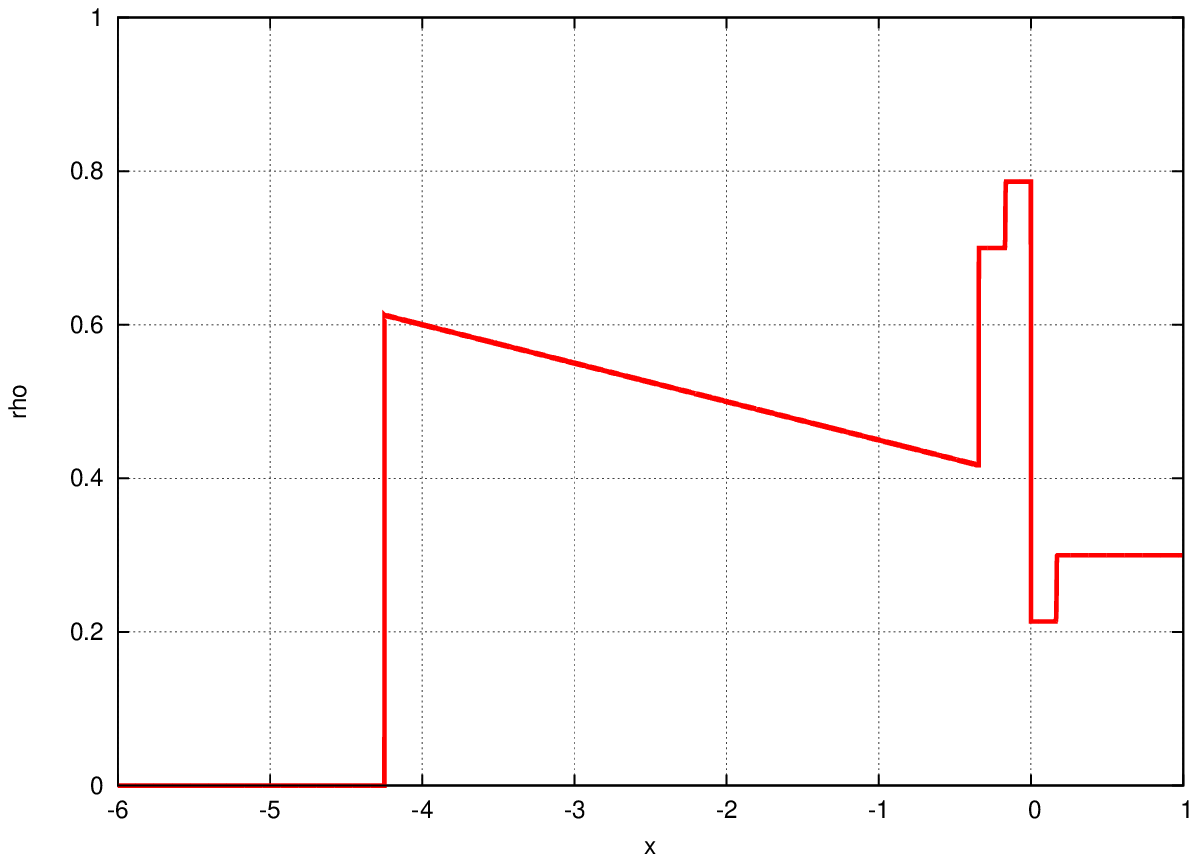}
    \end{psfrags}\label{fig:tiger4}}
\end{subfigure}\\
\begin{subfigure}[$\rho(0,x)$]
    {\begin{psfrags}
    \psfrag{r}[c,c]{}
    \psfrag{x}[c,t]{\tiny$\vphantom{\int}x$}
    \psfrag{z}[c,t]{\tiny$\vphantom{\int}1$}
    \psfrag{0}[c,t]{\tiny$\vphantom{\int}0$}
    \psfrag{1}[c,t]{\tiny$\vphantom{\int}-1$}
    \psfrag{2}[c,t]{\tiny$\vphantom{\int}-2$}
    \psfrag{3}[c,t]{\tiny$\vphantom{\int}-3$}
    \psfrag{4}[c,t]{\tiny$\vphantom{\int}-4$}
    \psfrag{5}[c,t]{\tiny$\vphantom{\int}-5$}
    \psfrag{6}[c,t]{\tiny$\vphantom{\int}-6$}
    \psfrag{a}[r,b]{\tiny$0.0$}
    \psfrag{b}[r,c]{\tiny$0.2$}
    \psfrag{c}[r,c]{\tiny$0.4$}
    \psfrag{d}[r,c]{\tiny$0.6$}
    \psfrag{e}[r,c]{\tiny$0.8$}
    \psfrag{f}[r,c]{\tiny$1.0$}
        \includegraphics[width=0.22\textwidth]{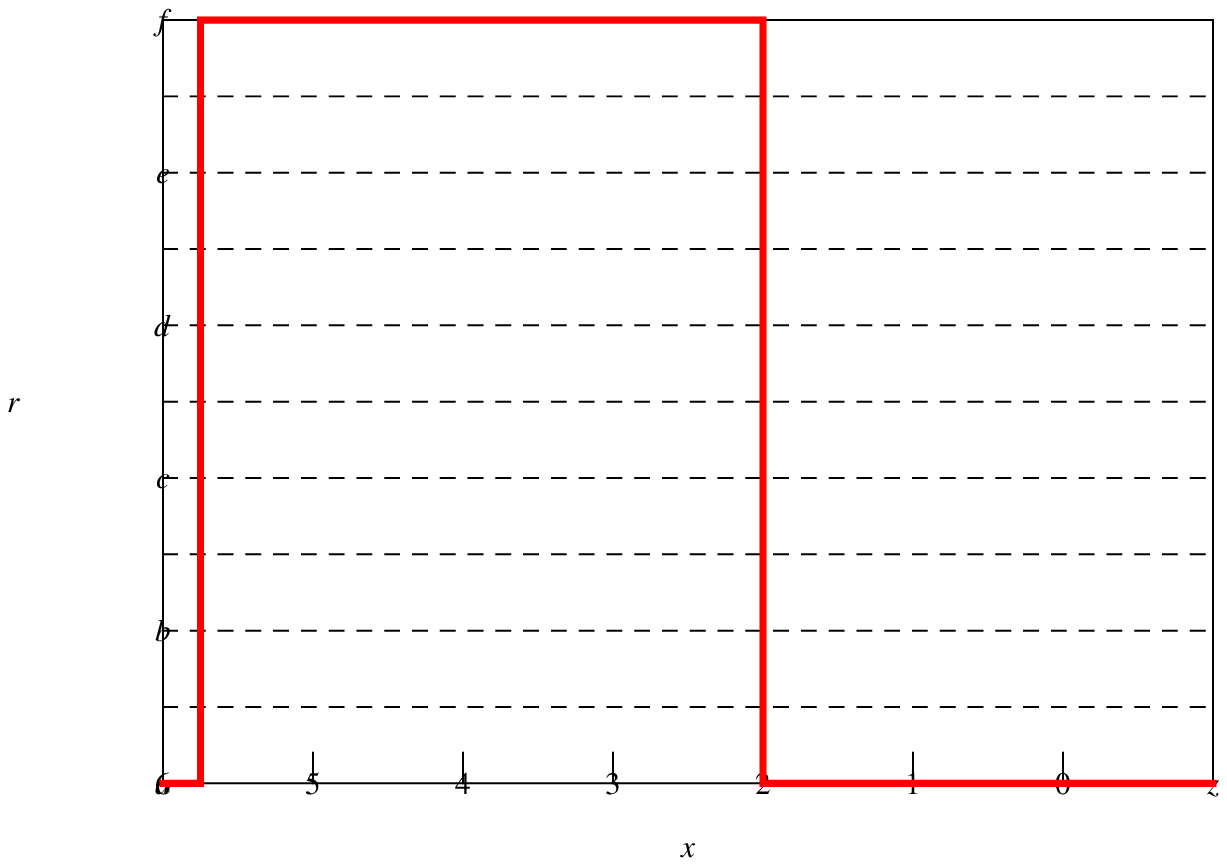}
    \end{psfrags}\label{fig:zebra1}}
\end{subfigure}~
\begin{subfigure}[$\rho(1,x)$]
    {\begin{psfrags}
    \psfrag{r}[c,c]{}
    \psfrag{x}[c,t]{\tiny$\vphantom{\int}x$}
    \psfrag{z}[c,t]{\tiny$\vphantom{\int}1$}
    \psfrag{0}[c,t]{\tiny$\vphantom{\int}0$}
    \psfrag{1}[c,t]{\tiny$\vphantom{\int}-1$}
    \psfrag{2}[c,t]{\tiny$\vphantom{\int}-2$}
    \psfrag{3}[c,t]{\tiny$\vphantom{\int}-3$}
    \psfrag{4}[c,t]{\tiny$\vphantom{\int}-4$}
    \psfrag{5}[c,t]{\tiny$\vphantom{\int}-5$}
    \psfrag{6}[c,t]{\tiny$\vphantom{\int}-6$}
    \psfrag{a}[r,b]{\tiny$0.0$}
    \psfrag{b}[r,c]{\tiny$0.2$}
    \psfrag{c}[r,c]{\tiny$0.4$}
    \psfrag{d}[r,c]{\tiny$0.6$}
    \psfrag{e}[r,c]{\tiny$0.8$}
    \psfrag{f}[r,c]{\tiny$1.0$}
        \includegraphics[width=0.22\textwidth]{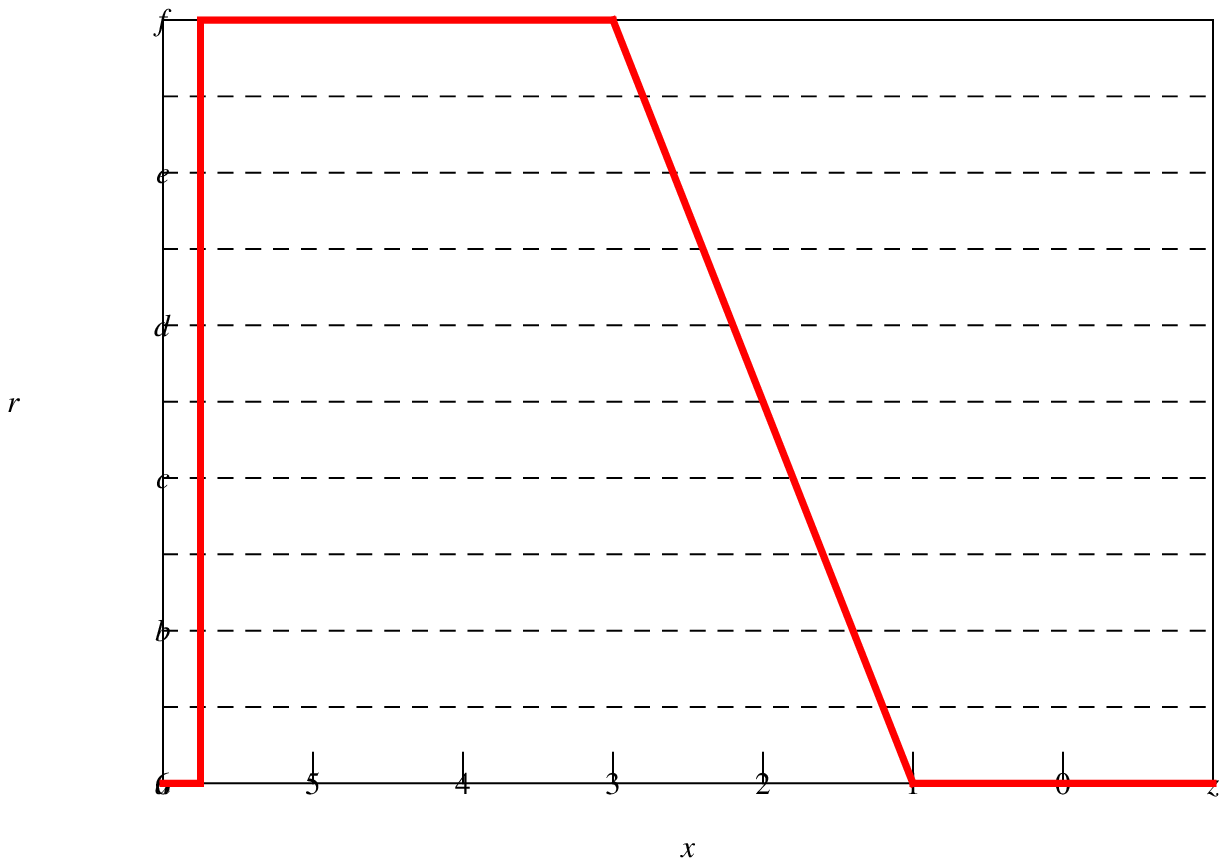}
    \end{psfrags}\label{fig:zebra2}}
\end{subfigure}~
\begin{subfigure}[$\rho(7.325,x)$]
    {\begin{psfrags}
    \psfrag{r}[c,c]{}
    \psfrag{x}[c,t]{\tiny$\vphantom{\int}x$}
    \psfrag{z}[c,t]{\tiny$\vphantom{\int}1$}
    \psfrag{0}[c,t]{\tiny$\vphantom{\int}0$}
    \psfrag{1}[c,t]{\tiny$\vphantom{\int}-1$}
    \psfrag{2}[c,t]{\tiny$\vphantom{\int}-2$}
    \psfrag{3}[c,t]{\tiny$\vphantom{\int}-3$}
    \psfrag{4}[c,t]{\tiny$\vphantom{\int}-4$}
    \psfrag{5}[c,t]{\tiny$\vphantom{\int}-5$}
    \psfrag{6}[c,t]{\tiny$\vphantom{\int}-6$}
    \psfrag{a}[r,b]{\tiny$0.0$}
    \psfrag{b}[r,c]{\tiny$0.2$}
    \psfrag{c}[r,c]{\tiny$0.4$}
    \psfrag{d}[r,c]{\tiny$0.6$}
    \psfrag{e}[r,c]{\tiny$0.8$}
    \psfrag{f}[r,c]{\tiny$1.0$}
        \includegraphics[width=0.22\textwidth]{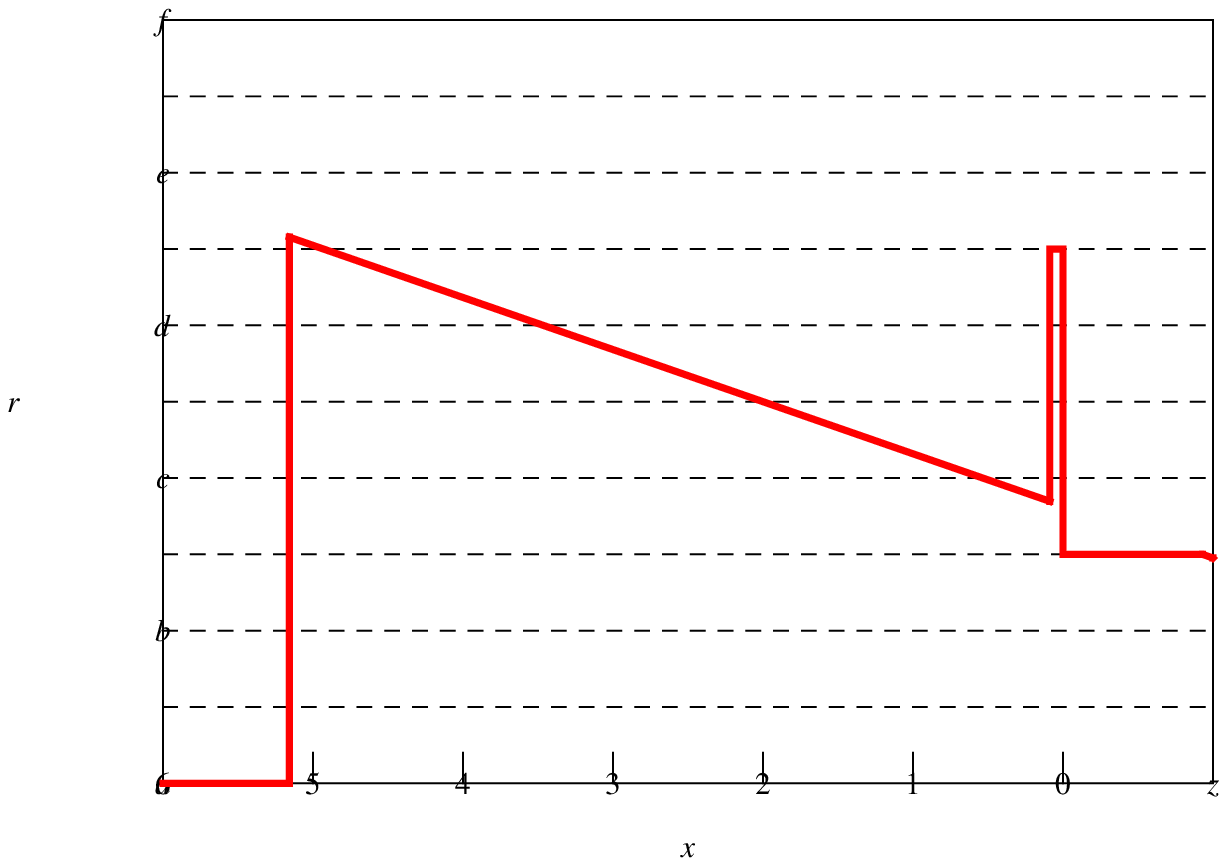}
    \end{psfrags}\label{fig:zebra3}}
\end{subfigure}~
\begin{subfigure}[$\rho(10,x)$]
    {\begin{psfrags}
    \psfrag{r}[c,c]{}
    \psfrag{x}[c,t]{\tiny$\vphantom{\int}x$}
    \psfrag{z}[c,t]{\tiny$\vphantom{\int}1$}
    \psfrag{0}[c,t]{\tiny$\vphantom{\int}0$}
    \psfrag{1}[c,t]{\tiny$\vphantom{\int}-1$}
    \psfrag{2}[c,t]{\tiny$\vphantom{\int}-2$}
    \psfrag{3}[c,t]{\tiny$\vphantom{\int}-3$}
    \psfrag{4}[c,t]{\tiny$\vphantom{\int}-4$}
    \psfrag{5}[c,t]{\tiny$\vphantom{\int}-5$}
    \psfrag{6}[c,t]{\tiny$\vphantom{\int}-6$}
    \psfrag{a}[r,b]{\tiny$0.0$}
    \psfrag{b}[r,c]{\tiny$0.2$}
    \psfrag{c}[r,c]{\tiny$0.4$}
    \psfrag{d}[r,c]{\tiny$0.6$}
    \psfrag{e}[r,c]{\tiny$0.8$}
    \psfrag{f}[r,c]{\tiny$1.0$}
        \includegraphics[width=0.22\textwidth]{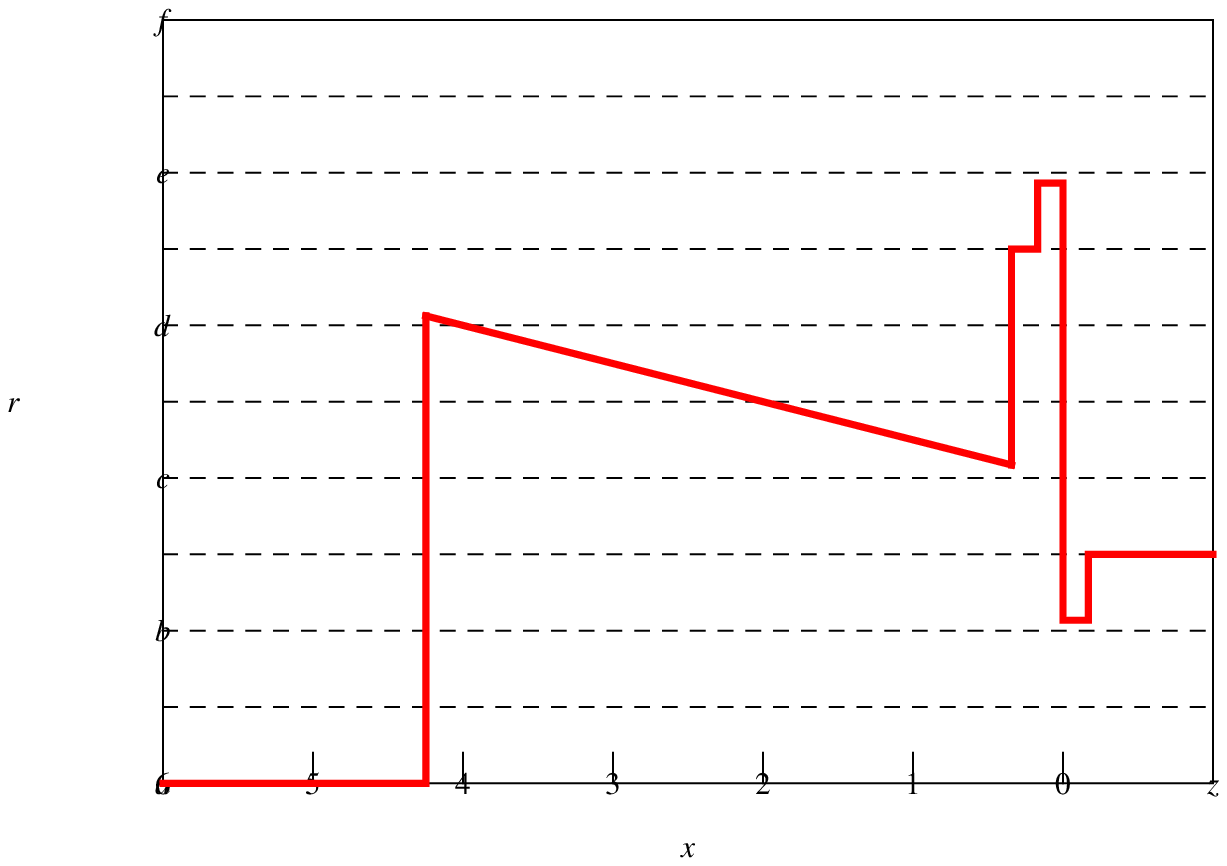}
    \end{psfrags}\label{fig:zebra4}}
\end{subfigure}\\
\begin{subfigure}[$\rho_\Delta(11.939,x)$]
    {\begin{psfrags}
    \psfrag{rho}[c,b]{}
    \psfrag{x}[c,t]{\tiny$x$}
        \includegraphics[width=0.22\textwidth]{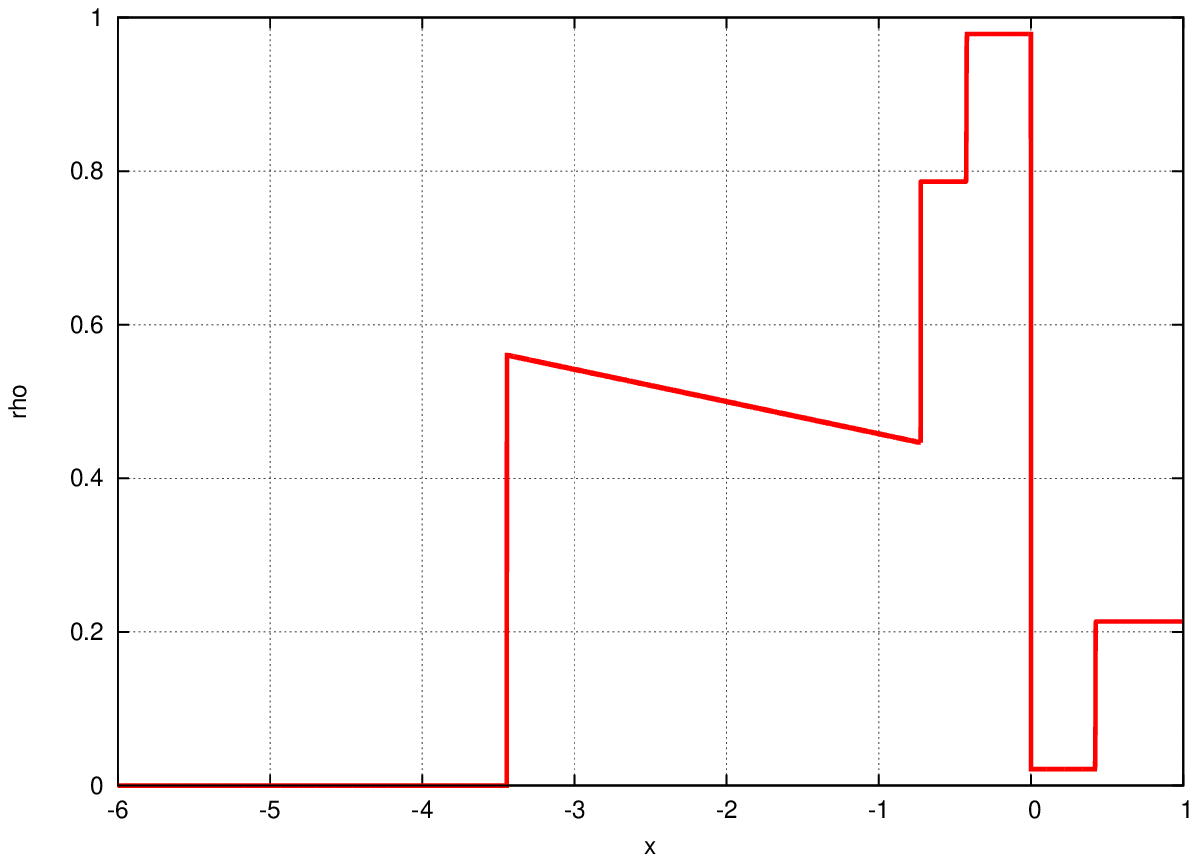}
    \end{psfrags}\label{fig:tiger5}}
\end{subfigure}~
\begin{subfigure}[$\rho_\Delta(85.2,x)$]
    {\begin{psfrags}
    \psfrag{rho}[c,b]{}
    \psfrag{x}[c,t]{\tiny$x$}
        \includegraphics[width=0.22\textwidth]{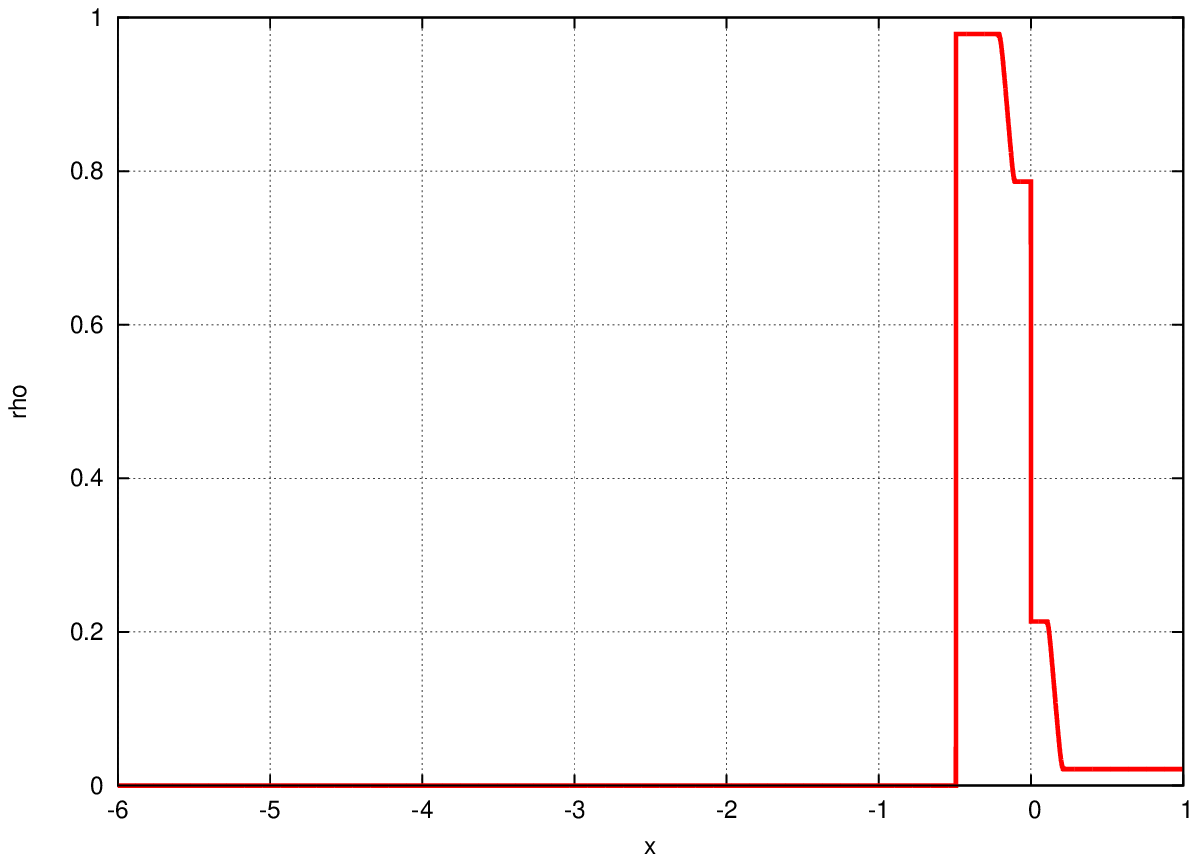}
    \end{psfrags}\label{fig:tiger6}}
\end{subfigure}~
\begin{subfigure}[$\rho_\Delta(85.5526,x)$]
    {\begin{psfrags}
    \psfrag{rho}[c,b]{}
    \psfrag{x}[c,t]{\tiny$x$}
        \includegraphics[width=0.22\textwidth]{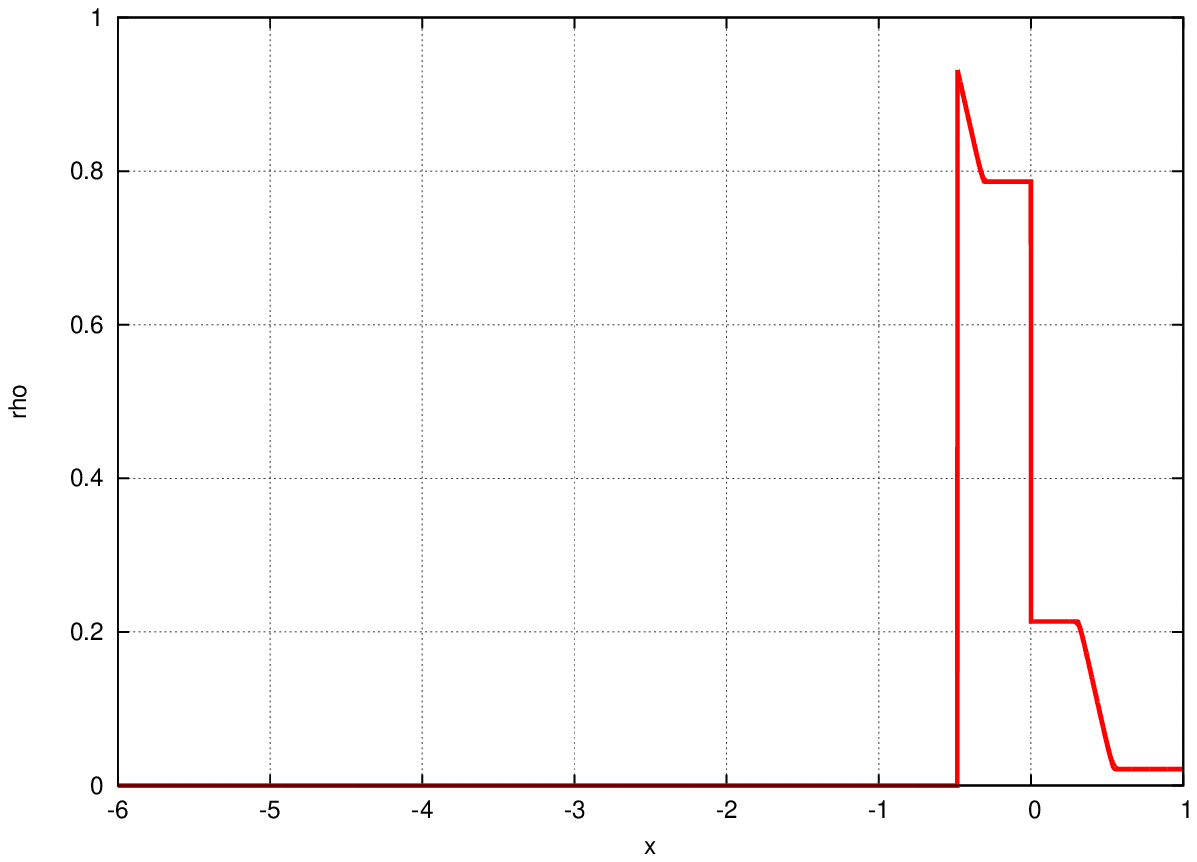}
    \end{psfrags}\label{fig:tiger7}}
\end{subfigure}~
\begin{subfigure}[$\rho_\Delta(87.4981,x)$]
    {\begin{psfrags}
    \psfrag{rho}[c,b]{}
    \psfrag{x}[c,t]{\tiny$x$}
        \includegraphics[width=0.22\textwidth]{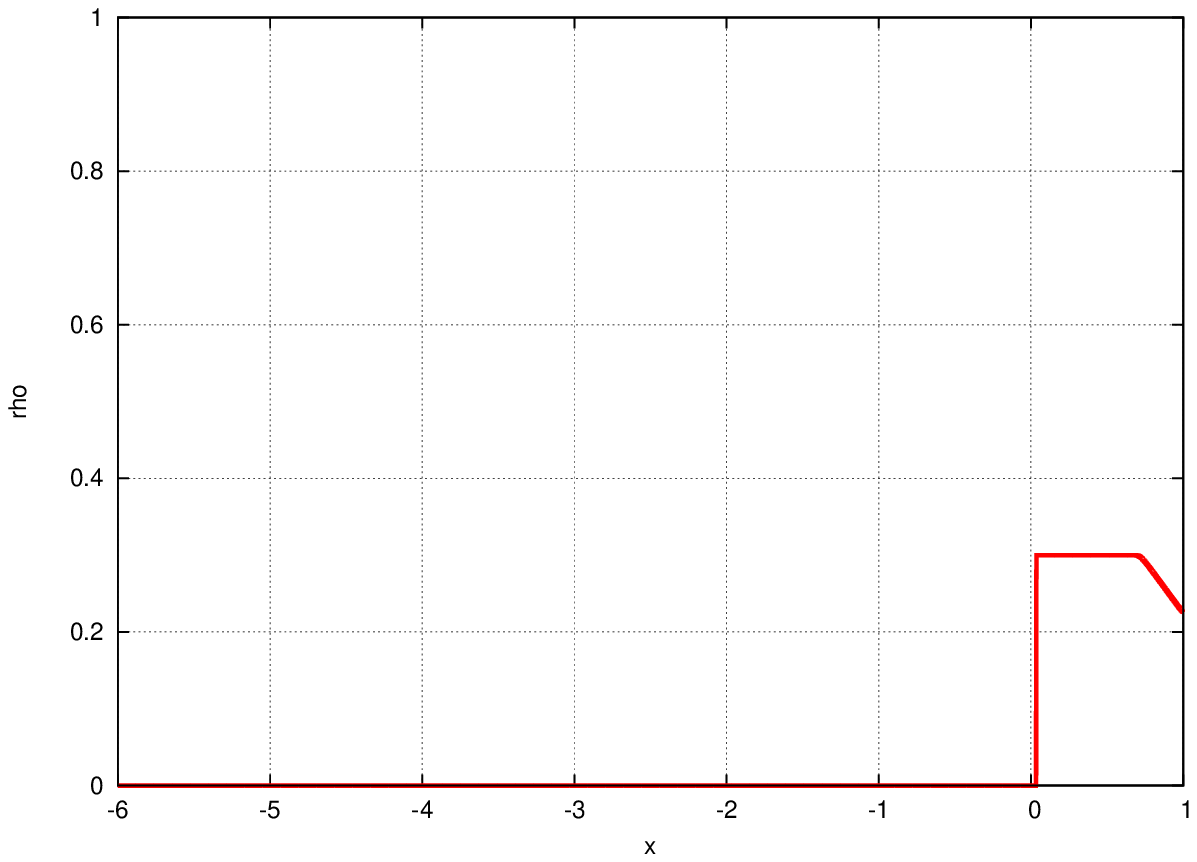}
    \end{psfrags}\label{fig:tiger8}}
\end{subfigure}\\
\begin{subfigure}[$\rho(11.939,x)$]
    {\begin{psfrags}
    \psfrag{r}[c,c]{}
    \psfrag{x}[c,t]{\tiny$\vphantom{\int}x$}
    \psfrag{z}[c,t]{\tiny$\vphantom{\int}1$}
    \psfrag{0}[c,t]{\tiny$\vphantom{\int}0$}
    \psfrag{1}[c,t]{\tiny$\vphantom{\int}-1$}
    \psfrag{2}[c,t]{\tiny$\vphantom{\int}-2$}
    \psfrag{3}[c,t]{\tiny$\vphantom{\int}-3$}
    \psfrag{4}[c,t]{\tiny$\vphantom{\int}-4$}
    \psfrag{5}[c,t]{\tiny$\vphantom{\int}-5$}
    \psfrag{6}[c,t]{\tiny$\vphantom{\int}-6$}
    \psfrag{a}[r,b]{\tiny$0.0$}
    \psfrag{b}[r,c]{\tiny$0.2$}
    \psfrag{c}[r,c]{\tiny$0.4$}
    \psfrag{d}[r,c]{\tiny$0.6$}
    \psfrag{e}[r,c]{\tiny$0.8$}
    \psfrag{f}[r,c]{\tiny$1.0$}
        \includegraphics[width=0.22\textwidth]{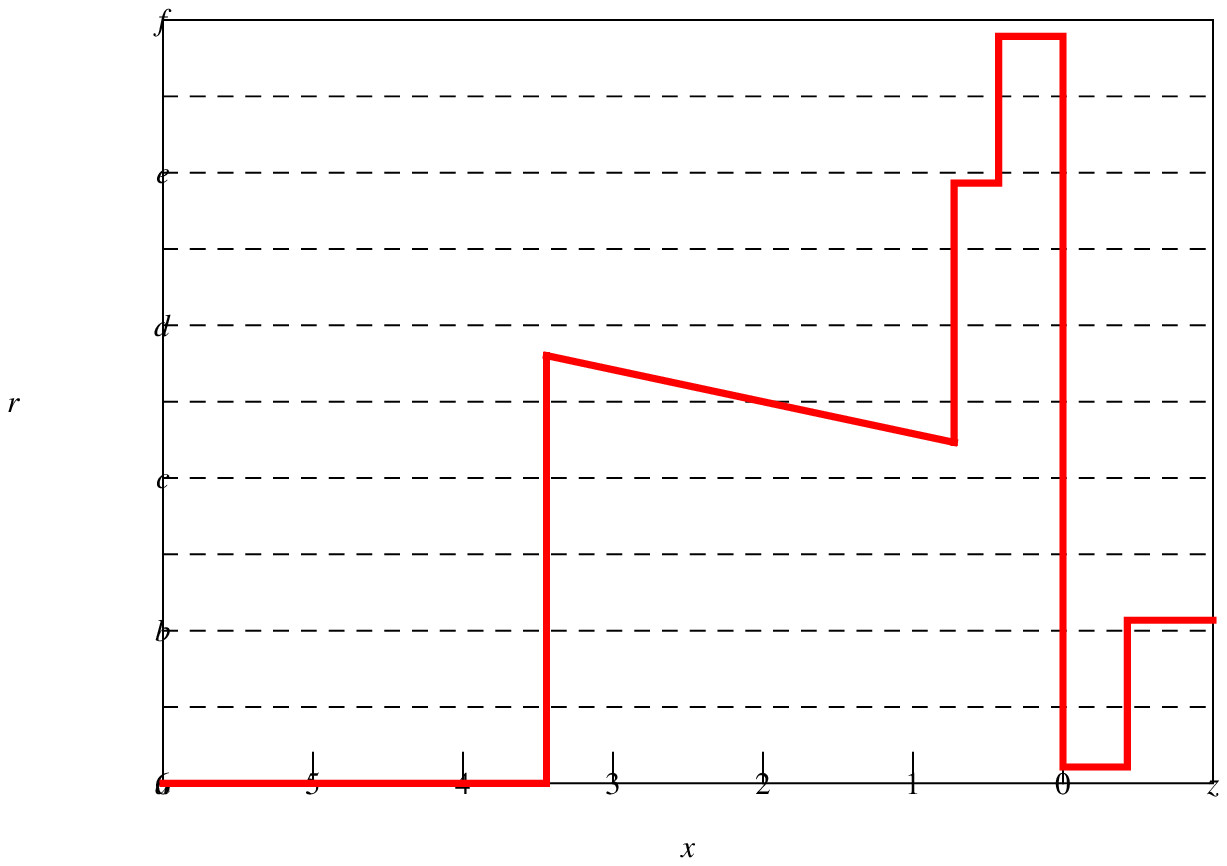}
    \end{psfrags}\label{fig:zebra5}}
\end{subfigure}~
\begin{subfigure}[$\rho(85.2,x)$]
    {\begin{psfrags}
    \psfrag{r}[c,c]{}
    \psfrag{x}[c,t]{\tiny$\vphantom{\int}x$}
    \psfrag{z}[c,t]{\tiny$\vphantom{\int}1$}
    \psfrag{0}[c,t]{\tiny$\vphantom{\int}0$}
    \psfrag{1}[c,t]{\tiny$\vphantom{\int}-1$}
    \psfrag{2}[c,t]{\tiny$\vphantom{\int}-2$}
    \psfrag{3}[c,t]{\tiny$\vphantom{\int}-3$}
    \psfrag{4}[c,t]{\tiny$\vphantom{\int}-4$}
    \psfrag{5}[c,t]{\tiny$\vphantom{\int}-5$}
    \psfrag{6}[c,t]{\tiny$\vphantom{\int}-6$}
    \psfrag{a}[r,b]{\tiny$0.0$}
    \psfrag{b}[r,c]{\tiny$0.2$}
    \psfrag{c}[r,c]{\tiny$0.4$}
    \psfrag{d}[r,c]{\tiny$0.6$}
    \psfrag{e}[r,c]{\tiny$0.8$}
    \psfrag{f}[r,c]{\tiny$1.0$}
        \includegraphics[width=0.22\textwidth]{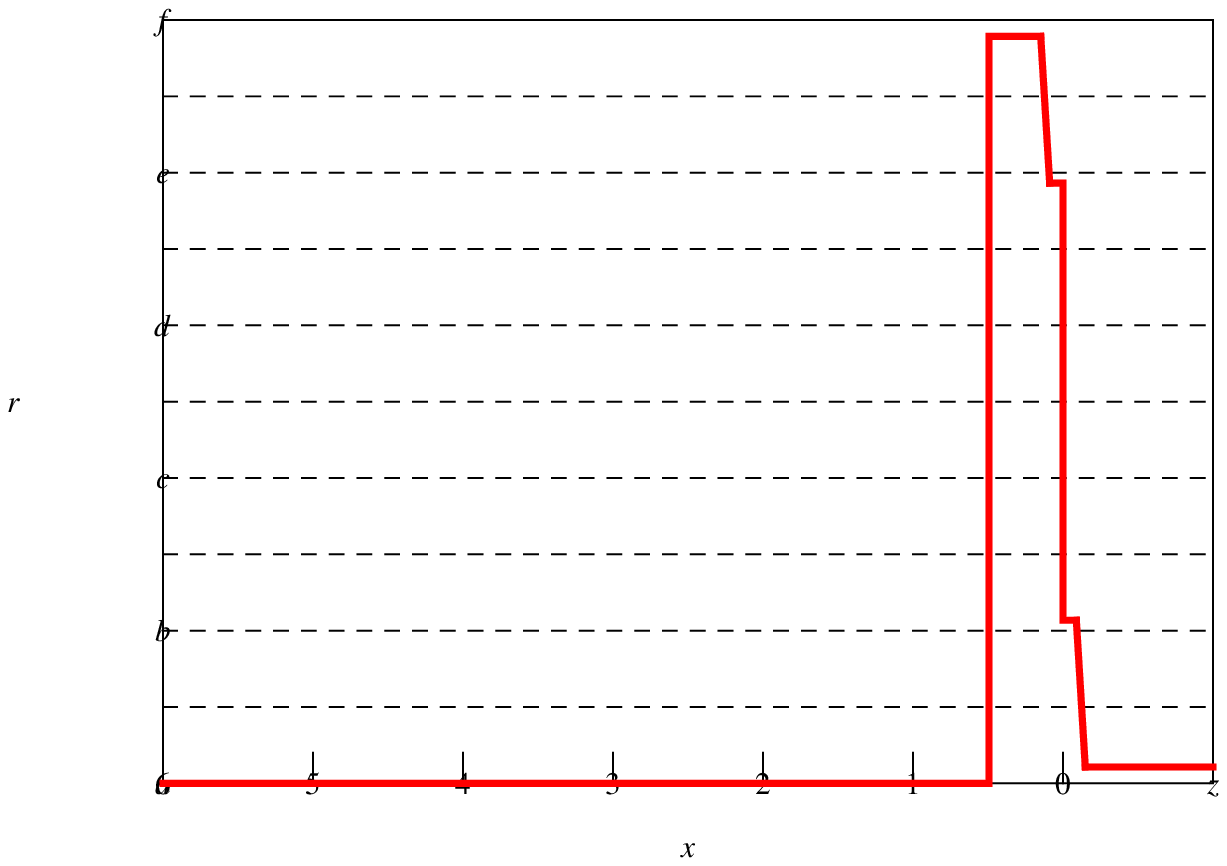}
    \end{psfrags}\label{fig:zebra6}}
\end{subfigure}~
\begin{subfigure}[$\rho(85.5526,x)$]
    {\begin{psfrags}
    \psfrag{r}[c,c]{}
    \psfrag{x}[c,t]{\tiny$\vphantom{\int}x$}
    \psfrag{z}[c,t]{\tiny$\vphantom{\int}1$}
    \psfrag{0}[c,t]{\tiny$\vphantom{\int}0$}
    \psfrag{1}[c,t]{\tiny$\vphantom{\int}-1$}
    \psfrag{2}[c,t]{\tiny$\vphantom{\int}-2$}
    \psfrag{3}[c,t]{\tiny$\vphantom{\int}-3$}
    \psfrag{4}[c,t]{\tiny$\vphantom{\int}-4$}
    \psfrag{5}[c,t]{\tiny$\vphantom{\int}-5$}
    \psfrag{6}[c,t]{\tiny$\vphantom{\int}-6$}
    \psfrag{a}[r,b]{\tiny$0.0$}
    \psfrag{b}[r,c]{\tiny$0.2$}
    \psfrag{c}[r,c]{\tiny$0.4$}
    \psfrag{d}[r,c]{\tiny$0.6$}
    \psfrag{e}[r,c]{\tiny$0.8$}
    \psfrag{f}[r,c]{\tiny$1.0$}
        \includegraphics[width=0.22\textwidth]{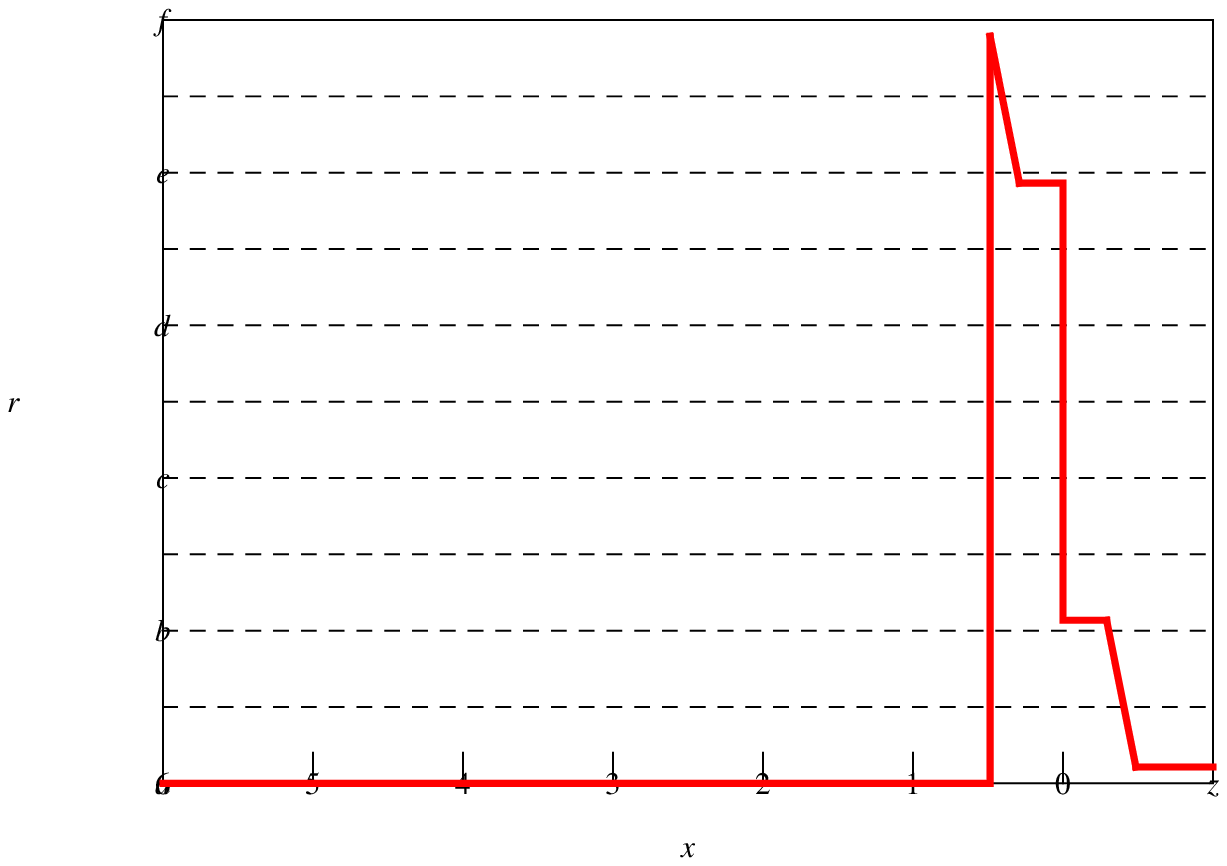}
    \end{psfrags}\label{fig:zebra7}}
\end{subfigure}~
\begin{subfigure}[$\rho(87.4981,x)$]
    {\begin{psfrags}
    \psfrag{r}[c,c]{}
    \psfrag{x}[c,t]{\tiny$\vphantom{\int}x$}
    \psfrag{z}[c,t]{\tiny$\vphantom{\int}1$}
    \psfrag{0}[c,t]{\tiny$\vphantom{\int}0$}
    \psfrag{1}[c,t]{\tiny$\vphantom{\int}-1$}
    \psfrag{2}[c,t]{\tiny$\vphantom{\int}-2$}
    \psfrag{3}[c,t]{\tiny$\vphantom{\int}-3$}
    \psfrag{4}[c,t]{\tiny$\vphantom{\int}-4$}
    \psfrag{5}[c,t]{\tiny$\vphantom{\int}-5$}
    \psfrag{6}[c,t]{\tiny$\vphantom{\int}-6$}
    \psfrag{a}[r,b]{\tiny$0.0$}
    \psfrag{b}[r,c]{\tiny$0.2$}
    \psfrag{c}[r,c]{\tiny$0.4$}
    \psfrag{d}[r,c]{\tiny$0.6$}
    \psfrag{e}[r,c]{\tiny$0.8$}
    \psfrag{f}[r,c]{\tiny$1.0$}
        \includegraphics[width=0.22\textwidth]{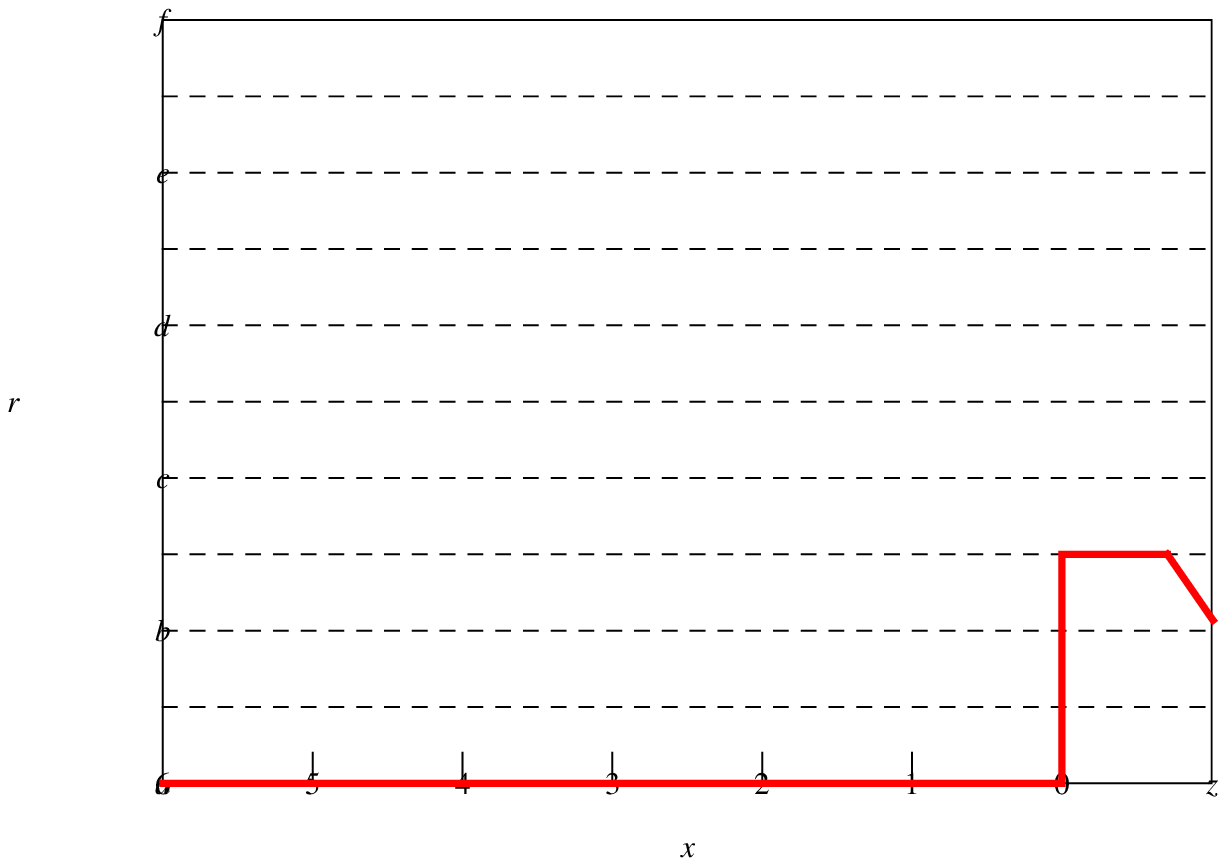}
    \end{psfrags}\label{fig:zebra8}}
\end{subfigure}\\
\caption{With reference to Subsection \ref{sec:validation}: The numerically computed solution $x \mapsto \rho_\Delta(t,x)$ and the explicitly computed solution $x \mapsto \rho(t,x)$ at different fixed times $t$.}\label{fig:validation3}
\end{figure}

\noindent A qualitative comparison between the numerically computed solution $x \mapsto \rho_\Delta(t,x)$ and the explicitly computed solution $x \mapsto \rho(t,x)$ at different fixed times $t$ is in Figure~\ref{fig:validation3}. We observe good agreements between $x \mapsto \rho(t,x)$ and $x \mapsto\rho_\Delta(t,x)$. The parameters for the numerically computed solution are $\Delta x=3.5\times10^{-4}$ and $\Delta t=7\times10^{-5}$.\\
A convergence analysis is also performed for this test. We introduce the relative $\L1$-error for the density $\rho$, at a given time $t^n$, defined by
$$E_{\L1}^n=\left [\sum_{j}\left|\rho(t^n,x_j)-\rho_j^n\right|\right ]\,\Big/\left [\sum_{j}\left|\rho(t^n,x_j)\right|\right ].$$
In Table~\ref{erreurs}, we computed the relative $\L1$-errors for different numbers of space cells at the fixed time $t=10$. We deduce that the order of convergence is approximatively $0.906$. As in~\cite{scontrainte}, we observe that the modification~\eqref{def.flux.num.contrainte} of the numerical flux does not affect the accuracy of the scheme. 

\begin{table}[ht]
\begin{center}
\begin{tabular}{|c|c|}
\hline
Number of cells & $\L1$-error\\
\hline
$625$ & $9.6843\times10^{-3}$\\
\hline
$1250$ & $6.2514\times10^{-3}$\\
\hline
$2500$ & $3.4143\times10^{-3}$\\
\hline
$5000$ & $1.3172\times10^{-3}$\\
\hline
$10000$ & $1.03\times10^{-3}$\\
\hline
$20000$ & $4.2544\times10^{-4}$\\
\hline
Order & $0.906$\\
\hline
\end{tabular}
\tiny\caption{Relative $\L1$-error at time $t=10$.}
\label{erreurs}
\end{center}
\end{table}


\section{Numerical simulations}\label{sec:simulations}

This section is devoted to the phenomenological description of some collective effects in crowd dynamics related to capacity drop, namely the Braess' paradox and the Faster Is Slower (FIS) effect.





\subsection{Faster is Slower effect}\label{sec:FIS}

The FIS effect was first described in~\cite{Helbing2000Simulating, Parisi2005606} in the context of the room evacuation problem. The authors studied the evolution of the evacuation time as a function of the maximal velocity reached by the pedestrians, and they shown that there exists an optimal velocity for which the evacuation time attains a minimum. Therefore, any acceleration beyond the optimal velocity worses the evacuation time. 
Following the studies above, the curve representing the evacuation time as a function of the average velocity takes a characteristic shape \cite[Figure 1c]{Parisi2005606}. 

The first numerical tests we performed aim to verify if such shape is obtained starting from the ADR model. 
To this end, we consider the corridor modeled by the segment [-6,1], with an exit at $x=0$. We consider the flux $f(\rho)=\rho \, v_{\max} \, (1-\rho)$ where $v_{\max}$ is the maximal velocity of the pedestrians and the maximal density is equal to one. We use the same weight function as for the validation of the scheme, $w(x)=2(1+x)\chi_{[-1,0]}(x)$ and, the same initial density, $\bar{\rho}=\chi_{[-5.75,-2]}$. The efficiency of the exit $p$ is now given by the following continuous function
\begin{eqnarray}
\label{P_continuous}
    p(\xi) &= \left\{
    \begin{array}{l@{\quad\text{ if }}l}
       p_0 & 0\le\xi<\xi_1,\\[6pt]
       \displaystyle\frac{(p_0-p_1)\xi+p_1\xi_1-p_0\xi_2}{\xi_1-\xi_2} & \xi_1\le\xi<\xi_2,\\[10pt]
       p_1 & \xi_2\le\xi\le1,
    \end{array}
    \right.
\end{eqnarray}  
where
\begin{align*}
    &p_0 =0.24, && p_1 =0.05, && \xi_1=0.5, && \xi_2=0.9.
\end{align*}

The space and time steps are fixed to $\Delta x=5\times10^{-3}$ and $\Delta t=5\times10^{-4}$. In Figure~\ref{fis_f_p} are plotted the flux $f$ corresponding to the maximal velocity $v_{\max}=1$ and the above efficiency of the exit.

\begin{figure}[ht]
\psfrag{rho, xi}[][][0.8]{$\rho,\,\xi$}
\psfrag{f}[][][0.3]{$f$}
\psfrag{p}[][][0.3]{$p$}
\psfrag{f, p}[][][0.8]{$f,\,p$}
\centering\includegraphics[width=0.48\hsize]{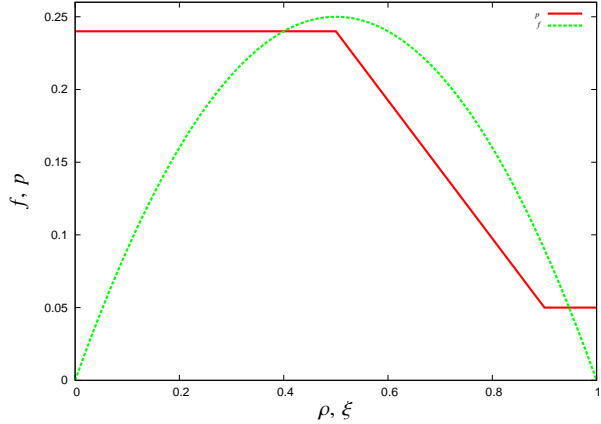}
\tiny \caption {The normalized flux $\rho\to f(\rho)$ and the constraint $\xi\to p(\xi)$ defined in~\eqref{P_continuous}.}
\label{fis_f_p}
\end{figure}

\begin{figure}[ht]
\psfrag{velocity}[][][0.8]{$v_{\max}$}
\psfrag{exit time}[][][0.8]{evacuation time}
\centering\includegraphics[width=0.48\hsize]{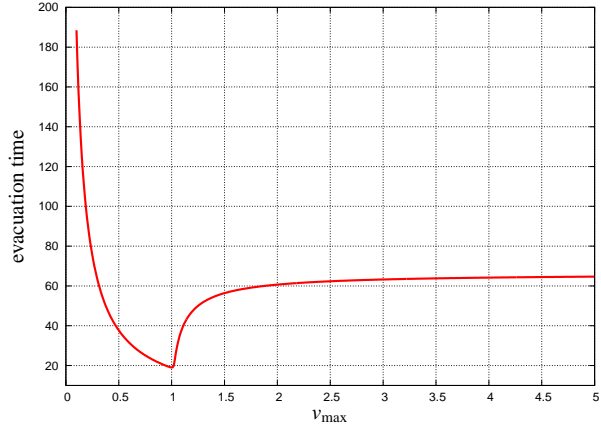}
\tiny \caption {With reference to Subsection \ref{sec:FIS}: Evacuation time as a function of the velocity $v_{\max}$.}
\label{fis_exit_times_velocities}
\end{figure}

\begin{figure}[ht]
\begin{subfigure}[$\rho\mapsto\rho_\Delta(0,t)$ for velocities $v_{\max}\le1$.]
{\begin{psfrags}
\psfrag{v1}[][][0.3]{$v_{\max}=1$\qquad}
\psfrag{v095}[][][0.3]{$v_{\max}=0.95$\qquad}
\psfrag{v099}[][][0.3]{$v_{\max}=0.99$\qquad}
\psfrag{rho}[][][0.6]{}
\psfrag{t}[][][0.8]{t}
\includegraphics[width=0.48\hsize]{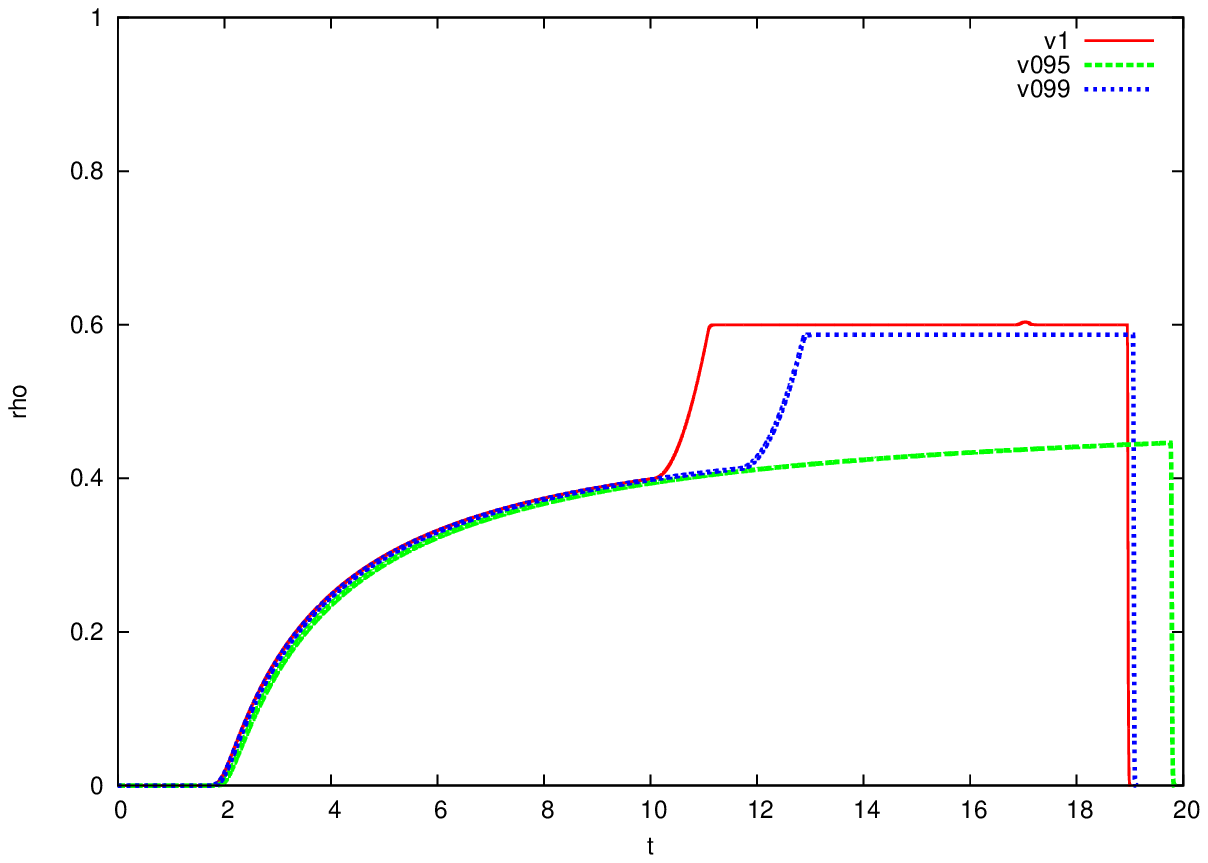}
\end{psfrags}}
\end{subfigure}~
\begin{subfigure}[$\rho\mapsto\rho_\Delta(0,t)$ for velocities $v_{\max}\ge1$.]
{\begin{psfrags}
\psfrag{v1}[][][0.3]{$v_{\max}=1$\qquad}
\psfrag{v101}[][][0.3]{$v_{\max}=1.01$\qquad}
\psfrag{v105}[][][0.3]{$v_{\max}=1.05$\qquad}
\psfrag{rho}[][][0.8]{}
\psfrag{t}[][][0.8]{t}
\includegraphics[width=0.48\hsize]{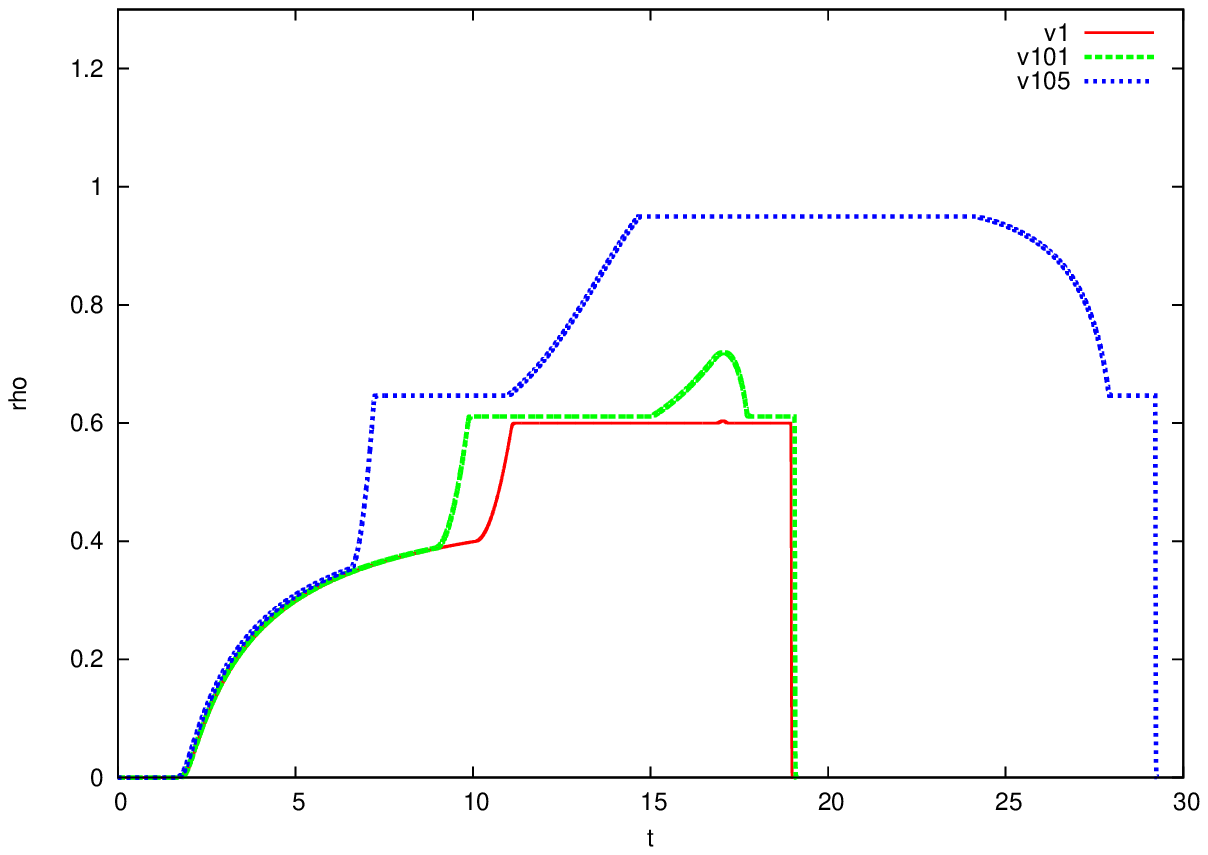}
\end{psfrags}}
\end{subfigure}
\tiny \caption{With reference to Subsection \ref{sec:FIS}: Densities at the exit as a function of time for different velocities.}
\label{fis_densities_doors}
\end{figure}

\begin{figure}[ht]
\psfrag{p1}[r][][0.3]{$\beta=1$}
\psfrag{p2}[r][][0.3]{$\beta=0.9$}
\psfrag{p3}[r][][0.3]{$\beta=0.8$}
\psfrag{xi}[][][0.8]{$\xi$}
\psfrag{p}[][][0.8]{$p_\beta$}
\centering\includegraphics[width=0.48\hsize]{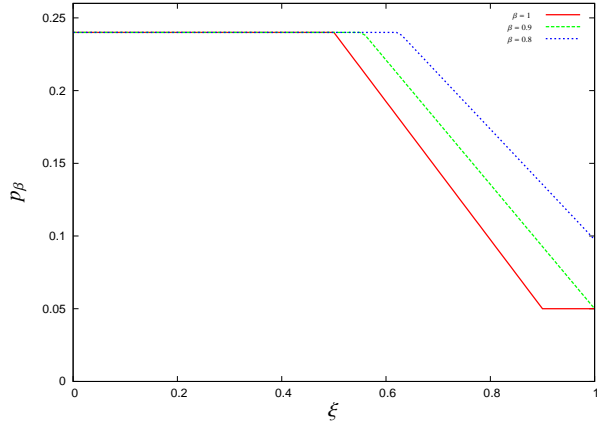}
\tiny \caption {With reference to Subsection \ref{sec:FIS}: The efficiencies $\xi\to p_\beta(\xi)$ for $\beta=0.8, 0.9, 1$.}
\label{p_beta}
\end{figure}

\begin{figure}[ht]\centering
\begin{subfigure}[Evacuation time as a function of $v_{\max}$ for different amounts of initial densities.]
{\begin{psfrags}
\psfrag{exit time}[][][0.8]{evacuation time}
\psfrag{init1}[l][][0.3]{$\bar{\rho}$}
\psfrag{init08}[l][][0.3]{$\bar{\rho}_1$}
\psfrag{init06}[l][][0.3]{$\bar{\rho}_2$}
\psfrag{velocity}[][][0.8]{$v_{\max}$}
\includegraphics[width=0.48\hsize]{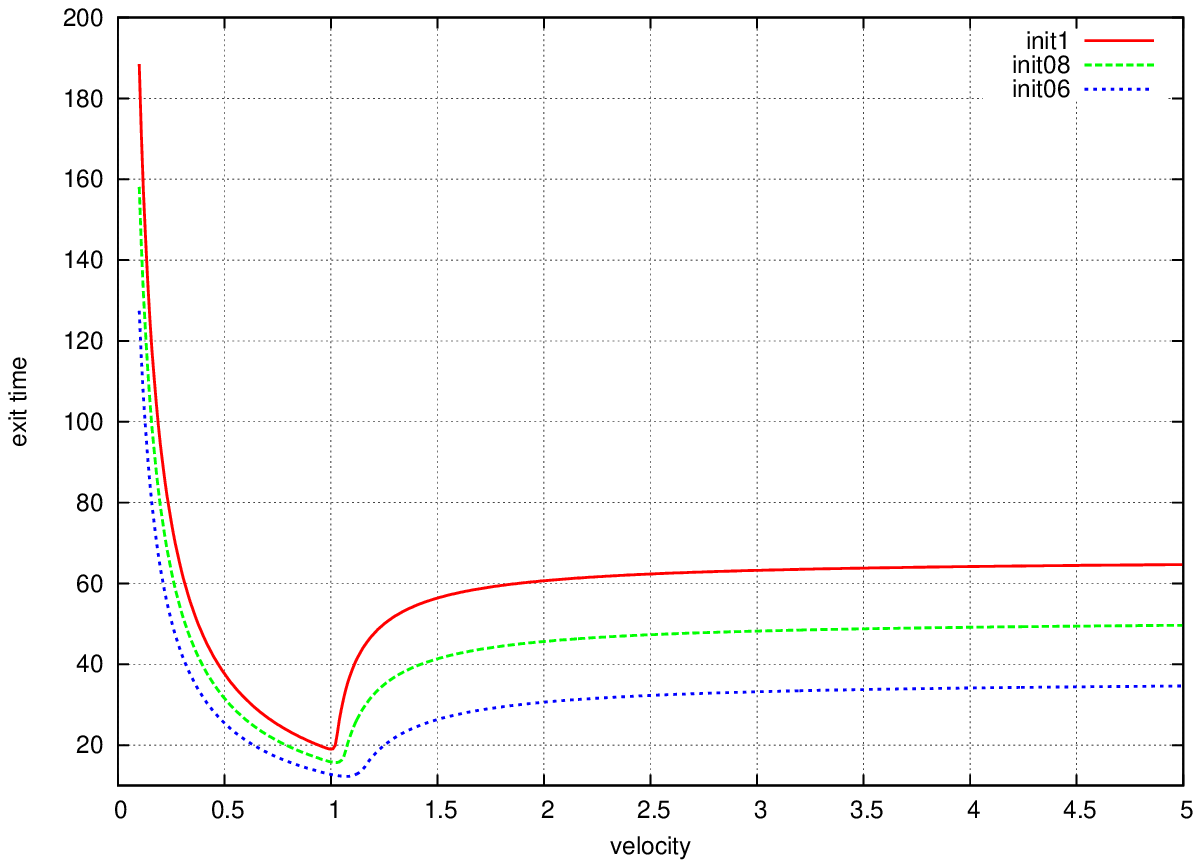}
\end{psfrags}}
\end{subfigure}
\begin{subfigure}[Evacuation time as a function of $v_{\max}$ for different efficiencies of the exit.]
{\begin{psfrags}
\psfrag{exit time}[][][0.8]{evacuation time}
\psfrag{beta1}[c][][0.3]{$\beta=1$}
\psfrag{beta09}[c][][0.3]{$\beta=0.9$}
\psfrag{beta08}[c][][0.3]{$\beta=0.8$}
\psfrag{velocity}[][][0.8]{$v_{\max}$}
\includegraphics[width=0.48\hsize]{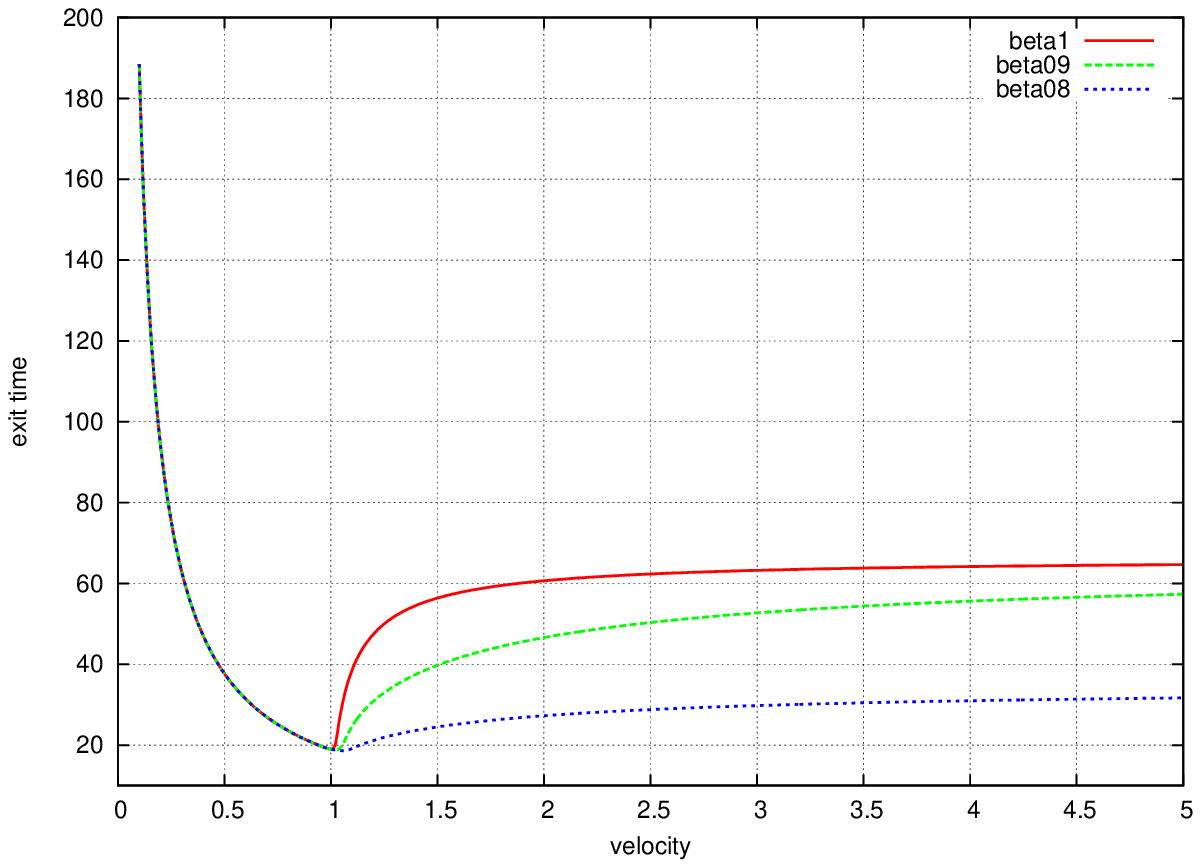}
\end{psfrags}}
\end{subfigure}
\begin{subfigure}[Evacuation time as a function of $v_{\max}$ for different locations of the initial density.]
{\begin{psfrags}
\psfrag{exit time}[][][0.8]{evacuation time}
\psfrag{init}[l][][0.3]{$\bar{\rho}$}
\psfrag{t12}[l][][0.3]{$\bar{\rho}_3$}
\psfrag{t20}[l][][0.3]{$\bar{\rho}_4$}
\psfrag{velocity}[][][0.8]{$v_{\max}$}
\includegraphics[width=0.48\hsize]{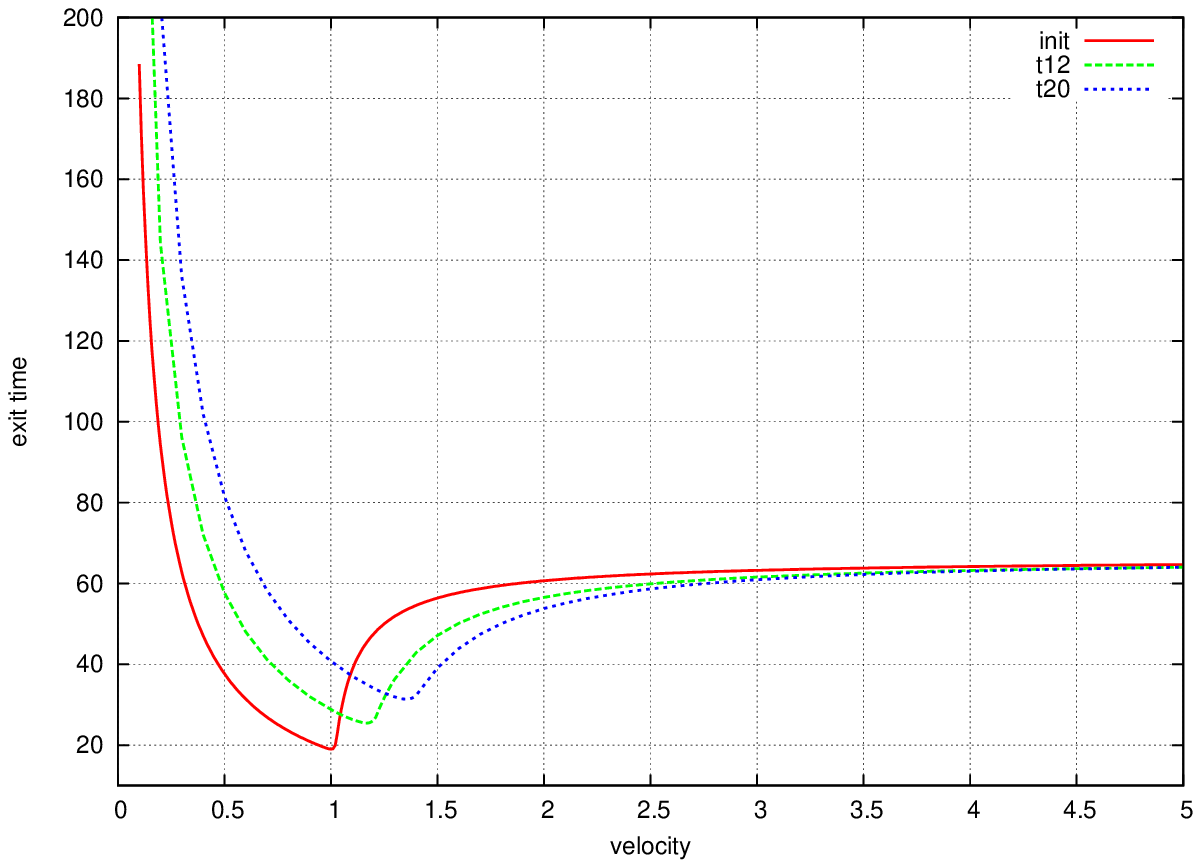}
\end{psfrags}}
\end{subfigure}
\tiny \caption {With reference to Subsection \ref{sec:FIS}: Evacuation time as a function of $v_{\max}$ for different parameters of the model.}
\label{fis_stabilities}
\end{figure}

Figure~\ref{fis_exit_times_velocities} represents the evacuation time as a function of the maximal velocity $v_{\max}$, as $v_{\max}$ varies in the interval $[0.1 , 5]$. As we can observe, the general shape described above is recovered. The numerical minimal evacuation time is $19.007$ and is obtained for $v_{\max}=1$.

In addition, we reported in Figure~\ref{fis_densities_doors} the density at the exit as a function of time for different values of the maximal velocity $v_{\max}$ around the optimal one. We notice that the maximal density at the exit and the time length where the density is maximal increase with the velocity. This expresses the jamming at the exit that leads to the FIS effect. 

\begin{figure}
\centering
\psfrag{t}[][][0.8]{evacuation time}
\psfrag{d}[][][0.8]{position of the obstacle}
\psfrag{A}[][][0.3]{with obstacle\hspace{1.5cm}}
\psfrag{B}[][][0.3]{without obstacle\hspace{2cm}}
\psfrag{D}[][][0.8]{$d$}
\includegraphics[width=0.48\hsize]{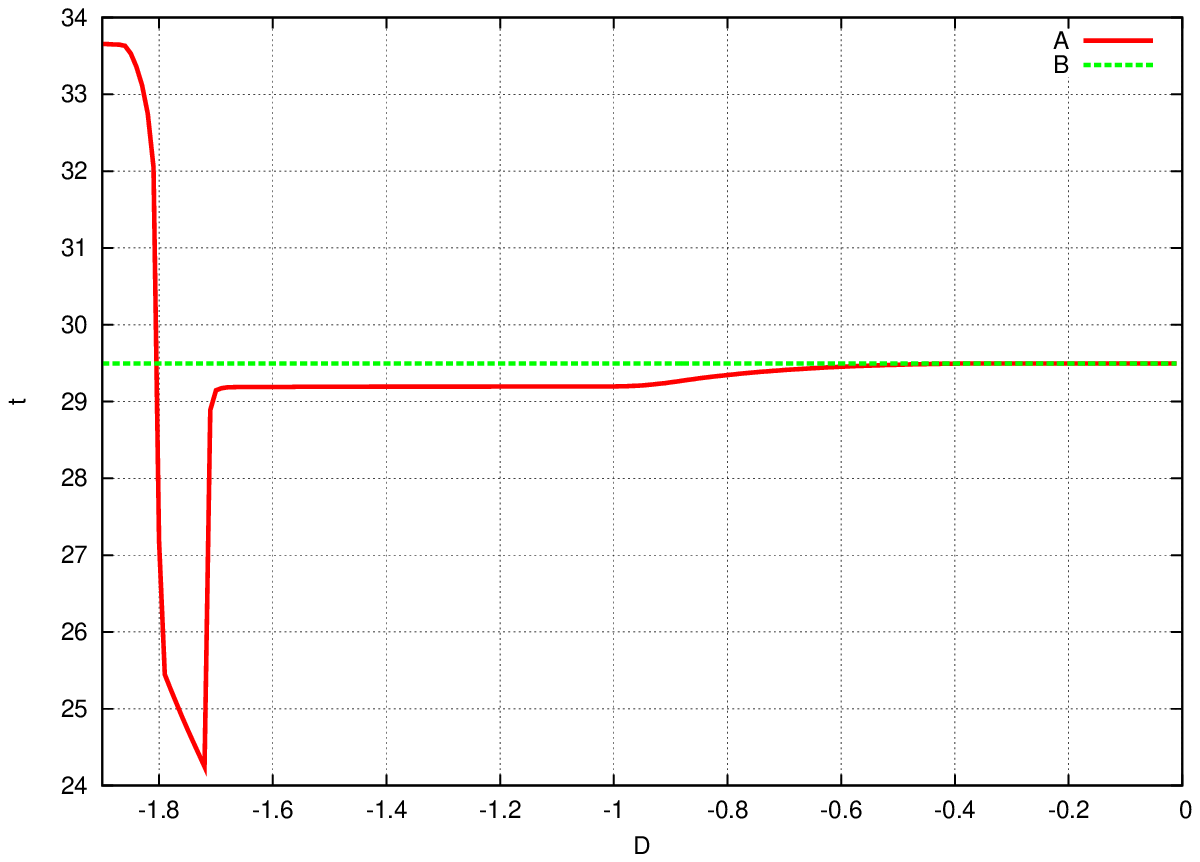}
\tiny \caption{With reference to Subsection \ref{sec:Braess}: Evacuation time as a function of the position of the obstacle.}
\label{braess_exit_times}
\end{figure}

\begin{figure}
\centering
\renewcommand{\arraystretch}{1.5}
\begin{tabular}{>{\centering\bfseries}m{0.05\hsize} @{}>{\centering}m{0.3\hsize} @{}>{\centering}m{0.3\hsize} @{}>{\centering\arraybackslash}m{0.3\hsize}}
& Without obstacle & Obstacle at $d=-1.85$ & Obstacle at $d=-1.72$\\
\rotatebox{90}{$t=1$}&
{\begin{psfrags}
\psfrag{t=1}[r][][0.4]{$\displaystyle\vphantom{\int^{\int^{\int}}}\rho(1,x)$}
\psfrag{rho}[][][0.8]{}
\psfrag{x}[][][0.8]{$x$}
\includegraphics[width=\hsize]{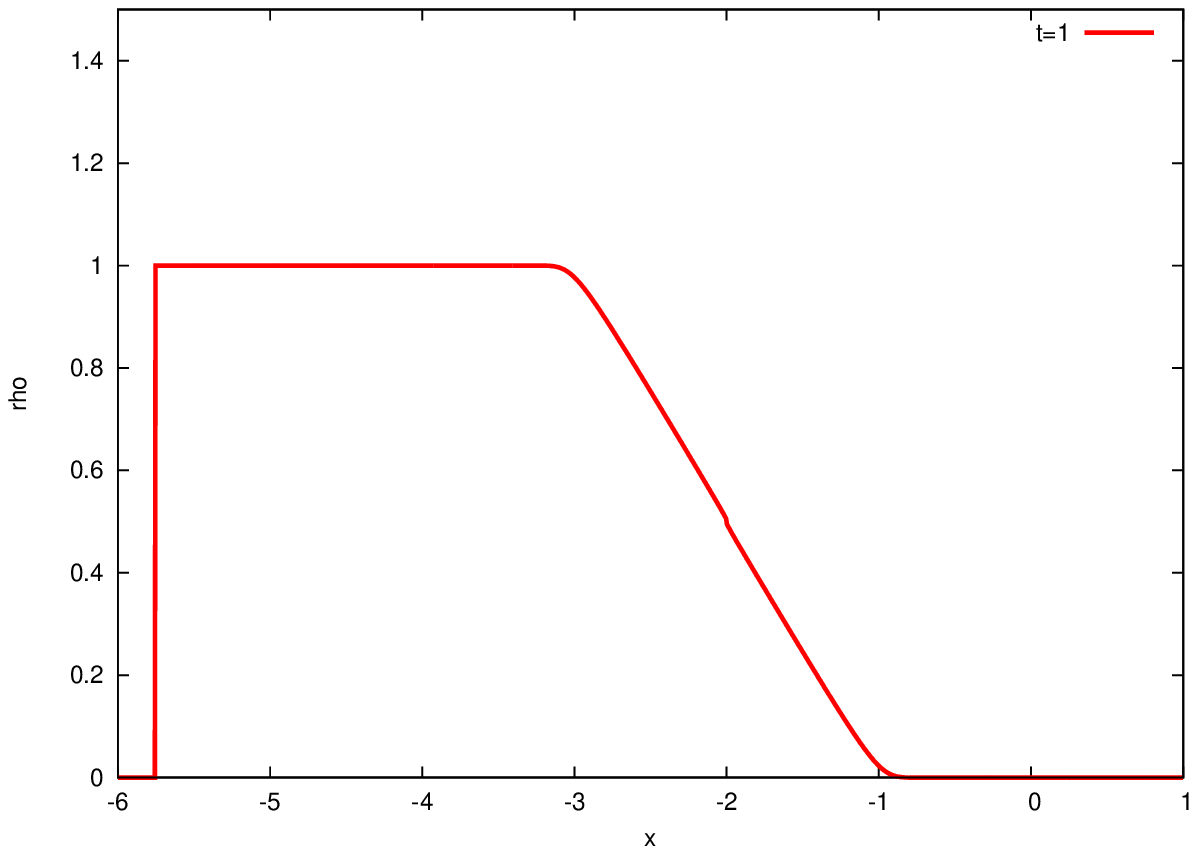}
\end{psfrags}}&
{\begin{psfrags}
\psfrag{t=1}[r][][0.4]{$\displaystyle\vphantom{\int^{\int^{\int}}}\rho(1,x)$}
\psfrag{rho}[][][0.8]{}
\psfrag{x}[][][0.8]{$x$}
\includegraphics[width=\hsize]{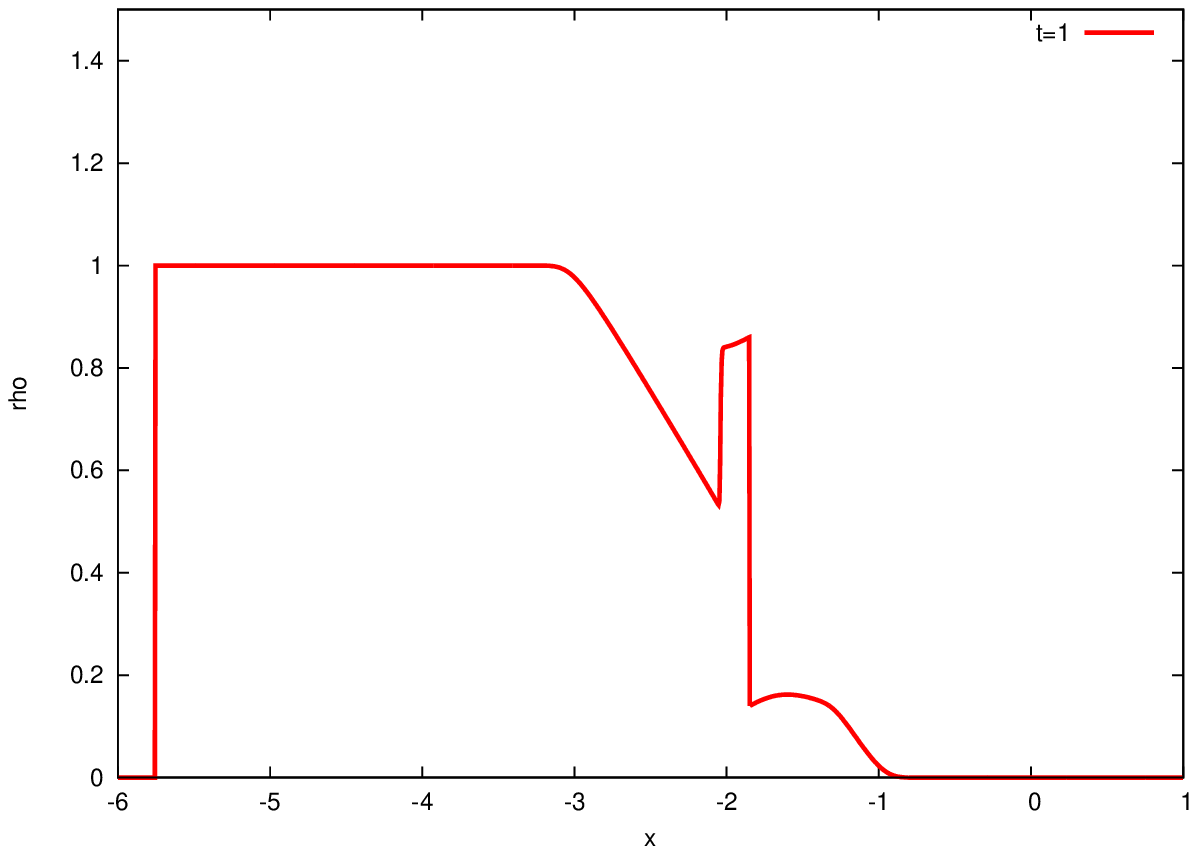}
\end{psfrags}}&
{\begin{psfrags}
\psfrag{t=1}[r][][0.4]{$\displaystyle\vphantom{\int^{\int^{\int}}}\rho(1,x)$}
\psfrag{rho}[][][0.8]{}
\psfrag{x}[][][0.8]{$x$}
\includegraphics[width=\hsize]{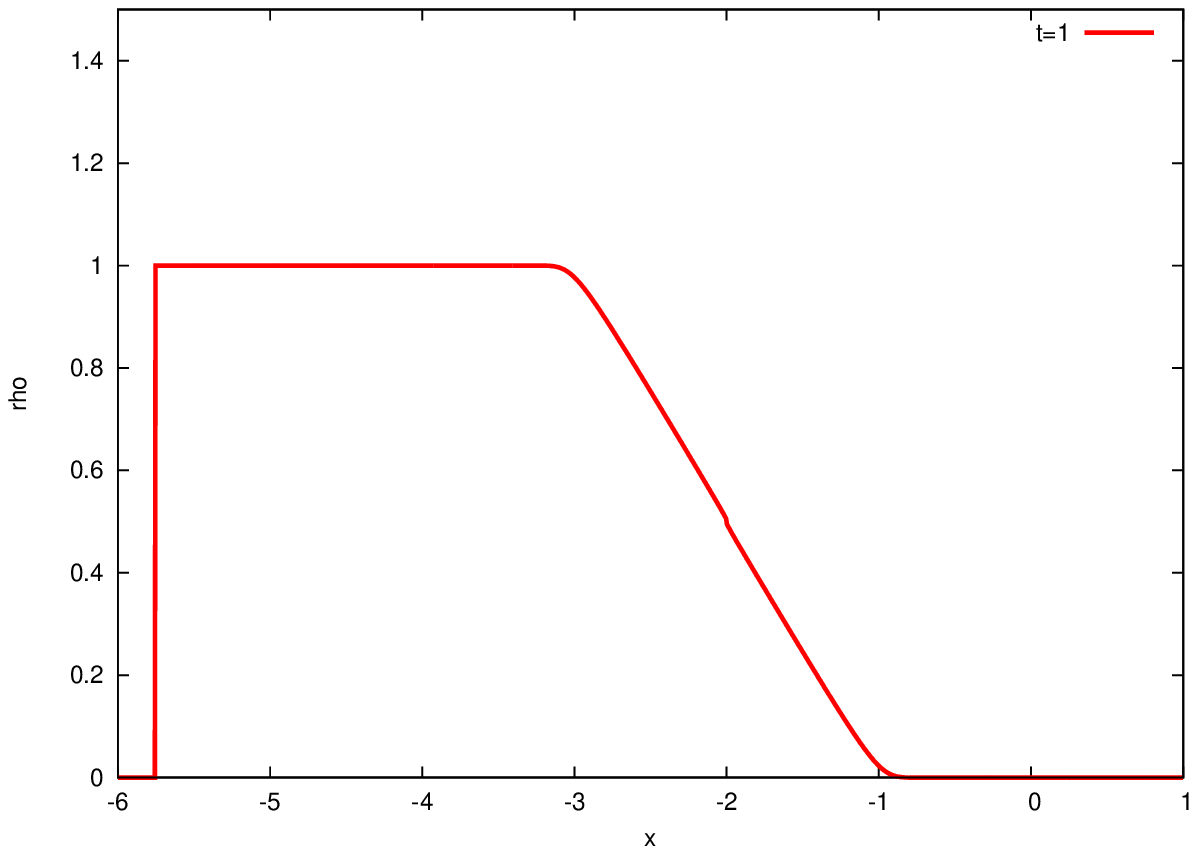}
\end{psfrags}}
\\
\rotatebox{90}{$t=7$}&
{\begin{psfrags}
\psfrag{t=7}[r][][0.4]{$\displaystyle\vphantom{\int^{\int^{\int}}}\rho(7,x)$}
\psfrag{rho}[][][0.8]{}
\psfrag{x}[][][0.8]{$x$}
\includegraphics[width=\hsize]{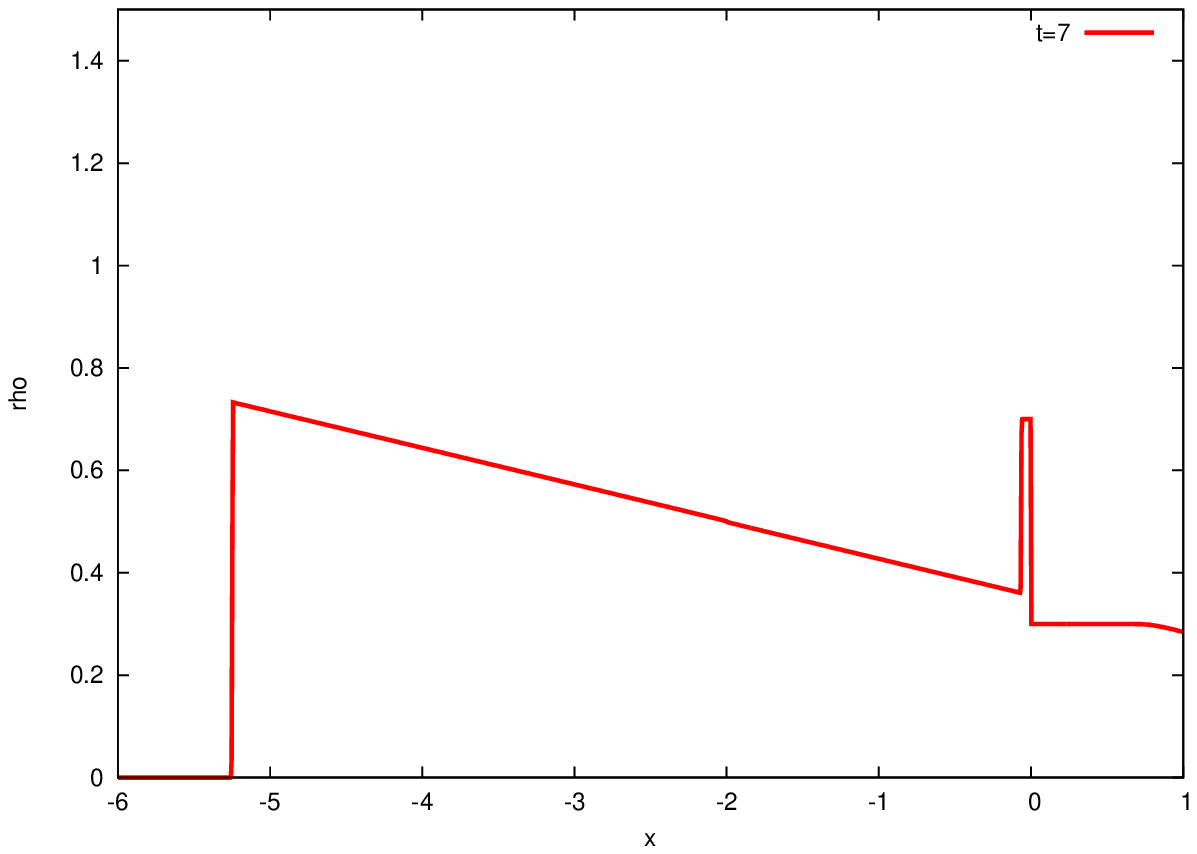}
\end{psfrags}}&
{\begin{psfrags}
\psfrag{t=7}[r][][0.4]{$\displaystyle\vphantom{\int^{\int^{\int}}}\rho(7,x)$}
\psfrag{rho}[][][0.8]{}
\psfrag{x}[][][0.8]{$x$}
\includegraphics[width=\hsize]{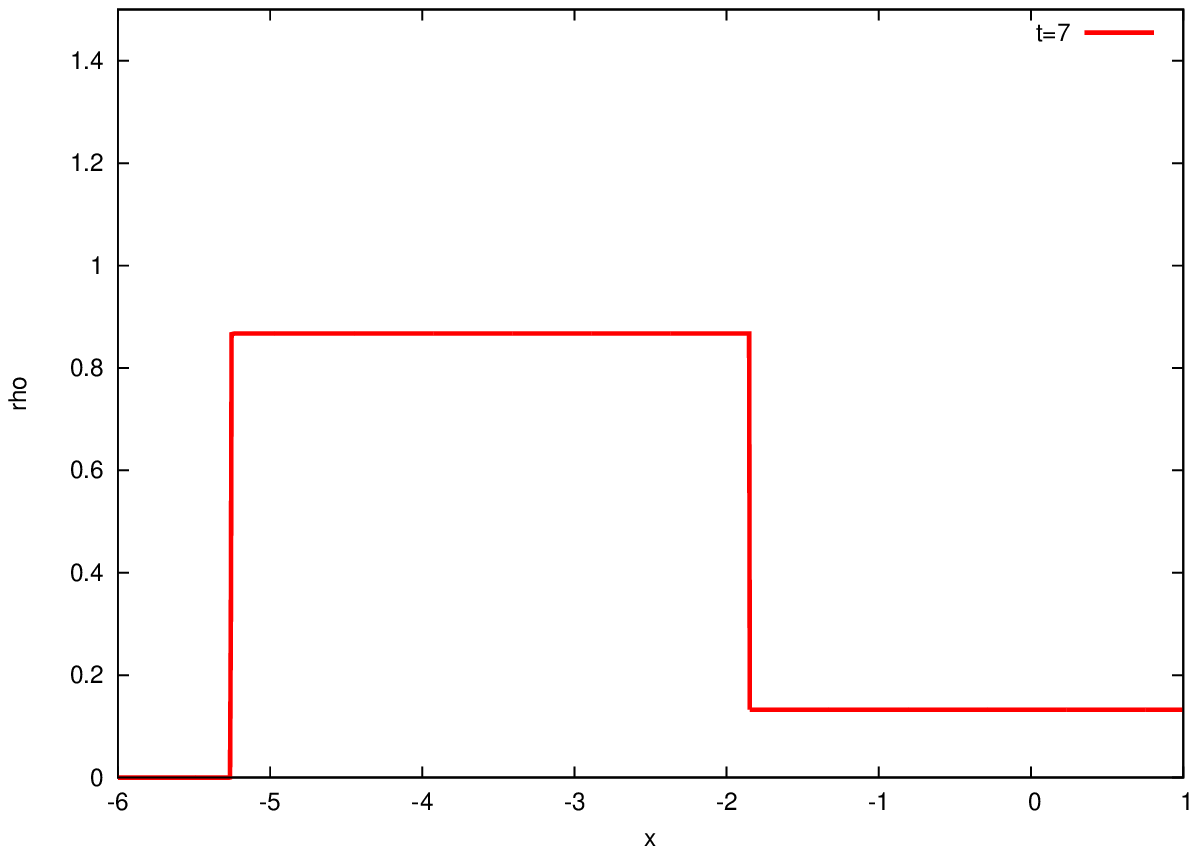}
\end{psfrags}}&
{\begin{psfrags}
\psfrag{t=7}[r][][0.4]{$\displaystyle\vphantom{\int^{\int^{\int}}}\rho(7,x)$}
\psfrag{rho}[][][0.8]{}
\psfrag{x}[][][0.8]{$x$}
\includegraphics[width=\hsize]{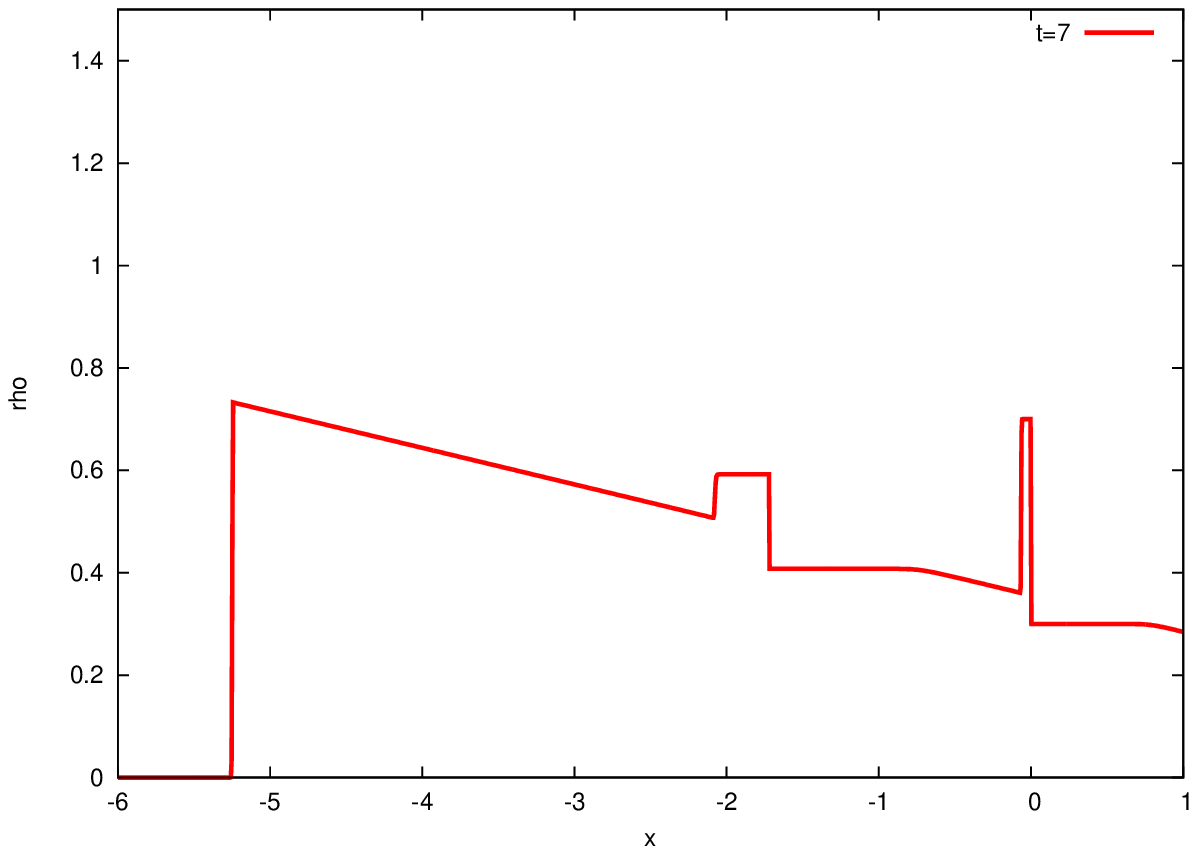}
\end{psfrags}}
\\
\rotatebox{90}{$t=15$}&
{\begin{psfrags}
\psfrag{t=15}[r][][0.4]{$\displaystyle\vphantom{\int^{\int^{\int}}}\rho(15,x)$}
\psfrag{rho}[][][0.8]{}
\psfrag{x}[][][0.8]{$x$}
\includegraphics[width=\hsize]{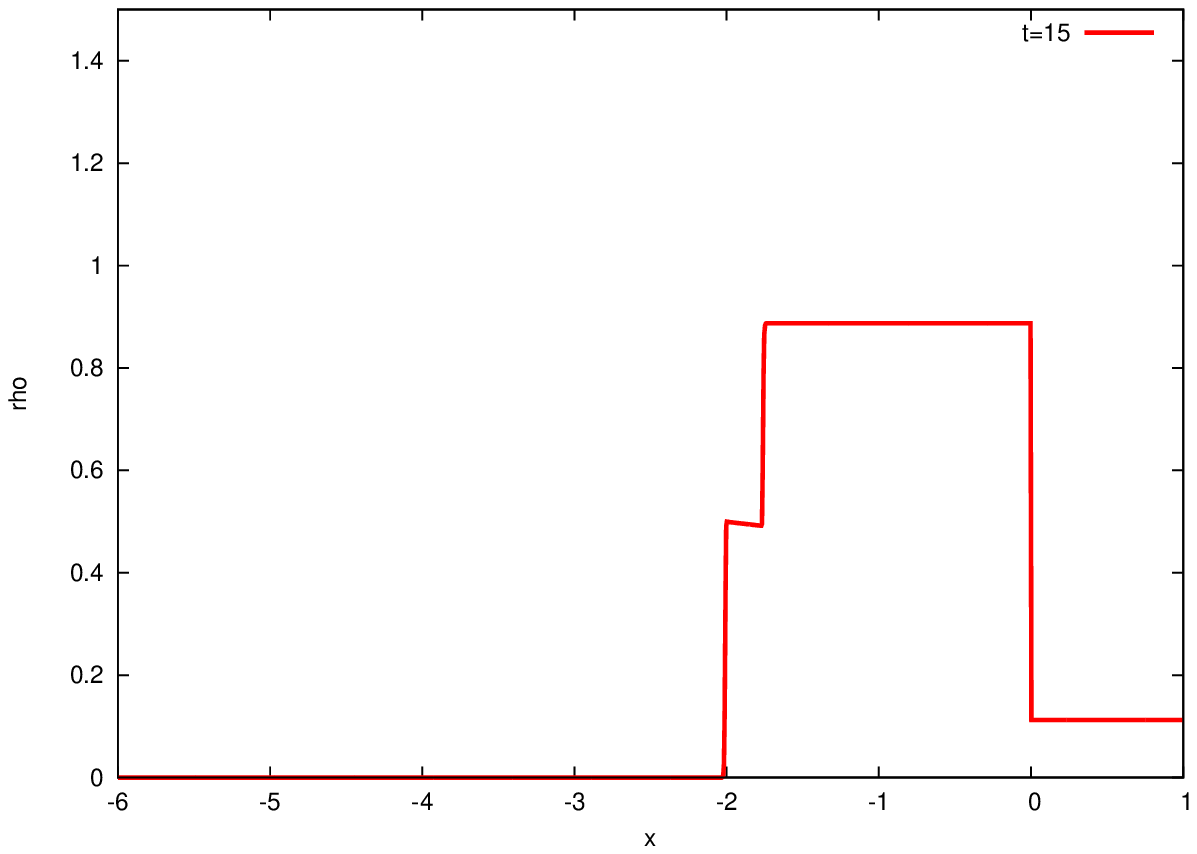}
\end{psfrags}}&
{\begin{psfrags}
\psfrag{t=15}[r][][0.4]{$\displaystyle\vphantom{\int^{\int^{\int}}}\rho(15,x)$}
\psfrag{rho}[][][0.8]{}
\psfrag{x}[][][0.8]{$x$}
\includegraphics[width=\hsize]{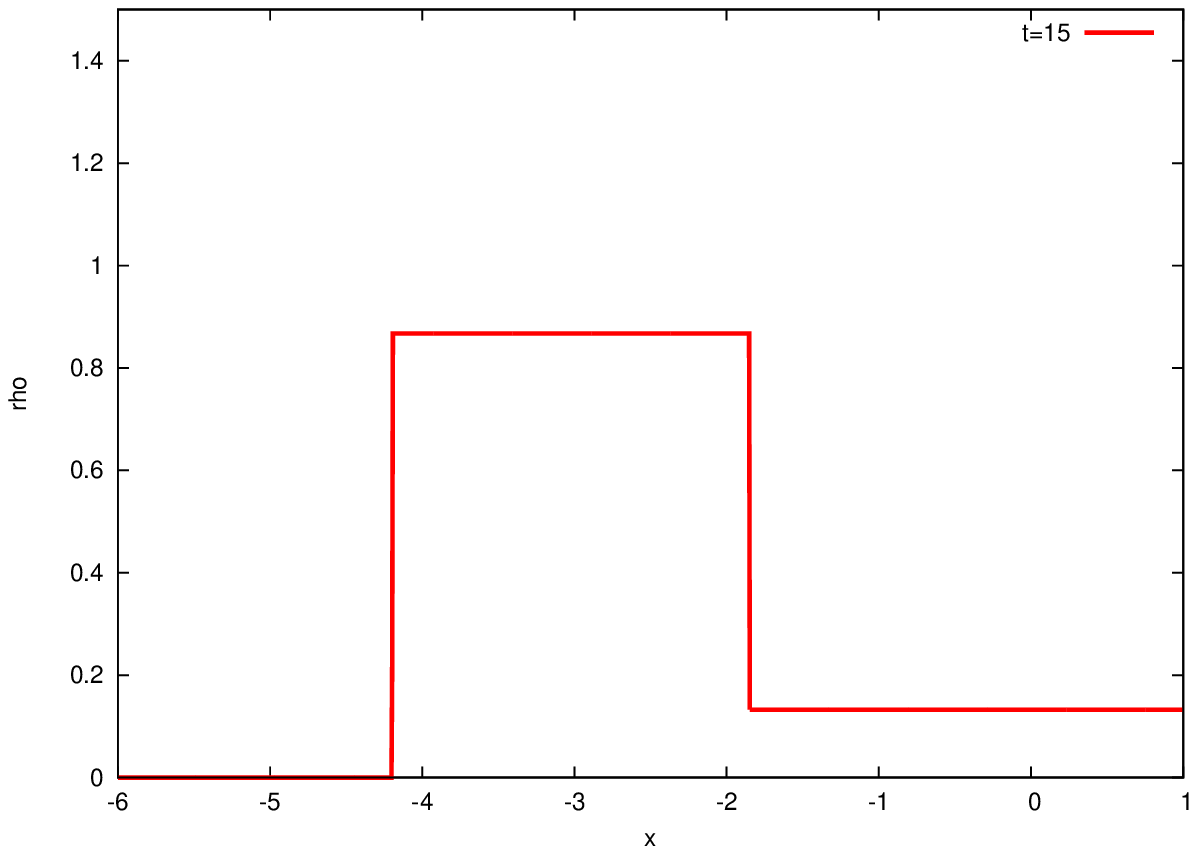}
\end{psfrags}}&
{\begin{psfrags}
\psfrag{t=15}[r][][0.4]{$\displaystyle\vphantom{\int^{\int^{\int}}}\rho(15,x)$}
\psfrag{rho}[][][0.8]{}
\psfrag{x}[][][0.8]{$x$}
\includegraphics[width=\hsize]{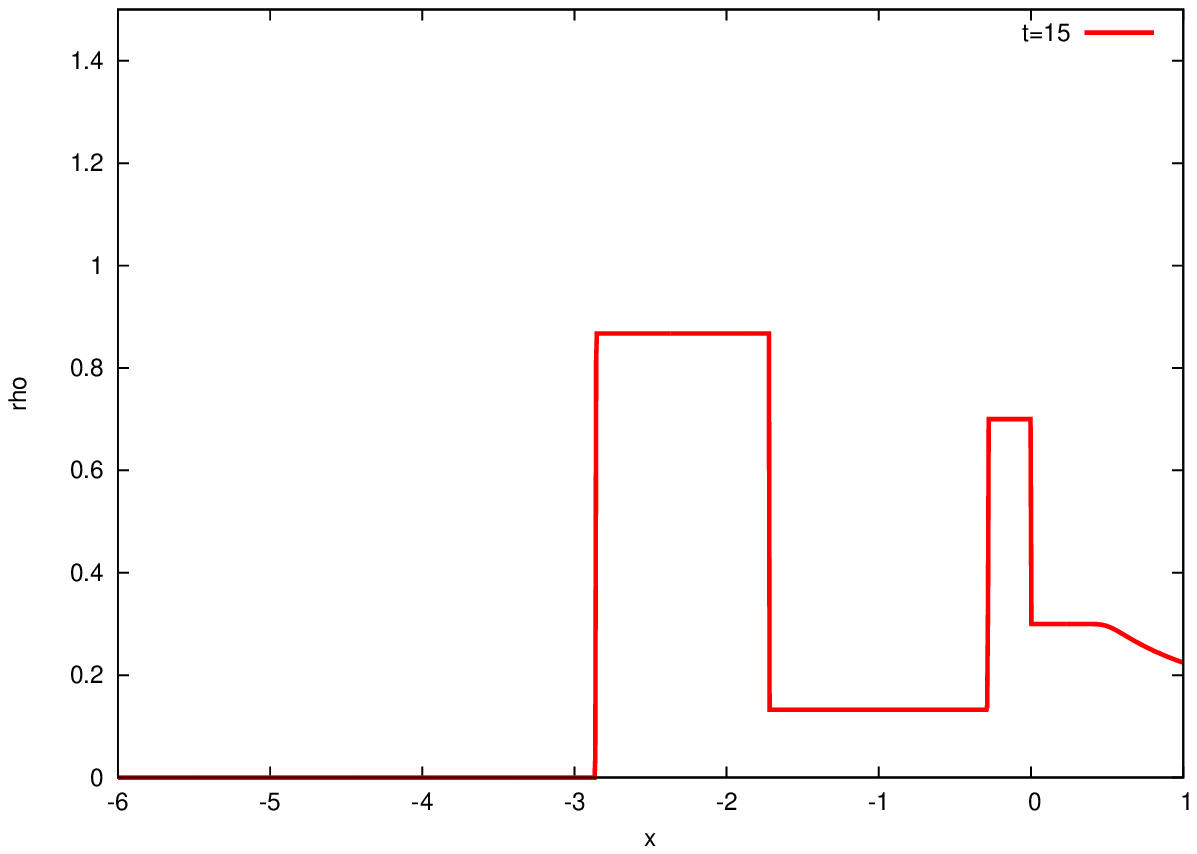}
\end{psfrags}}
\\
\rotatebox{90}{$t=19$}&
{\begin{psfrags}
\psfrag{t=19}[r][][0.4]{$\displaystyle\vphantom{\int^{\int^{\int}}}\rho(19,x)$}
\psfrag{rho}[][][0.8]{}
\psfrag{x}[][][0.8]{$x$}
\includegraphics[width=\hsize]{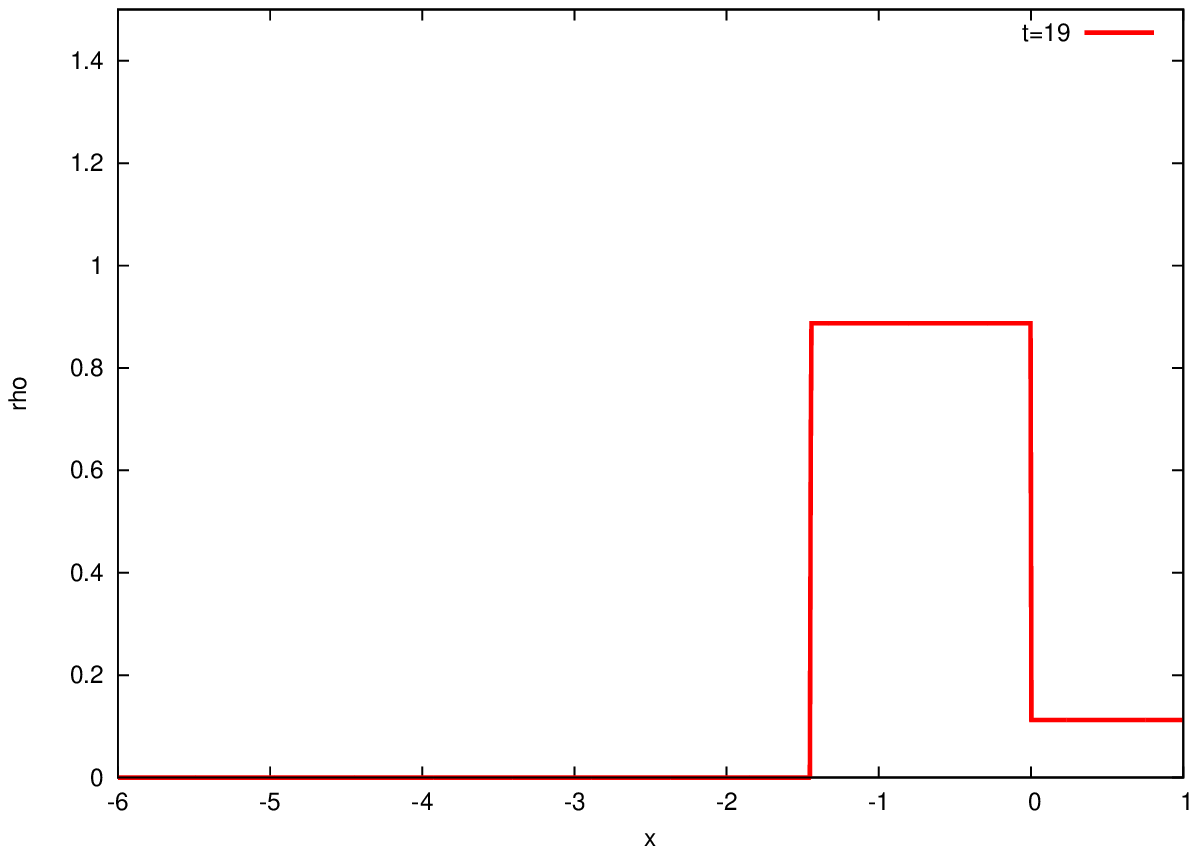}
\end{psfrags}}&
{\begin{psfrags}
\psfrag{t=19}[r][][0.4]{$\displaystyle\vphantom{\int^{\int^{\int}}}\rho(19,x)$}
\psfrag{rho}[][][0.8]{}
\psfrag{x}[][][0.8]{$x$}
\includegraphics[width=\hsize]{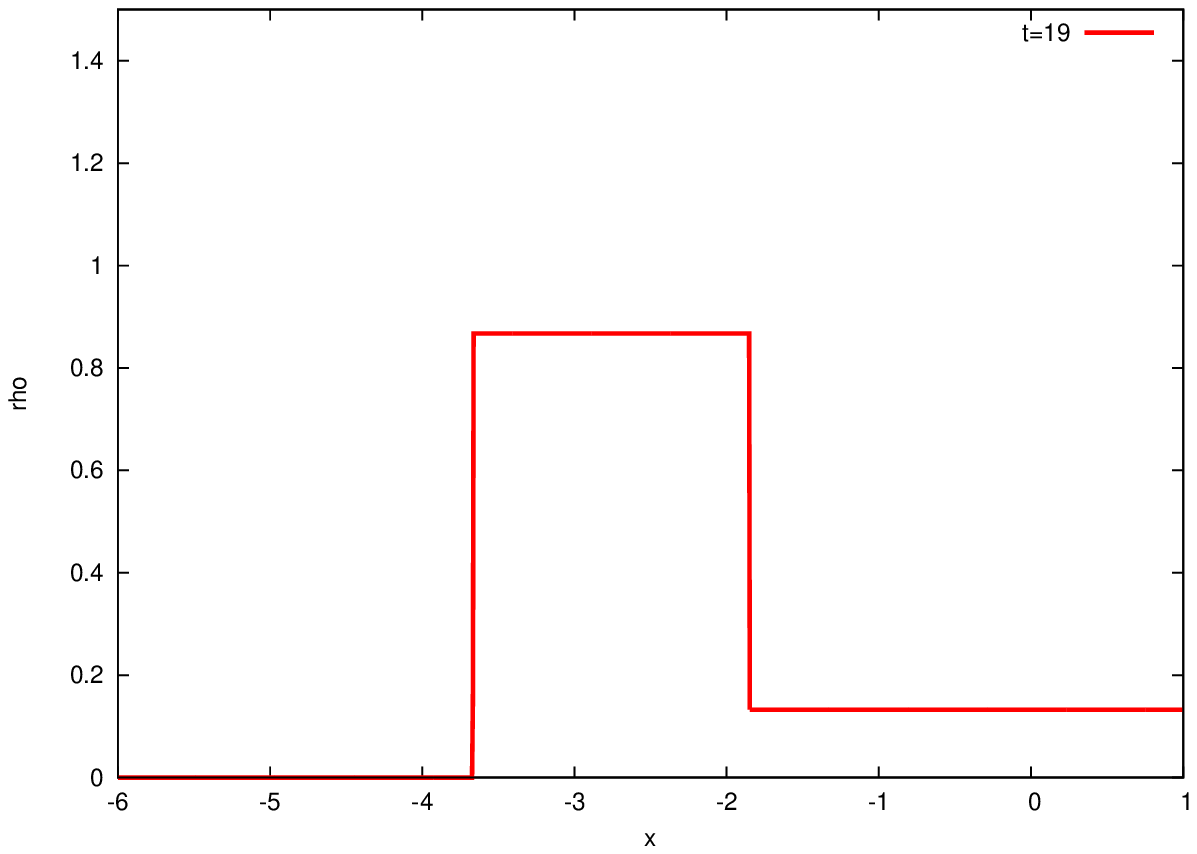}
\end{psfrags}}&
{\begin{psfrags}
\psfrag{t=19}[r][][0.4]{$\displaystyle\vphantom{\int^{\int^{\int}}}\rho(19,x)$}
\psfrag{rho}[][][0.8]{}
\psfrag{x}[][][0.8]{$x$}
\includegraphics[width=\hsize]{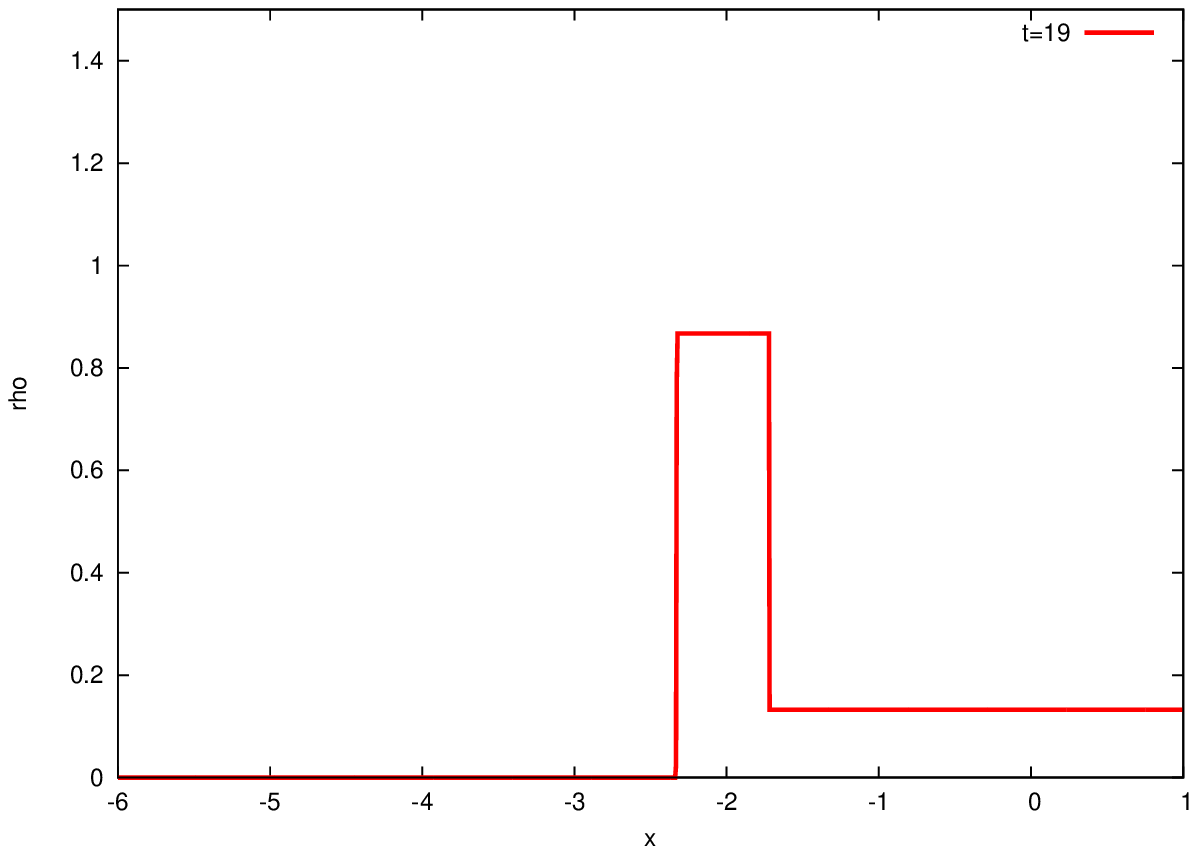}
\end{psfrags}}
\\
\rotatebox{90}{$t=24.246$}&
{\begin{psfrags}
\psfrag{final}[r][][0.4]{$\displaystyle\vphantom{\int^{\int^{\int}}}\rho(24.246,x)$}
\psfrag{rho}[][][0.8]{}
\psfrag{x}[][][0.8]{$x$}
\includegraphics[width=\hsize]{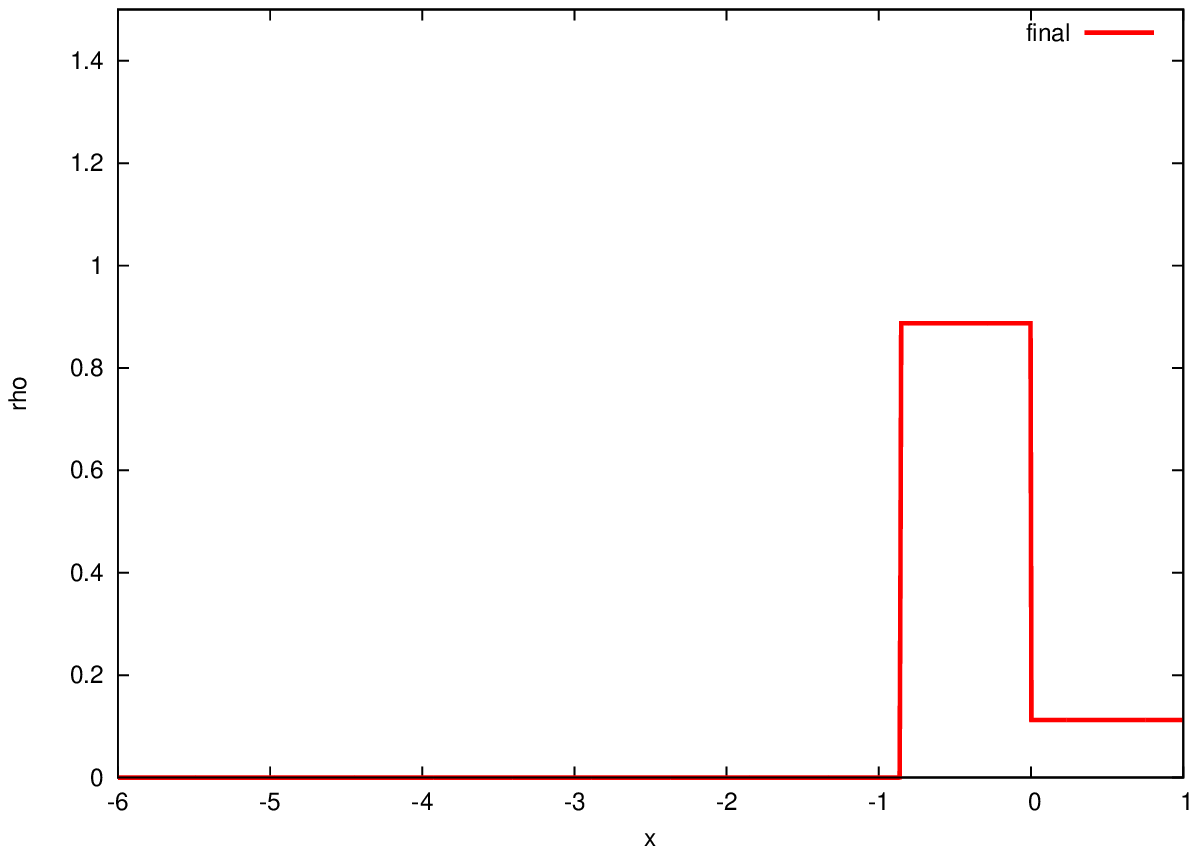}
\end{psfrags}}&
{\begin{psfrags}
\psfrag{final}[r][][0.4]{$\displaystyle\vphantom{\int^{\int^{\int}}}\rho(24.246,x)$}
\psfrag{rho}[][][0.8]{}
\psfrag{x}[][][0.8]{$x$}
\includegraphics[width=\hsize]{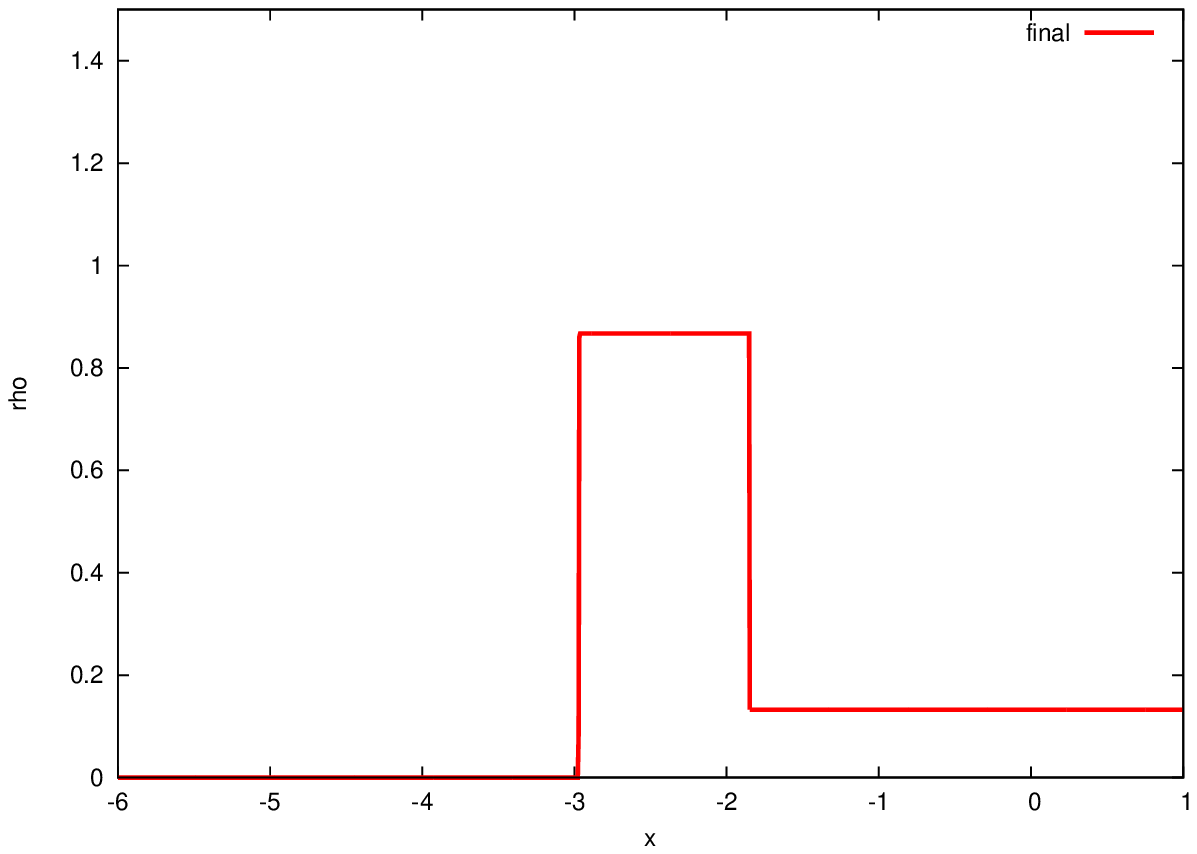}
\end{psfrags}}&
{\begin{psfrags}
\psfrag{final}[r][][0.4]{$\displaystyle\vphantom{\int^{\int^{\int}}}\rho(24.246,x)$}
\psfrag{rho}[][][0.8]{}
\psfrag{x}[][][0.8]{$x$}
\includegraphics[width=\hsize]{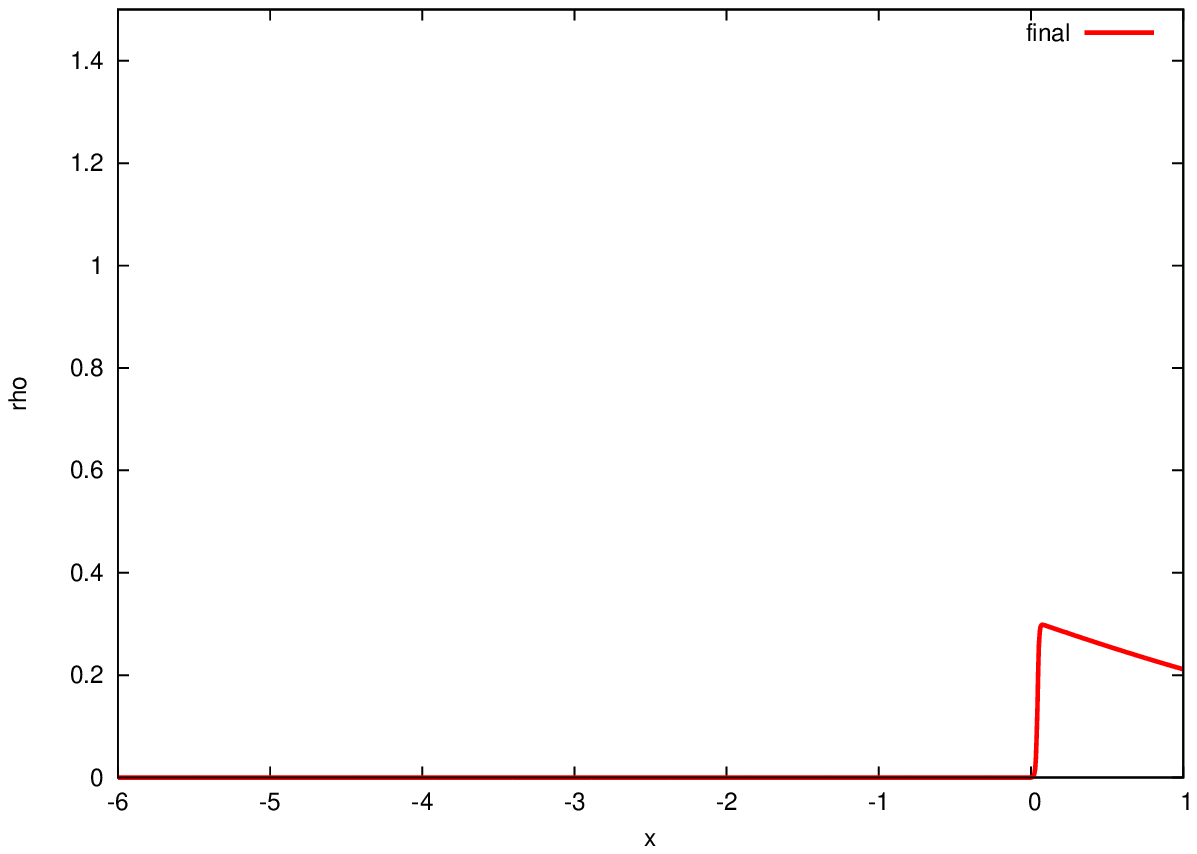}
\end{psfrags}}
  \end{tabular}
\caption{With reference to Subsection \ref{sec:Braess}: Braess paradox simulations: density profiles at times $t=1$ (first line), $t=7$ (second line), $t=15$ (third line), $t=19$  (fourth line) and $t=24.246$ (last line).}
\label{Braess_snapshots}
\end{figure}

Then we performed some series of tests to see how the general shape obtained in Figure~\ref{fis_exit_times_velocities} changes with respect to variations of the parameters of the model. 
 In Figure~\ref{fis_stabilities}~(a), we show this variation when we consider different initial densities, namely, $\bar\rho$, $\bar{\rho}_1$ and $\bar{\rho}_2$ with $\bar\rho_1(x)=0.8\chi_{[-5.75,-2]}$ and $\bar\rho_2(x)=0.6\chi_{[-5.75,-2]}$. The general shape of the curves is conserved. We observe that the evacuation time increases with the initial amount of pedestrians while the optimal velocity decreases as the initial amount of pedestrians increases. The minimal evacuation time and the corresponding optimal maximal velocity are $12.259$ and $1.07$ for $\bar{\rho}_2$ and $15.691$ and $1.03$ for $\bar{\rho}_1$.
 
Next we explore the case where the efficiency of the exit varies. We consider the function $p$ defined in~\eqref{P_continuous} and the modification $p_\beta$ such that $p_\beta(\xi)=p(\beta\xi)$. In Figure~\ref{p_beta}, we plotted the functions $p$, $p_\beta$ for $\beta=0.8$ and $\beta=0.9$. 
Then, in Figure~\ref{fis_stabilities}~(b) are plotted the evacuation time curves corresponding to these three efficiencies of the exit. 
As minimum evacuation times, we obtain $18.586$ and $18.827$ for $\beta=0.8$, $0.9$ respectively. As expected, the minimal evacuation time increases with lower efficiency of the exit. The corresponding velocities are approximatively $1.06$ and $1.02$ respectively.

Finally, we change the location of the initial density. In addition to the corridor $[-6,1]$, we consider two other corridors modeled by the segments $[-12,1]$ and $[-20,1]$. In these two corridors we take as initial densities $\bar{\rho}_3(x)=\chi_{[-11.75,-8]}$ and $\bar{\rho}_4(x)=\chi_{[-19.75,-16]}$ respectively. We have reported the obtained evacuation time curves in Figure~\ref{fis_stabilities}~(c). As expected, the minimal evacuation time increases with the distance between the exit and the initial density location.     

\subsection{Braess' paradox}\label{sec:Braess}

The presence of obstacles, such as columns upstream from the exit, may prevent the crowd density from reaching dangerous values and may actually help to minimize the evacuation time, since in a moderate density regime the full capacity of the exit can be exploited. 
From a microscopic point of view, the decrease of the evacuation time may seem unexpected, as some of the pedestrians are forced to chose a longer path to reach the exit. 

The ADR model is able to reproduce the Braess' paradox for pedestrians, as we show in the following simulations. 
We consider, as in the previous subsection, the corridor modeled by the segment $[-6,1]$ with an exit at $x=0$.
We compute the solution corresponding to the flux $f(\rho)=\rho(1-\rho)$, the initial density $\bar\rho(x)=\chi_{[-5.75,-2]}(x)$, the efficiency of the exit $p$ of the form~\eqref{P_continuous} with the parameters 
\begin{align*}
    &p_0 =0.21, && p_1 =0.1, && \xi_1=0.566, && \xi_2=0.731
\end{align*}
and the same weight function $w(x)=2(1+x)\chi_{[-1,0]}(x)$. The space and time steps are fixed to $\Delta x=5\times10^{-3}$ and $\Delta t=5\times10^{-4}$. Without any obstacle, the numerical evacuation time is $29.496$. In these following simulations we place an obstacle at $x=d$, with $-2<d<0$.
The obstacle reduces the capacity of the corridor and can be seen as a door, which we assume larger than the one at $x=0$. Following these ideas we define an efficiency function $p_d(\xi)=1.15 p(\xi)$ and a weight function $w_d(x)=2(x-d+1)\chi_{[d-1,d]}(x)$ associated to the obstacle. 

In Figure~\ref{braess_exit_times} we have reported the evolution of the evacuation time when the position of the obstacle varies in the interval $[-1.9,-0.01]$ with a step of 0.01. We observe that for $-1.8\le d\le-1.72$, the evacuation time is lower than in the absence of the obstacle. The optimal position of the obstacle is obtained for $d=-1.72$ and the corresponding evacuation time is $24.246$. We compare in Figure~\ref{Braess_snapshots} five snapshots of the solution without obstacle and the solutions with an obstacle placed at $d=-1.72$ and $d=-1.85$. This latter location corresponds to a case where the evacuation time is greater than the one without an obstacle. In these snapshots, we see that the obstacle placed at $d=-1.85$ becomes congested very soon. This is due to the fact that the obstacle is too close to the location of the initial density. When the obstacle is placed at $d=-1.72$, it delays the congestion at the exit.
  
\subsection{Zone of low velocity}\label{sec:Braess+FIS}
  
\begin{figure}
\centering
\begin{subfigure}[Evacuation time as a function of $\lambda$.]
{\begin{psfrags}
\psfrag{t}[][][0.8]{evacuation time}
\psfrag{lambda}[][][0.8]{$\lambda$}
\includegraphics[width=0.48\hsize]{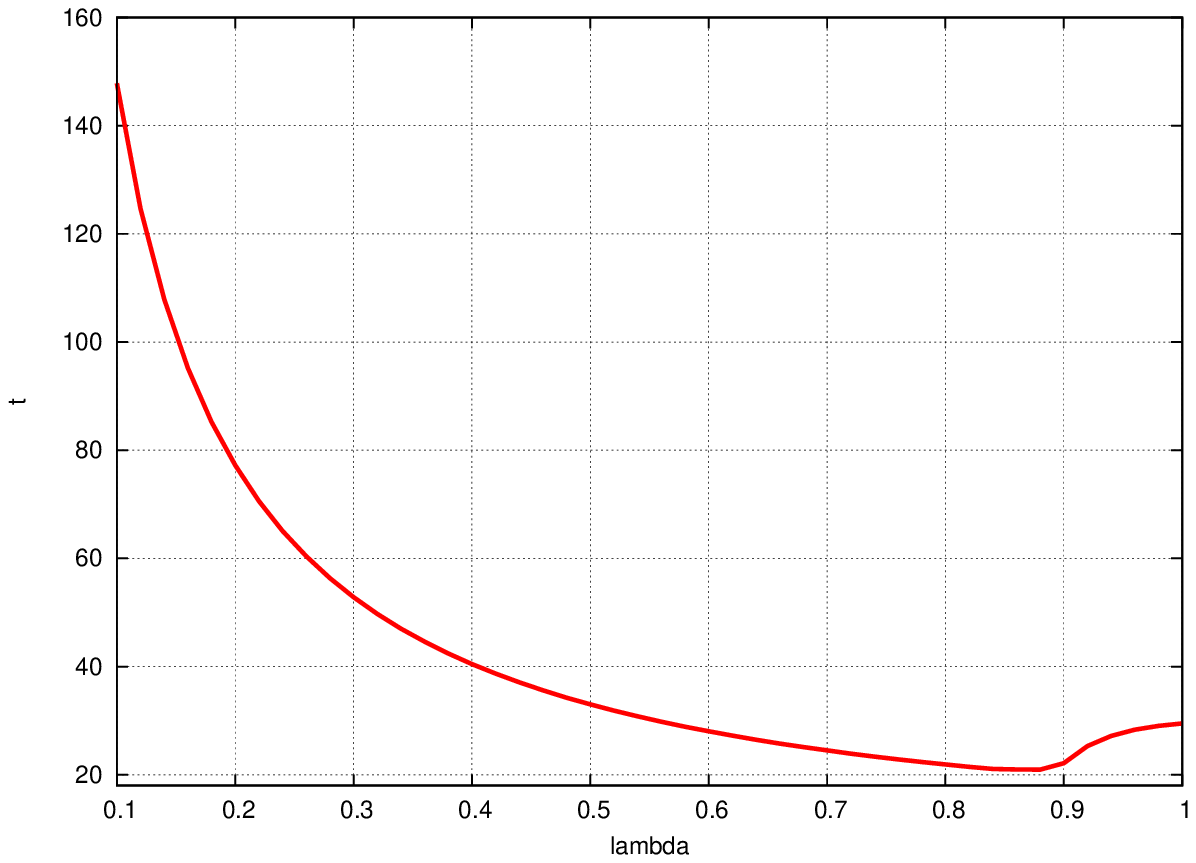}
\end{psfrags}}
\end{subfigure}~
\begin{subfigure}[Evacuation time as a function of $d$.]
{\begin{psfrags}
\psfrag{t}[][][0.8]{evacuation time}
\psfrag{d}[][][0.8]{$d$}
\includegraphics[width=0.48\hsize]{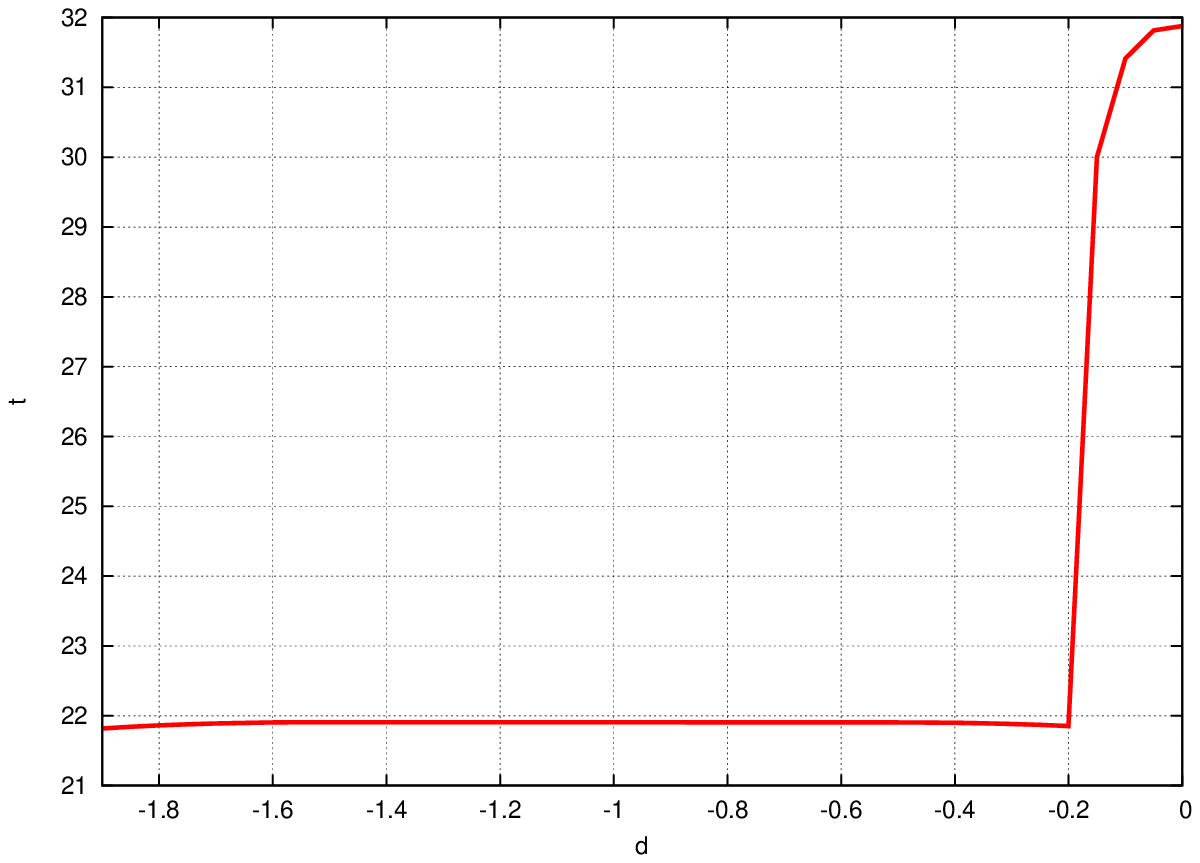}
\end{psfrags}}
\end{subfigure}

\begin{subfigure}[Evacuation time as a function of $v_{\max}$.]
{\begin{psfrags}
\psfrag{t}[][][0.8]{evacuation time}
\psfrag{v}[][][0.8]{$v_{\max}$}
\includegraphics[width=0.48\hsize]{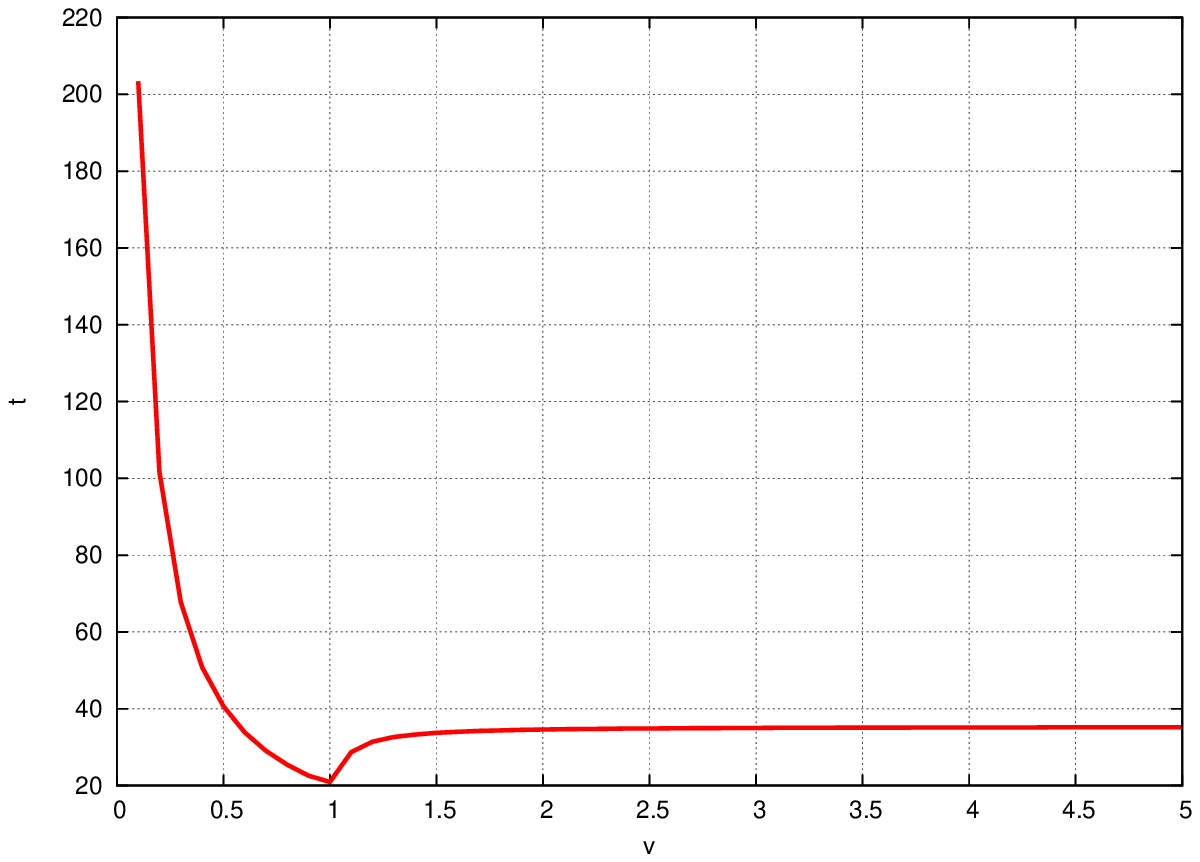}
\end{psfrags}}
\end{subfigure}
\tiny \caption {With reference to Subection \ref{sec:Braess+FIS}: Evacuation time as a function of different parameters of the model.}
\label{slowzone}
\end{figure}

\noindent In this section, we perform a series of simulations where the obstacle introduced in Subsection~\ref{sec:Braess} is now replaced by a zone where the velocity of pedestrians is lower than elsewhere in the domain.
The effect we want to observe here is similar to the one we see in Braess' Paradox. Namely we prevent an high concentration of pedestrians in front of the exit by constraining their flow in an upstream portion of the corridor. In this case however the constraint is local, as the maximal value allowed for the flow only depends on the position in the corridor. 

\begin{figure}
\centering
\renewcommand{\arraystretch}{1.5}
\begin{tabular}{>{\centering\bfseries}m{0.05\hsize} @{}>{\centering}m{0.3\hsize} @{}>{\centering}m{0.3\hsize} @{}>{\centering\arraybackslash}m{0.3\hsize}}
& Without obstacle & Obstacle at $d=-1.72$ & Zone of low velocity centered at $d=-1.72$\\
\rotatebox{90}{$t=1$}&
{\begin{psfrags}
\psfrag{t=1}[r][][0.4]{$\displaystyle\vphantom{\int^{\int^{\int}}}\rho(1,x)$}
\psfrag{rho}[][][0.8]{}
\psfrag{x}[][][0.8]{$x$}
\includegraphics[width=\hsize]{SO_1.eps}
\end{psfrags}}&
{\begin{psfrags}
\psfrag{t=1}[r][][0.4]{$\displaystyle\vphantom{\int^{\int^{\int}}}\rho(1,x)$}
\psfrag{rho}[][][0.8]{}
\psfrag{x}[][][0.8]{$x$}
\includegraphics[width=\hsize]{d172_1.eps}
\end{psfrags}}&
{\begin{psfrags}
\psfrag{t=1}[r][][0.4]{$\displaystyle\vphantom{\int^{\int^{\int}}}\rho(1,x)$}
\psfrag{rho}[][][0.8]{}
\psfrag{x}[][][0.8]{$x$}
\includegraphics[width=\hsize]{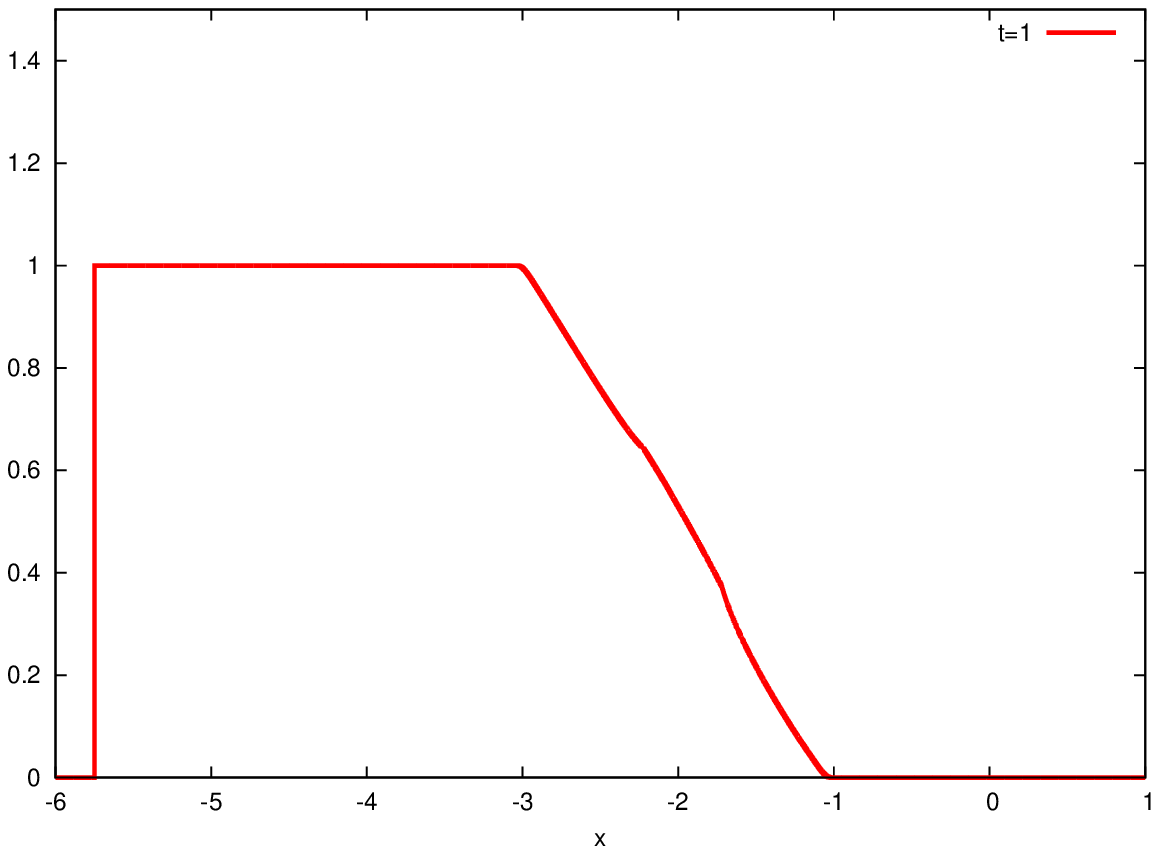}
\end{psfrags}}
\\
\rotatebox{90}{$t=7$}&
{\begin{psfrags}
\psfrag{t=7}[r][][0.4]{$\displaystyle\vphantom{\int^{\int^{\int}}}\rho(7,x)$}
\psfrag{rho}[][][0.8]{}
\psfrag{x}[][][0.8]{$x$}
\includegraphics[width=\hsize]{SO_7.eps}
\end{psfrags}}&
{\begin{psfrags}
\psfrag{t=7}[r][][0.4]{$\displaystyle\vphantom{\int^{\int^{\int}}}\rho(7,x)$}
\psfrag{rho}[][][0.8]{}
\psfrag{x}[][][0.8]{$x$}
\includegraphics[width=\hsize]{d172_7.eps}
\end{psfrags}}&
{\begin{psfrags}
\psfrag{t=7}[r][][0.4]{$\displaystyle\vphantom{\int^{\int^{\int}}}\rho(7,x)$}
\psfrag{rho}[][][0.8]{}
\psfrag{x}[][][0.8]{$x$}
\includegraphics[width=\hsize]{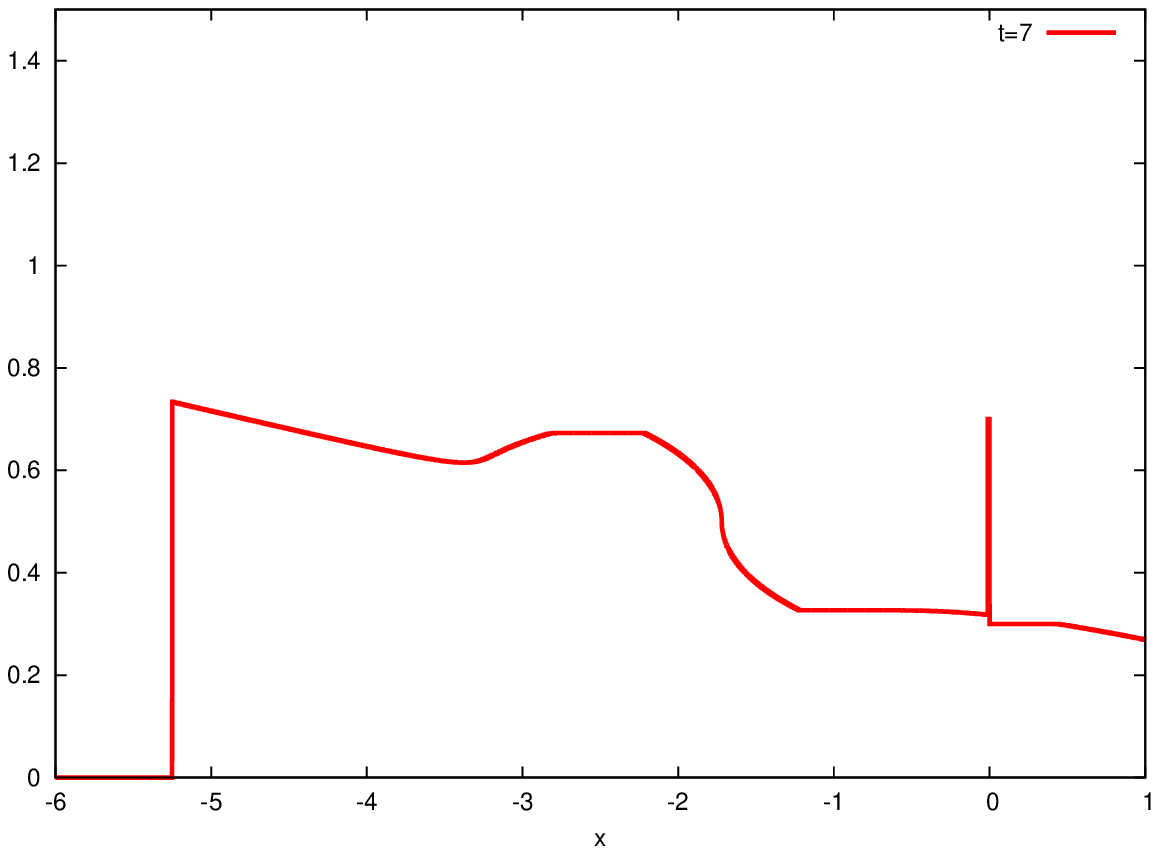}
\end{psfrags}}
\\
\rotatebox{90}{$t=15$}&
{\begin{psfrags}
\psfrag{t=15}[r][][0.4]{$\displaystyle\vphantom{\int^{\int^{\int}}}\rho(15,x)$}
\psfrag{rho}[][][0.8]{}
\psfrag{x}[][][0.8]{$x$}
\includegraphics[width=\hsize]{SO_15.eps}
\end{psfrags}}&
{\begin{psfrags}
\psfrag{t=15}[r][][0.4]{$\displaystyle\vphantom{\int^{\int^{\int}}}\rho(15,x)$}
\psfrag{rho}[][][0.8]{}
\psfrag{x}[][][0.8]{$x$}
\includegraphics[width=\hsize]{d172_15.eps}
\end{psfrags}}&
{\begin{psfrags}
\psfrag{t=15}[r][][0.4]{$\displaystyle\vphantom{\int^{\int^{\int}}}\rho(15,x)$}
\psfrag{rho}[][][0.8]{}
\psfrag{x}[][][0.8]{$x$}
\includegraphics[width=\hsize]{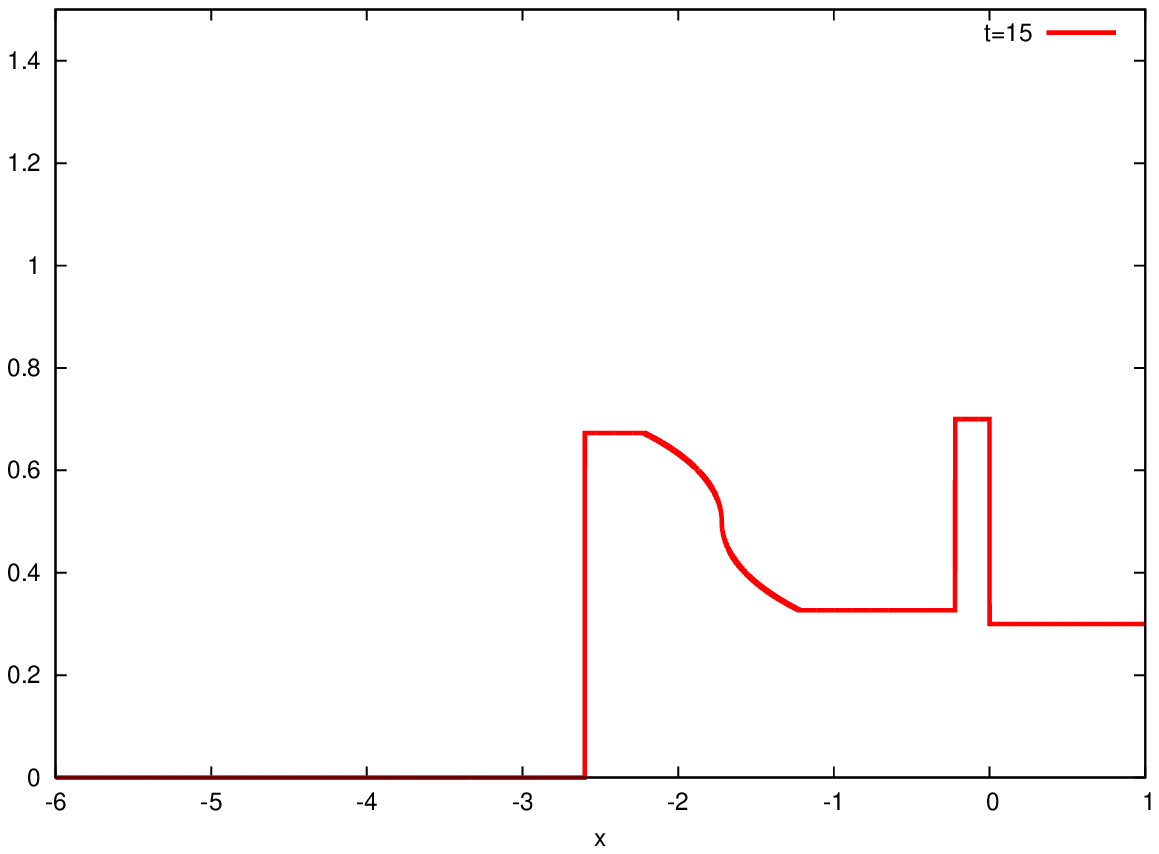}
\end{psfrags}}
\\
\rotatebox{90}{$t=19$}&
{\begin{psfrags}
\psfrag{t=19}[r][][0.4]{$\displaystyle\vphantom{\int^{\int^{\int}}}\rho(19,x)$}
\psfrag{rho}[][][0.8]{}
\psfrag{x}[][][0.8]{$x$}
\includegraphics[width=\hsize]{SO_19.eps}
\end{psfrags}}&
{\begin{psfrags}
\psfrag{t=19}[r][][0.4]{$\displaystyle\vphantom{\int^{\int^{\int}}}\rho(19,x)$}
\psfrag{rho}[][][0.8]{}
\psfrag{x}[][][0.8]{$x$}
\includegraphics[width=\hsize]{d172_19.eps}
\end{psfrags}}&
{\begin{psfrags}
\psfrag{t=19}[r][][0.4]{$\displaystyle\vphantom{\int^{\int^{\int}}}\rho(19,x)$}
\psfrag{rho}[][][0.8]{}
\psfrag{x}[][][0.8]{$x$}
\includegraphics[width=\hsize]{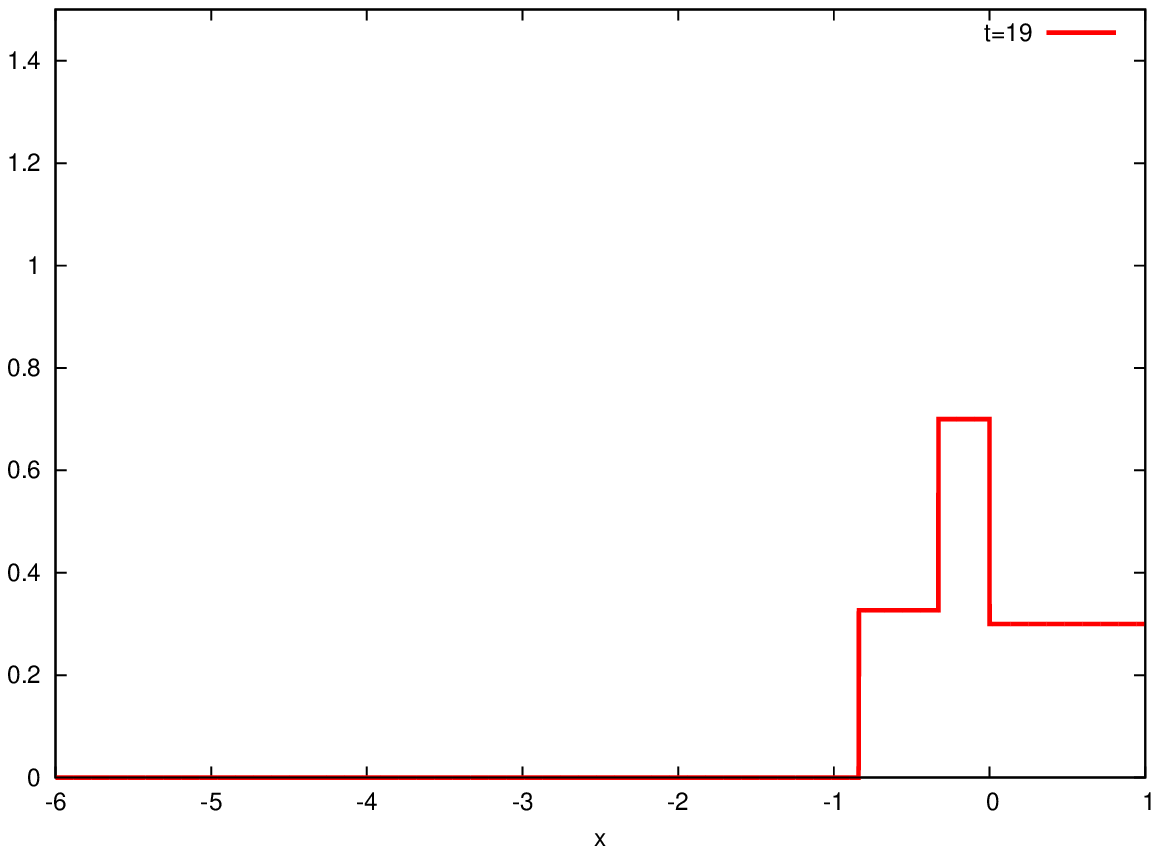}
\end{psfrags}}
\\
\rotatebox{90}{$t=20.945$}&
{\begin{psfrags}
\psfrag{t=20.945}[r][][0.4]{$\displaystyle\vphantom{\int^{\int^{\int}}}\rho(20.945,x)$}
\psfrag{rho}[][][0.8]{}
\psfrag{x}[][][0.8]{$x$}
\includegraphics[width=\hsize]{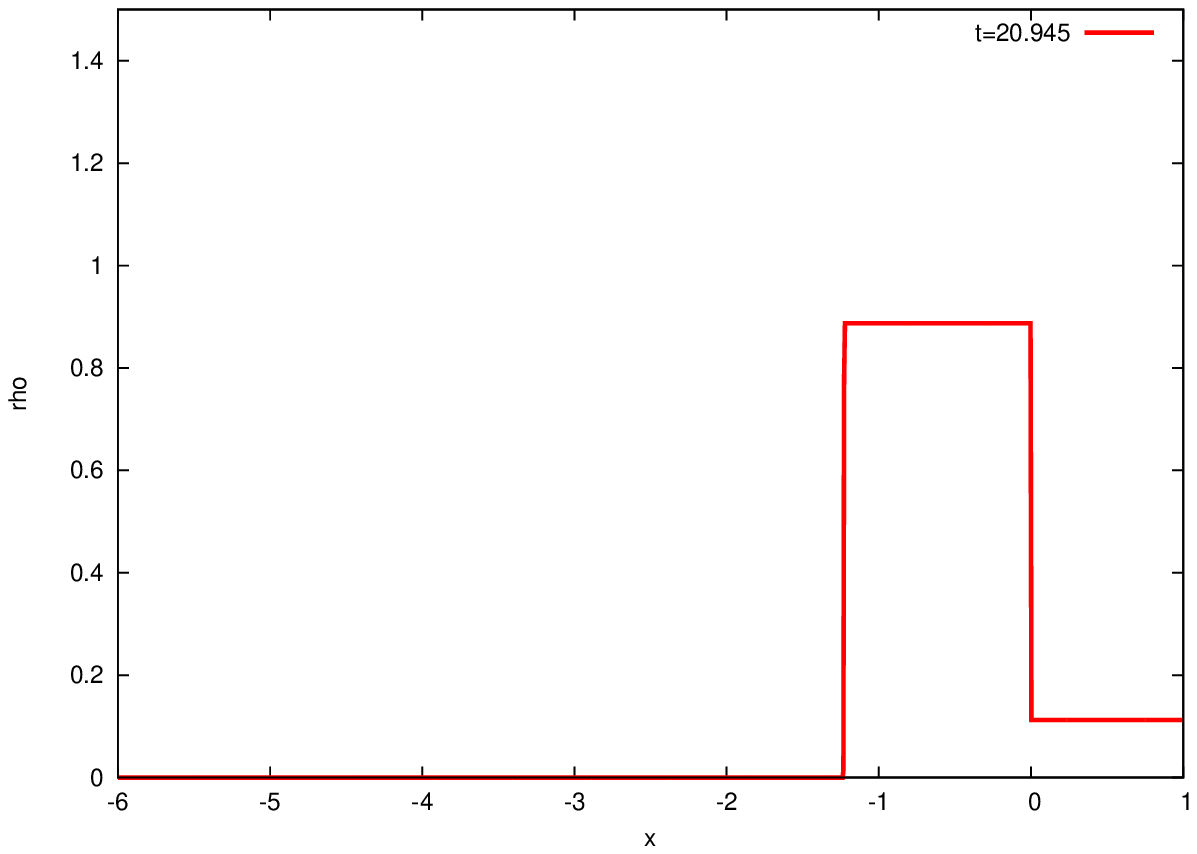}
\end{psfrags}}&
{\begin{psfrags}
\psfrag{t=20.945}[r][][0.4]{$\displaystyle\vphantom{\int^{\int^{\int}}}\rho(20.945,x)$}
\psfrag{rho}[][][0.8]{}
\psfrag{x}[][][0.8]{$x$}
\includegraphics[width=\hsize]{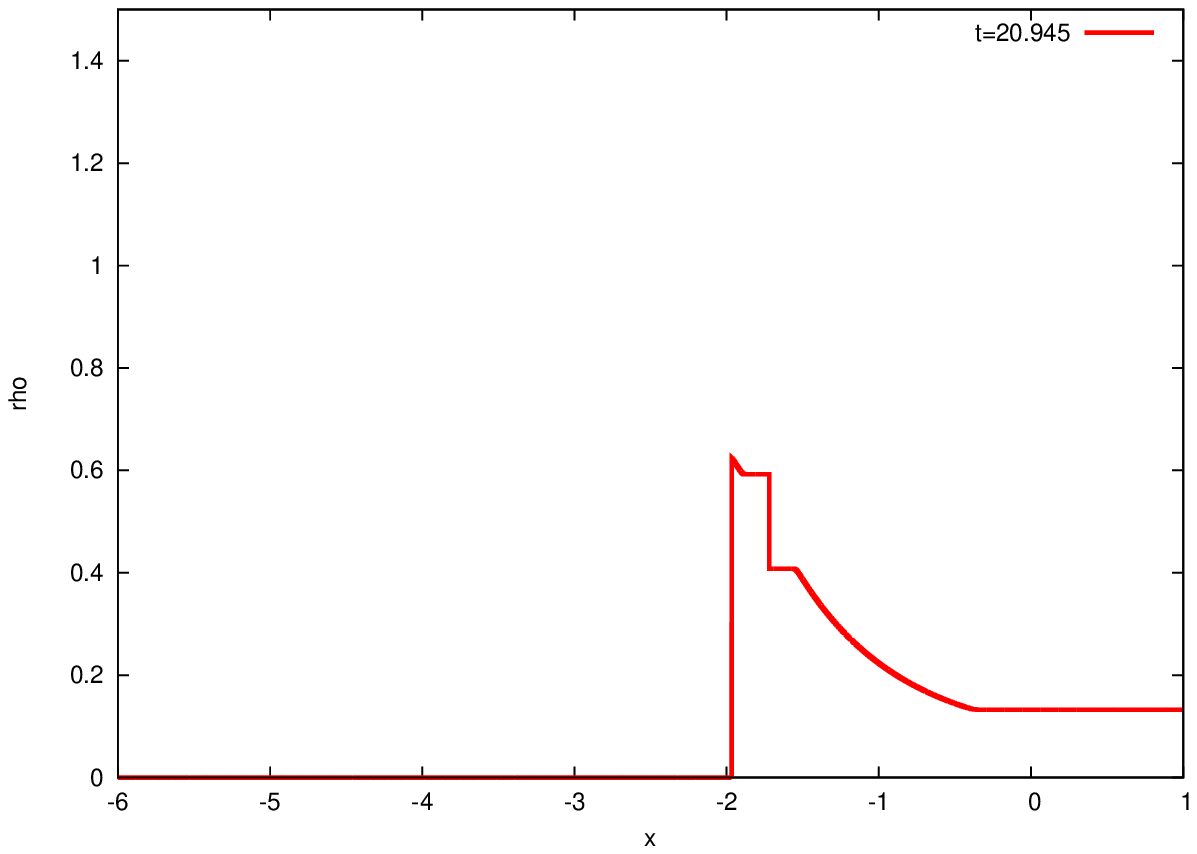}
\end{psfrags}}&
{\begin{psfrags}
\psfrag{t=20.945}[r][][0.4]{$\displaystyle\vphantom{\int^{\int^{\int}}}\rho(20.945,x)$}
\psfrag{rho}[][][0.8]{}
\psfrag{x}[][][0.8]{$x$}
\includegraphics[width=\hsize]{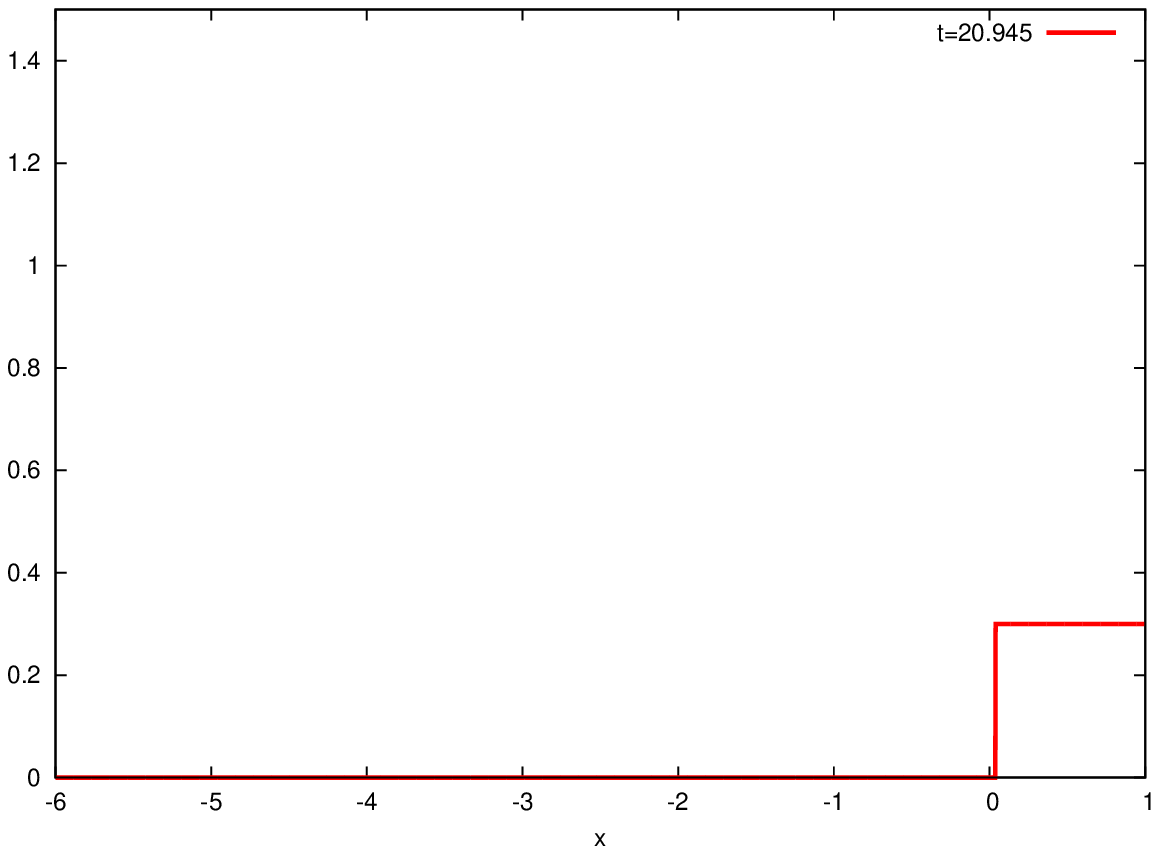}
\end{psfrags}}
  \end{tabular}
\caption{With reference to Subsection \ref{sec:Braess+FIS}: Braess' paradox and zone of low velocity simulations: density profiles at times $t=1$ (first line), $t=7$ (second line), $t=15$ (third line), $t=19$  (fourth line) and $t=20.945$ (last line).}
\label{column_snapshots}
\end{figure}
We consider again the corridor modeled by the segment $[-6,1]$ with an exit at $x=0$. The efficiency of the exit and the initial density are the same as in the previous subsection. Assume that the slow zone is of size one and is centred at $x=d$, where $-1.9\le d\le0$. Define the following function

\begin{eqnarray}
\label{def_k}
    k(x) &= \left\{
    \begin{array}{l@{\quad\text{ if }}l}
       1 & x\le d-0.5,\\[6pt]
       -2(x-d) & d-0.5\le x\le d,\\[10pt]
       2(x-d) & d\le x\le d+0.5,\\[10pt]
       1 & x\ge d+0.5,
    \end{array}
    \right.
\end{eqnarray}  
and the following velocity $v(x,\rho)= \left[\lambda+(1-\lambda) \, k(x)\right] v_{\max} \, (1-\rho)$, where $\lambda\in[0,1]$ and $v_{\max}\ge1$ is the maximal velocity. With such velocity, the maximal velocity of pedestrians decreases in the interval $[d-0.5,d]$, reaching its minimal value $\lambda \, v_{\max}$ at $x=d$. Then the velocity increases in the interval $[d,d+0.5]$ reaching the maximum value $v_{\max}$, that corresponds to the maximal velocity away from the slow zone. Finally we consider the flux $f(x,\rho)=\rho\,v(x,\rho)$ and the space and time steps are fixed to $\Delta x=5\times10^{-3}$ and $\Delta t=5\times10^{-4}$.

\noindent Figure~\ref{slowzone}~(a) shows the evolution of the evacuation time as a function of the parameter $\lambda$ varying in the interval $[0.1,1]$ when the center of the slow zone is fixed at $d=-1.5$. We observe that the optimal minimal velocity in the slow zone is for $\lambda=0.88$ and the corresponding evacuation time is $20.945$. Recalling that without the slow zone the evacuation time is $29.496$, we see that the introduction of the slow zone allows to reduce the evacuation time. In Figure~\ref{slowzone}~(b), we show the evolution of the evacuation time when varying the center of the slow zone $d$ in the interval $[-1.9,0]$ and when the minimal and the maximal velocities are fixed and correspond to $\lambda=0.88$ and $v_{\max}=1$. We observe here that, unlike in the Braess paradox tests case, the evacuation time does not depend on the location of the slow zone, except when this latter is close enough to the exit. 
Indeed, when the slow zone gets too close to the exit, the evacuation time grows. This is due to the fact that pedestrians do not have time to speed up before reaching the exit.

Fix now $d=-1.5$ and $\lambda=0.88$ and assume that $v_{\max}$ varies in the interval $[0.1,5]$. The evolution of the evacuation time as a function of $v_{\max}$ is reported in Figure~\ref{slowzone}~(c). We observe that we get the characteristic shape already obtained in the FIS effect.

Finally we present in Figure~\ref{column_snapshots} five snapshots for three different solutions. The first two solutions are the ones computed in Subsection~\ref{sec:Braess}, without obstacle and with an obstacle located at $d=~-1.72$ respectively. The third solution is computed with a zone of low velocity centered at $d=-1.72$, $\lambda=0.88$ and $v_{\max}=1$. In order to have a good resolution of this third solution, the space and time steps where fixed to $\Delta x=3.5\times10^{-4}$ and $\Delta t=7\times10^{-5}$.We note that in the case where a zone of low velocity is placed in the domain, we do not see the capacity drop, as the density of pedestrians never attains very high values in the region next to the exit.

\section{Conclusions}\label{sec:conclusions}

Qualitative features that are characteristic of pedestrians' macroscopic behaviour at bottlenecks (Faster is Slower, Braess' paradox) are reproduced in the setting of the simple scalar model with non-local point constraint introduced in~\cite{BorisCarlottaMax-M3AS}. These effects are shown to be persistent for large intervals of values of parameters. The validation is done by means of a simple and robust time-explicit splitting finite volume scheme which is proved to be convergent, with experimental rate close to one. 

The results presented in this paper allow to consider more complex models.  Indeed, as ADR is a first order model, it is not able to capture more complicated effects related to crowd dynamics. Typically, ADR fails to reproduce the amplification of small perturbations. This leads to consider second order model such as the model proposed by Aw, Rascle and Zhang~\cite{Aw_SIAP_2000, Zhang_Method_2002} in the framework of vehicular traffic.

Another extension of this work is to consider the ADR model with constraints that are non-local in time. Such constraints allow to tackle optimal management problems in the spirit of \cite{ColomboFacchiMaterniniRosini, CGRESAIM}.  

Finally, this work can also be extended to two-dimensional models where experimental validations may be possible. 
 
\section*{Acknowledgment}

All the authors are supported by French ANR JCJC grant CoToCoLa and Polonium 2014 (French-Polish cooperation program) No.331460NC. The first author is grateful to IRMAR, Universit\'e de Rennes, for the hospitality during the preparation of this paper. The second author is also supported by the Universit\'e de Franche-Comt\'e, soutien aux EC 2014.

Projekt zosta{\l} sfinansowany ze \'srodk\'ow Narodowego Centrum Nauki przyznanych na podstawie decyzji nr: DEC-2011/01/B/ST1/03965.



\section*{References}

\bibliographystyle{elsarticle-num}
\bibliography{bibliography-BCMU-Numerix}

\begin{thebibliography}{10}
\expandafter\ifx\csname url\endcsname\relax
  \def\url#1{\texttt{#1}}\fi
\expandafter\ifx\csname urlprefix\endcsname\relax\def\urlprefix{URL }\fi
\expandafter\ifx\csname href\endcsname\relax
  \def\href#1#2{#2} \def\path#1{#1}\fi

\bibitem{BorisCarlottaMax-M3AS}
B.~Andreianov, C.~Donadello, M.~D. Rosini, Crowd dynamics and conservation laws
  with nonlocal constraints and capacity drop, Mathematical Models and Methods
  in Applied Sciences 24~(13) (2014) 2685--2722.
\newblock \href {http://dx.doi.org/10.1142/S0218202514500341}
  {\path{doi:10.1142/S0218202514500341}}.

\bibitem{LWR1}
M.~J. {Lighthill}, G.~B. {Whitham}, {On Kinematic Waves. II. A Theory of
  Traffic Flow on Long Crowded Roads}, Royal Society of London Proceedings
  Series A 229 (1955) 317--345.
\newblock \href {http://dx.doi.org/10.1098/rspa.1955.0089}
  {\path{doi:10.1098/rspa.1955.0089}}.

\bibitem{LWR2}
P.~I. Richards, Shock waves on the highway, Operations Research 4~(1) (1956)
  42--51.
\newblock \href {http://dx.doi.org/10.1287/opre.4.1.42}
  {\path{doi:10.1287/opre.4.1.42}}.

\bibitem{Greenshields_1934}
B.~D. Greenshields, A study of traffic capacity, Proceeding of the Highway
  Research Board 14 (1934) 448--477.

\bibitem{Schadschneider2008Evacuation}
A.~Schadschneider, W.~Klingsch, H.~Kl{\"u}pfel, T.~Kretz, C.~Rogsch,
  A.~Seyfried, {Evacuation Dynamics: Empirical Results, Modeling and
  Applications}, in: R.~A. Meyers (Ed.), Extreme Environmental Events, Springer
  New York, 2011, pp. 517--550.
\newblock \href {http://dx.doi.org/10.1007/978-1-4419-7695-6_29}
  {\path{doi:10.1007/978-1-4419-7695-6_29}}.

\bibitem{Cepolina2009532}
E.~M. Cepolina, Phased evacuation: An optimisation model which takes into
  account the capacity drop phenomenon in pedestrian flows, Fire Safety Journal
  44~(4) (2009) 532--544.
\newblock \href {http://dx.doi.org/10.1016/j.firesaf.2008.11.002}
  {\path{doi:10.1016/j.firesaf.2008.11.002}}.

\bibitem{Hoogendoorn01052005}
S.~P. Hoogendoorn, W.~Daamen, Pedestrian behavior at bottlenecks,
  Transportation Science 39~(2) (2005) 147--159.
\newblock \href {http://dx.doi.org/10.1287/trsc.1040.0102}
  {\path{doi:10.1287/trsc.1040.0102}}.

\bibitem{Kopylow}
V.~Kopylow, {The study of people' motion parameters under forced egress
  situations}, {Ph.D.~Thesis}, Moscow Civil Engineering Institute (1974).

\bibitem{Schreckenberg2006}
T.~Kretz, A.~Gr{\"u}nebohm, M.~Kaufman, F.~Mazur, M.~Schreckenberg,
  Experimental study of pedestrian counterflow in a corridor, Journal of
  Statistical Mechanics: Theory and Experiment 2006~(10) (2006) P10001.
\newblock \href {http://dx.doi.org/10.1088/1742-5468/2006/10/P10001}
  {\path{doi:10.1088/1742-5468/2006/10/P10001}}.

\bibitem{Seyfried:59568}
A.~Seyfried, T.~Rupprecht, A.~Winkens, O.~Passon, B.~Steffen, W.~Klingsch,
  M.~Boltes, \href{http://juser.fz-juelich.de/record/59568}{{Capacity
  Estimation for Emergency Exits and Bottlenecks}}, in: Interflam 2007, 2007,
  pp. 247--258, record converted from VDB: 12.11.2012.
\newline\urlprefix\url{http://juser.fz-juelich.de/record/59568}

\bibitem{Zhang20132781}
X.~L. Zhang, W.~G. Weng, H.~Y. Yuan, J.~G. Chen, Empirical study of a
  unidirectional dense crowd during a real mass event, Physica A: Statistical
  Mechanics and its Applications 392~(12) (2013) 2781--2791.

\bibitem{Helbing_disaster}
D.~Helbing, A.~Johansson, H.~Z. Al-Abideen, {Dynamics of crowd disasters: An
  empirical study}, Phys. Rev. E 75 (2007) 046109.
\newblock \href {http://dx.doi.org/10.1103/PhysRevE.75.046109}
  {\path{doi:10.1103/PhysRevE.75.046109}}.

\bibitem{Helbing2000Simulating}
D.~Helbing, I.~Farkas, T.~Vicsek, Simulating dynamical features of escape
  panic, Nature 407~(6803) (2000) 487--490.
\newblock \href {http://dx.doi.org/doi:10.1038/35035023}
  {\path{doi:doi:10.1038/35035023}}.

\bibitem{Soria20121584}
S.~A. Soria, R.~Josens, D.~R. Parisi, {Experimental evidence of the ``Faster is
  Slower'' effect in the evacuation of ants}, Safety Science 50~(7) (2012)
  1584--1588.
\newblock \href {http://dx.doi.org/10.1016/j.ssci.2012.03.010}
  {\path{doi:10.1016/j.ssci.2012.03.010}}.

\bibitem{hughes2003flow}
R.~L. Hughes, The flow of human crowds, Annual review of fluid mechanics 35~(1)
  (2003) 169--182.
\newblock \href {http://dx.doi.org/10.1146/annurev.fluid.35.101101.161136}
  {\path{doi:10.1146/annurev.fluid.35.101101.161136}}.

\bibitem{Rosinibook}
M.~D. Rosini, \href{http://dx.doi.org/10.1007/978-3-319-00155-5}{Macroscopic
  models for vehicular flows and crowd dynamics: theory and applications},
  Understanding Complex Systems, Springer, Heidelberg, 2013.
\newblock \href {http://dx.doi.org/10.1007/978-3-319-00155-5}
  {\path{doi:10.1007/978-3-319-00155-5}}.
\newline\urlprefix\url{http://dx.doi.org/10.1007/978-3-319-00155-5}

\bibitem{scontrainte}
B.~Andreianov, P.~Goatin, N.~Seguin, {Finite volume schemes for locally
  constrained conservation laws}, Numerische Mathematik 115 (2010) 609--645.
\newblock \href {http://dx.doi.org/10.1007/s00211-009-0286-7}
  {\path{doi:10.1007/s00211-009-0286-7}}.

\bibitem{AndreianovDonadelloRosiniRazafisonProc}
B.~Andreianov, C.~Donadello, U.~Razafison, M.~D. Rosini, Riemann problems with
  non--local point constraints and capacity drop, Mathematical Biosciences and
  Engineering 12~(2) (2015) 259--278.
\newblock \href {http://dx.doi.org/10.3934/mbe.2015.12.259}
  {\path{doi:10.3934/mbe.2015.12.259}}.

\bibitem{ColomboGoatinConstraint}
R.~M. Colombo, P.~Goatin, \href{http://dx.doi.org/10.1016/j.jde.2006.10.014}{A
  well posed conservation law with a variable unilateral constraint}, J.
  Differential Equations 234~(2) (2007) 654--675.
\newblock \href {http://dx.doi.org/10.1016/j.jde.2006.10.014}
  {\path{doi:10.1016/j.jde.2006.10.014}}.
\newline\urlprefix\url{http://dx.doi.org/10.1016/j.jde.2006.10.014}

\bibitem{chalonsgoatinseguin}
C.~Chalons, P.~Goatin, N.~Seguin,
  \href{http://dx.doi.org/10.3934/nhm.2013.8.433}{General constrained
  conservation laws. {A}pplication to pedestrian flow modeling}, Netw. Heterog.
  Media 8~(2) (2013) 433--463.
\newblock \href {http://dx.doi.org/10.3934/nhm.2013.8.433}
  {\path{doi:10.3934/nhm.2013.8.433}}.
\newline\urlprefix\url{http://dx.doi.org/10.3934/nhm.2013.8.433}

\bibitem{Kruzkov}
S.~Kru{\v{z}}hkov, First order quasilinear equations with several independent
  variables., Mat. Sb. (N.S.) 81 (123) (1970) 228--255.

\bibitem{GodlewskiRaviartBook}
E.~Godlewski, P.-A. Raviart,
  \href{http://dx.doi.org/10.1007/978-1-4612-0713-9}{Numerical approximation of
  hyperbolic systems of conservation laws}, Vol. 118 of Applied Mathematical
  Sciences, Springer-Verlag, New York, 1996.
\newblock \href {http://dx.doi.org/10.1007/978-1-4612-0713-9}
  {\path{doi:10.1007/978-1-4612-0713-9}}.
\newline\urlprefix\url{http://dx.doi.org/10.1007/978-1-4612-0713-9}

\bibitem{LevequeBook}
R.~J. LeVeque, Finite volume methods for hyperbolic problems, Cambridge Texts
  in Applied Mathematics, Cambridge University Press, Cambridge, 2002.
\newblock \href {http://dx.doi.org/10.1017/CBO9780511791253}
  {\path{doi:10.1017/CBO9780511791253}}.

\bibitem{Parisi2005606}
D.~Parisi, C.~Dorso, Microscopic dynamics of pedestrian evacuation, Physica A:
  Statistical Mechanics and its Applications 354~(0) (2005) 606 -- 618.
\newblock \href
  {http://dx.doi.org/http://dx.doi.org/10.1016/j.physa.2005.02.040}
  {\path{doi:http://dx.doi.org/10.1016/j.physa.2005.02.040}}.

\bibitem{Aw_SIAP_2000}
A.~Aw, M.~Rascle, Resurrection of ``second order'' models of traffic flow, SIAM
  J. Appl. Math. 60~(3) (2000) 916--938 (electronic).
\newblock \href {http://dx.doi.org/10.1137/S0036139997332099}
  {\path{doi:10.1137/S0036139997332099}}.

\bibitem{Zhang_Method_2002}
H.~Zhang, A non-equilibrium traffic model devoid of gas-like behavior,
  Transportation Research Part B: Methodological 36~(3) (2002) 275 -- 290.
\newblock \href
  {http://dx.doi.org/http://dx.doi.org/10.1016/S0191-2615(00)00050-3}
  {\path{doi:http://dx.doi.org/10.1016/S0191-2615(00)00050-3}}.

\bibitem{ColomboFacchiMaterniniRosini}
R.~M. Colombo, G.~Facchi, G.and~Maternini, M.~D. Rosini, On the continuum
  modeling of crowds, in: Hyperbolic problems: theory, numerics and
  applications, Vol.~67 of Proc. Sympos. Appl. Math., Amer. Math. Soc.,
  Providence, RI, 2009, pp. 517--526.
\newblock \href {http://dx.doi.org/10.1090/psapm/067.2/2605247}
  {\path{doi:10.1090/psapm/067.2/2605247}}.

\bibitem{CGRESAIM}
R.~Colombo, P.~Goatin, M.~Rosini, {On the modelling and management of traffic},
  ESAIM: Mathematical Modelling and Numerical Analysis 45~(05) (2011) 853--872.
\newblock \href {http://dx.doi.org/10.1051/m2an/2010105}
  {\path{doi:10.1051/m2an/2010105}}.

\end{thebibliography}





\end{document}